\documentclass[a4paper, 12pt]{article}

\usepackage[top = 2.5cm, bottom = 2.5cm, left = 2.4cm, right = 2.4cm]{geometry}
\usepackage{amsfonts, amsmath, amsthm, amssymb, mathtools, cases, bbm, array}
\usepackage{lmodern, indentfirst, setspace, changepage, fancyhdr}
\usepackage{tikz, pgfplots, float, subcaption}
\usepackage{thmtools, thm-restate}
\usepackage[title]{appendix}
\usepackage[english]{babel}
\usepackage{cite}

\usepackage{hyperref}
\hypersetup{
	colorlinks = true,
	linkcolor = blue,
	citecolor = blue,
}

\newtheorem{proposition}{Proposition}[section]
\newtheorem{corollary}[proposition]{Corollary}
\newtheorem{theorem}[proposition]{Theorem}
\newtheorem{lemma}[proposition]{Lemma}
\theoremstyle{definition}
\newtheorem{assumption}[proposition]{Assumption}

\newtheorem{definition}[proposition]{Definition}

\newtheorem{remark}[proposition]{Remark}

\newcommand{\norm}[1]{\left|\left|#1\right|\right|}
\newcommand{\floor}[1]{\left\lfloor#1\right\rfloor}
\newcommand{\ceil}[1]{\left\lceil#1\right\rceil}

\newcommand{\angled}[2]{\left\langle#1,#2\right\rangle}

\newcommand{\defeq}{\vcentcolon=}
\newcommand{\eqdef}{=\vcentcolon}
\newcommand{\scdot}{{}\cdot{}}
\newcommand{\e}{\mathrm{e}}

\newcommand{\ind}[1]{\mathbbm{1}_{\left\{#1\right\}}}
\newcommand{\map}[3]{#1 : #2 \longrightarrow #3}
\newcommand{\asetc}[2]{\left\{#1 : #2\right\}}
\newcommand{\asetr}[2]{\left(#1 : #2\right)}
\newcommand{\setc}[2]{\{#1 : #2\}}
\newcommand{\supp}{\mathrm{supp}}

\newcommand{\bgamma}{\boldsymbol{\gamma}}
\newcommand{\bzeta}{\boldsymbol{\zeta}}

\newcommand{\bmu}{\boldsymbol{\mu}}
\newcommand{\bnu}{\boldsymbol{\nu}}
\newcommand{\bxi}{\boldsymbol{\xi}}

\newcommand{\br}{\boldsymbol{r}}

\newcommand{\bu}{\boldsymbol{u}}
\newcommand{\bv}{\boldsymbol{v}}
\newcommand{\bw}{\boldsymbol{w}}
\newcommand{\bx}{\boldsymbol{x}}
\newcommand{\by}{\boldsymbol{y}}

\newcommand{\bA}{\boldsymbol{A}}

\newcommand{\bD}{\boldsymbol{D}}

\newcommand{\bF}{\boldsymbol{F}}

\newcommand{\bX}{\boldsymbol{X}}
\newcommand{\bY}{\boldsymbol{Y}}
\newcommand{\bZ}{\boldsymbol{Z}}

\newcommand{\calC}{\mathcal{C}}

\newcommand{\calF}{\mathcal{F}}
\newcommand{\calG}{\mathcal{G}}
\newcommand{\calH}{\mathcal{H}}

\newcommand{\calL}{\mathcal{L}}
\newcommand{\calM}{\mathcal{M}}

\newcommand{\calR}{\mathcal{R}}
\newcommand{\calS}{\mathcal{S}}
\newcommand{\calT}{\mathcal{T}}

\newcommand{\indc}{\mathbbm{1}}

\newcommand{\N}{\mathbbm{N}}
\newcommand{\Q}{\mathbbm{Q}}
\newcommand{\R}{\mathbbm{R}}
\newcommand{\Z}{\mathbbm{Z}}

\newcommand{\expect}{E\expectarg}

\DeclarePairedDelimiterX{\expectarg}[1]{[}{]}{%
	\ifnum\currentgrouptype=16 \else\begingroup\fi
	\activatebar#1
	\ifnum\currentgrouptype=16 \else\endgroup\fi
}

\newcommand{\cprob}{P\probarg}

\DeclarePairedDelimiterX{\probarg}[1]{(}{)}{%
	\ifnum\currentgrouptype=16 \else\begingroup\fi
	\activatebar#1
	\ifnum\currentgrouptype=16 \else\endgroup\fi
}
\newcommand{\innermid}{\nonscript\;\delimsize\vert\nonscript\;}
\newcommand{\activatebar}{%
	\begingroup\lccode`\~=`\|
	\lowercase{\endgroup\let~}\innermid 
	\mathcode`|=\string"8000
}

\usetikzlibrary{automata, arrows, positioning, calc, external, babel, backgrounds, matrix, shapes}
\usepgfplotslibrary{fillbetween}
\pgfplotsset{
	compat = 1.16,
	ticklabel style = {font = \footnotesize},
	every axis/.append style = {
		grid style = {dashed, gray, opacity = 0.2},
		label style = {font = \footnotesize}, 
		width = \columnwidth,
		height = 0.618 * 1 * \columnwidth
	}
}

\definecolor{britishracinggreen}{rgb}{0.0, 0.26, 0.15}
\definecolor{bostonuniversityred}{rgb}{0.8, 0.0, 0.0}
\definecolor{ceruleanblue}{rgb}{0.16, 0.32, 0.75}
\definecolor{airforceblue}{rgb}{0.36, 0.54, 0.66}
\definecolor{cadmiumgreen}{rgb}{0.0, 0.42, 0.24}
\definecolor{ao(english)}{rgb}{0.0, 0.5, 0.0}
\definecolor{coolblack}{rgb}{0.0, 0.18, 0.39}
\definecolor{byzantine}{rgb}{0.74, 0.2, 0.64}
\definecolor{alizarin}{rgb}{0.82, 0.1, 0.26}
\definecolor{arsenic}{rgb}{0.23, 0.27, 0.29}
\definecolor{cobalt}{rgb}{0.0, 0.28, 0.67}
\definecolor{amber}{rgb}{1.0, 0.75, 0.0}

\title{Measure-valued fluid limits for partial service queues\\with correlated patience and service times\vspace{\baselineskip}}

\author{
	\begin{tabular}{ccccc}
		\normalsize Diego Goldsztajn & \hspace{5mm} & \normalsize Fernando Paganini & \hspace{5.mm} & \normalsize Andres Ferragut \\ 
		\scriptsize Universidad ORT Uruguay & \hspace{5mm} & \scriptsize Universidad ORT Uruguay & \hspace{5mm} & \scriptsize Universidad ORT Uruguay \\
		\scriptsize\texttt{goldsztajn@ort.edu.uy} & \hspace{5mm} & \scriptsize\texttt{paganini@ort.edu.uy} & \hspace{5mm} & \scriptsize\texttt{ferragut@ort.edu.uy} \\
	\end{tabular}
}

\date{\vspace{\baselineskip} August 24, 2026}

\begin{document}

	
\maketitle

\noindent\rule{\textwidth}{1pt}

\vspace{2\baselineskip}

\onehalfspacing

\begin{adjustwidth}{0.8cm}{0.8cm}
	\begin{center}
		\textbf{Abstract}
	\end{center}
	
	\vspace{0.3\baselineskip}
	
	\noindent
	We consider a many-server queue where tasks with general patience and service times arrive as a renewal process. Contrary to the standard assumptions, we allow for preemption and abandonment during service, relevant in applications such as electric vehicle charging and anytime algorithms in cloud computing. Moreover, we do not assume independence of patience and service times, which has been a critical assumption in the literature. We describe the state of the system using a discrete measure on a two-dimensional orthant, such that the coordinates of its atoms represent the attained service and time in the system of tasks. Under mild assumptions, we derive the fluid limit as the number of servers approaches infinity and the arrival rate of tasks grows proportionally. The limit is given by a measure-valued integral transport equation that we solve explicitly when the initial condition is the null measure. We also prove that all solutions, regardless of the initial condition, converge to the same fixed point, which admits a closed-form expression. Our results focus on the preemptive Last-Come-First-Served (LCFS) policy, which has been identified as practically appealing in recent work.
	
	\vspace{\baselineskip}
	
	
	\vspace{0.3\baselineskip}
	
	\small{\noindent The research presented in this paper was supported by ANII-Uruguay under fellowship PD\_NAC\_2024\_1\_182118 and AFOSR under grant FA9550-23-1-0350.} 
\end{adjustwidth}

\newpage


\section{Introduction}
\label{sec: introduction}

The present paper considers many-server queues where tasks with general and possibly correlated patience and service times arrive as renewal processes. A standard assumption in the study of similar models has been that tasks do not abandon after their service has started. While this is natural in call centers and some computer systems, there exist several modern applications where tasks may abandon during service and still benefit from the partial service received. A clear example is parking lots for charging electric vehicles, which behave as many-server queues because the power supplied by the grid limits the maximum number of vehicles that can be charged simultaneously. In these systems, vehicles typically abandon with a partially charged battery that makes the vehicle functional and can be completed elsewhere. Other applications include the execution of anytime tasks in the cloud and the updating of mobile IoT devices at wireless access points.

We model these \emph{partial service queues} regarding tasks as particles that move over the two-dimensional orthant. The coordinates of a particle represent the attained service time and time spent in the system so far. The former increases at unit rate only while the task is being served, whereas the latter always increases at unit rate. Arriving tasks appear at the origin since they have not received any service nor spent any time in the system, and tasks disappear due to service completion or abandonment. The former situation occurs if the attained service reaches the service time required by the task, whereas the latter situation happens if the time in the system reaches the patience. The overall state of the system is represented by a discrete measure with a unit-mass atom at each particle.

We consider the many-server limiting regime where the number of servers grows to infinity and the arrival rate of tasks grows proportionally; this scaling regime offers far greater tractability than the situation where the number of servers is fixed. We further normalize the state of the system by the number of servers, and prove a measure-valued fluid limit: the sequence of processes describing partial service queues is tight, and the limit of every convergent subsequence solves a two-dimensional integral transport equation with discontinuous feedback. For the null initial condition, corresponding to an initially empty system, we solve the equation explicitly, and in general we obtain a structural result. We show that solutions can be expressed as the superposition of two measure-valued functions: one represents tasks that were present at time zero, whereas the other is the solution with null initial condition, and represents tasks arriving after time zero. We then prove that all solutions converge to a single fixed point that admits a closed-form expression.

\subsection{Background and motivation}
\label{sub: background and motivation}

Measure-valued limits originate in the study of weakly interacting particle systems. The limiting dynamics were introduced by McKean in \cite{mckean1966class,mckean1967propagation} by proving propagation of chaos. A law of large numbers for the empirical measure process itself, together with a martingale and tightness framework close to that adopted here, were established by Oelschläger in the foundational paper \cite{oelschlager1984martingale}. That literature deals largely with closed systems of exchangeable particles with diffusive dynamics, whereas queueing systems are open, with arrivals and departures, and typically exhibit discontinuous dynamics. Measure-valued state descriptors were brought into queueing theory by Doytchinov, Lehoczky and Shreve in \cite{doytchinov2001real}, and Gromoll, Puha and Williams derived the first measure-valued fluid limit  in \cite{gromoll2002fluid}, both for single-server queues. In the more pertinent context of many-server queues, Kaspi and Ramanan obtained the fluid limit of the First-Come-First-Served (FCFS) policy without abandonment in \cite{kaspi2011law}, and Kang and Ramanan extended it to the model with abandonment in \cite{kang2010fluid}.

The work of Kaspi, Kang and Ramanan laid the foundations for proving measure-valued fluid limits of many-server queues, and has been extended in multiple ways. In particular, Zhang treated the residual service and patience times in \cite{zhang2013fluid}, instead of the attained counterparts, and Walsh Zu\~niga considered the attained descriptor in the situation where the service requirement and patience distributions are not absolutely continuous, in \cite{zuniga2014fluid}. Atar, Kaspi and Shimkin proved the fluid limit of multi-class many-server queues with abandonment in \cite{atar2014fluid}. Also, Atar, Biswas and Kaspi used a measure-valued Skorohod map to derive the fluid limit of the Earliest-Deadline-First (EDF) policy in \cite{atar2018law}. More recently, Puha and Ward obtained fluid limits for multi-class many-server queues and a broad class of priority policies in \cite{puha2022fluid}, and Aveklouris, Puha and Ward extended the results of Kaspi, Kang and Ramanan to bipartite queueing systems, called matching queues, in \cite{aveklouris2024fluid}.

In a classic many-server queue, a task that reaches a server is served uninterruptedly until completion, as preemption and abandonment during service are not allowed in any of the papers mentioned above. In contrast, the partial service queues that we consider here allow for both, a seemingly subtle difference with fundamental consequences for describing how the system evolves over time. In the classic model, time in the system is irrelevant for tasks in service, as they do not abandon, and the attained service of tasks in the queue is zero since preemption is not possible. As a result, the state of the system can be described using two one-dimensional discrete measures in tandem: one for the waiting times of tasks in the queue and another one for the attained services of tasks being served. Instead, in the partial service queues that we consider, both the time in system and attained service are relevant during service, and tasks in the queue may have received service since preemption is possible. So in order to describe how a partial service queue evolves over time, we must simultaneously keep track of both the attained service and time in system of all tasks.

Including this fundamental difference in the model, our paper extends the theory of measure-valued fluid limits for many-server queues in the following ways.
\begin{enumerate}
	\item \emph{Abandonment and preemption.} The literature considers nonpreemptive policies and assumes that tasks are served uninterruptedly until completion once they start being served. Instead, we allow for preemption and abandonment during service, which is important for modeling several modern applications.
	
	\item \emph{State-descriptor.} As noted earlier, the literature has relied on two one-dimensional discrete measures to describe the state of the system. One represents the \emph{potential} waiting times of tasks; the term potential is used since tasks may have started service or even left. The other one describes the attained service of tasks currently in service, and the model does not keep track of which potential waiting time and attained service correspond to the same task. In contrast, we use a two-dimensional measure to keep track of the attained service and time in system simultaneously.
	
	\item \emph{Independence assumption.} The literature assumes that the service requirement and patience of each task are independent, which implies that the potential waiting times process can be analyzed separately from the attained services process. This crucial property substantially simplifies the analysis and is the reason for using potential waiting times, instead of effective waiting times, and for assuming independence. Our two-dimensional state-descriptor allows for correlated service requirements and patience times. Analyzing many-server queues where such a correlation exists has been identified as an open problem by Moyal and Perry in \cite{moyal2022many}.
	
	\item \emph{Limiting transport equation.} A key element in the fluid limits of many-server queues has been the one-dimensional transport equation obtained by Kaspi and Ramanan in \cite{kaspi2011law}. The fluid limits in \cite{atar2014fluid,atar2018law,kang2010fluid,puha2022fluid,aveklouris2024fluid,kaspi2011law} are systems of equations involving one or more transport equations of this kind which are complexly coupled. In contrast, our fluid limit is a novel two-dimensional transport equation. Unlike the transport equation obtained by Kaspi and Ramanan, the term associated with departures in our equation depends on a discontinuous feedback that reflects the service policy. In addition, our departures term involves a two-dimensional hazard rate vector field instead of the one-dimensional hazard rate function that appears in \cite{kaspi2011law}.
\end{enumerate}

Gromoll, Robert and Zwart also used a two-dimensional state descriptor and considered correlated patience and service times in \cite{gromoll2008fluid}, where they obtained the measure-valued fluid limit of a Processor-Sharing (PS) system with abandonment. However, their state descriptor is based on residual patience and service times, and thus behaves significantly differently from the state descriptors considered in this paper and by Kang, Kaspi and Ramanan in \cite{kaspi2011law, kang2010fluid}. Unlike our state descriptor, the one based on residual times is typically impossible to observe in practice and cannot be used to make decisions for controlling the system, as it requires knowing the patience times and service requirements in advance. Here these quantities are assumed unknown and not included in the filtration that represents the history of the system. Therefore, the analysis of departures involves martingale techniques and a hazard rate vector field, whereas none of these elements are present in \cite{gromoll2008fluid}.

While parts of our analysis are fairly independent of the service policy, other parts are specific to the preemptive Last-Come-First-Served (LCFS) policy that we consider. Particularly, the characterization of subsequential limits and a large part of the analysis of the limiting transport equation. Our motivation for considering the preemptive LCFS policy comes from a remarkable property observed by Ferragut, Goldsztajn and Paganini in \cite{ferragut2025last}. Specifically, the performance of the preemptive EDF and LCFS policies appears to be the same in the many-server regime, in the sense that the distribution of the attained service of tasks at the time of departure is the same. The former policy is popular for providing need-based fairness, but has the disadvantage of requiring deadline information, which is often unavailable or difficult to estimate. In contrast, preemptive LCFS has the important advantage of not requiring any prior information about the tasks.

The property observed in \cite{ferragut2025last} was obtained by analyzing fluid models for stationary systems, which were postulated without proving a fluid limit. For preemptive LCFS, the fluid model of \cite{ferragut2025last} is a two-dimensional transport equation similar to the one obtained here, except that the equation in \cite{ferragut2025last} does not involve the time variable since it describes a system in equilibrium. The present paper establishes a rigorous connection between the preemptive LCFS partial service queue and a time-dependent transport equation, proving a measure-valued fluid limit. Moreover, we show that all solutions of this transport equation converge to a single fixed point which solves the stationary transport equation of \cite{ferragut2025last}.





\subsection{Organization of the paper}

The rest of Section \ref{sec: introduction} introduces standard notation and terminology. Section \ref{sec: model description} defines stochastic processes that model partial service queues, introduces the limiting regime that we consider and states our assumptions. Section \ref{sec: main results} states our main results, including the measure-valued fluid limit and several results pertaining to the limiting transport equation. The fluid limit is proved in Section \ref{sec: proof of the fluid limit} and the limiting transport equation is analyzed in Section \ref{sec: analysis of the fluid problem}. Some auxiliary results are proved in Appendices \ref{app: construction of sample paths}, \ref{app: measurability properties} and \ref{app: proofs of auxiliary results}.

\subsection{Notation and terminology}
\label{sub: notation and terminology}

We use the symbols $P$ and $E$ to denote the probability of events and the expectation of functions, respectively. The underlying probability space to which these symbols refer will always be clear from the context or explicitly indicated. The natural numbers $\N$ include zero, the set of strictly positive integers is denoted by $\Z_+$ and $\R_+ \defeq [0, \infty)$. We let $x \wedge y \defeq \min\{x, y\}$ and $x^+ \defeq \max\{x, 0\}$ for all $x, y \in \R$.

\subsubsection{Function spaces}

Let $A \subset \R^d$ and $\map{f}{A}{\R}$. If they exist, then $\nabla f$ and $\partial_i f$ denote the gradient and the partial derivative with respect to the $i$th variable, respectively. Let $e_i$ denote the $i$th vector in the canonical basis of $\R^d$. Then we let
\begin{equation*}
	\partial_i^- f(x) \defeq \lim_{\delta \downarrow 0} \frac{f(x - \delta e_i) - f(x)}{\delta} \quad \text{and} \quad \partial_i^+ f(x) \defeq \lim_{\delta \downarrow 0} \frac{f(x + \delta e_i) - f(x)}{\delta}
\end{equation*}
denote the lateral partial derivatives of $f$ when the limits are well-defined and exist. In this paper, the set $A$ will typically be the nonnegative orthant $\R_+^2$ or $\R_+^3$.

The space of continuous functions on $A$ is denoted by $C(A)$. If $A$ is open in $\R^d$, then $C^k(A)$ denotes the space of $k$-times continuously differentiable functions. We often consider the space $C^1(\R_+^d)$ defined in the standard way. Namely, it consists of the functions $f$ such that $\partial_i f$ and $\partial_i^+ f$ exist and are continuous in $\setc{x \in \R_+^d}{x_i > 0}$ and $\R_+^d$, respectively. We use the subscripts $b$ and $c$ to indicate that the functions in the space under consideration are bounded and compactly supported, respectively. The support of $f$ is the closure in $A$ of the set $\setc{x \in A}{f(x) \neq 0}$ with respect to the relative topology inherited from $\R^d$, and we denote the support of $f$ by $\supp(f)$. We endow $C_b(A)$, $C_c(A)$, $C_b^k(A)$ and $C_c^k(A)$ with the uniform topology, induced by the standard supremum norm:
\begin{equation*}
	\norm{f}_\infty \defeq \sup_{x \in A} |f(x)| \quad \text{for all} \quad f \in C_b(A).
\end{equation*}

The total variation over $[0, x]$ of a function $\map{f}{\R_+}{\R}$ is defined as
\begin{equation*}
	TV[f](x) \defeq \sup \asetc{\sum_{i = 1}^n \left|f(x_i) - f(x_{i - 1})\right|}{n \in \N\ \text{and}\ 0 = x_0 < x_1 < \dots < x_n = x}.
\end{equation*}
Given a set $A \subset \R^d$ and a function $\map{f}{A}{\R}$, we write $f \in L_{loc}^1(A)$ if the function $f$ is locally integrable with respect to the Lebesgue measure.

\subsubsection{Measure spaces}

We consider possibly signed measures defined on the Borel $\sigma$-algebra of $\R_+^d$. The space of such measures that are Radon is denoted by $\calM(\R_+^d)$, and the subset of nonnegative measures is denoted by $\calM^+(\R_+^d)$; their subsets of finite measures are denoted by adding the subscript $F$. Since the underlying space is the Euclidean space $\R_+^d$, the Radon measures are the Borel measures $\mu$ with locally finite total variation $|\mu|$. We let $\delta_x$ denote the Dirac delta at the point $x \in \R_+^d$, and if $\mu \in \calM(\R_+^d)$ and $\map{f}{\R_+^d}{\R}$ is integrable, then
\begin{equation*}
	\angled{f}{\mu} \defeq \int_{\R_+^d} f(x)\mu(dx).
\end{equation*}

We endow $\calM^+(\R_+^d)$ with the vague topology. Specifically, a sequence of measures $\mu^n$ converges vaguely to the measure $\mu$ as $n \to \infty$ if and only if we have
\begin{equation*}
	\lim_{n \to \infty} \angled{f}{\mu^n} = \angled{f}{\mu} \quad \text{for all} \quad f \in C_c\left(\R_+^d\right).
\end{equation*}
On the other hand, we equip $\calM_F^+(\R_+^d)$ with the weak topology: $\mu^n \to \mu$ weakly as $n \to \infty$ if and only if the above limit holds for all $f \in C_b(\R_+^d)$. By \cite[Theorem 8.9.4]{bogachev2007measure}, this space is Polish, i.e., separable with a complete metric, the Bounded-Lipschitz metric defined as:
\begin{equation*}
	d_{BL}(\mu, \nu) \defeq \sup_{f \in \calL} \left|\angled{f}{\mu} - \angled{f}{\nu}\right|,
\end{equation*}
where $\calL$ is the set of all $\map{f}{\R_+^d}{[-1, 1]}$ with unit Lipschitz constant.

\subsubsection{Stochastic processes}

For stochastic processes, and other deterministic functions that depend on time, we typically use boldface notation and write the time variable as a subscript.

Let $\bX$ be a real-valued process with a given stochastic basis. If the total variation process $TV[\bX]$ is locally integrable, then the compensator of $\bX$ is the predictable process $\bX^c$ such that the process $\bX - \bX^c$ is a local martingale; see \cite[Theorem I.3.18]{jacod2013limit} for details. Given two real-valued locally square-integrable martingales $\bX$ and $\bY$, their predictable quadratic covariation is the predictable process $\angled{\bX}{\bY}$ such that $\bX\bY - \angled{\bX}{\bY}$ is a local martingale, as in \cite[Theorem I.4.2]{jacod2013limit}. The quadratic covariation of two semimartingales $\bX$ and $\bY$ is denoted by $[\bX, \bY]$ and defined as in \cite[Section I.4e]{jacod2013limit}.

If $S$ is a metric space and $\map{f}{[0, \infty)}{S}$ is a function, then
\begin{equation*}
	f_{t^-} \defeq \lim_{s \uparrow t} f_s \quad \text{and} \quad f_{t^+} \defeq \lim_{s \downarrow t} f_s
\end{equation*}
are the left-limit of $f$ at $t > 0$ and the right-limit of $f$ at $t \geq 0$, respectively. The function is c\`adl\`ag if the left-limits exist for all $t > 0$ and the right-limits exist and $f_{t^+} = f_t$ for all $t \geq 0$. The space of c\`adl\`ag functions is denoted by $D_S[0, \infty)$ and is endowed with the standard Skorohod-$J_1$ topology, which makes it Polish when the space $S$ is Polish. The sample paths of the processes that we consider are c\`adl\`ag functions with values in the Polish space $S = \R$ or $S = \calM_F^+(\R_+^d)$. We write $\bX^n \Rightarrow \bX$ to indicate that the processes $\bX^n$ converge weakly to $\bX$ as $n \to \infty$ in the Skorohod-$J_1$ topology. Sometimes we use the topology of uniform convergence over compact sets, but this is explicitly indicated.

\section{Model description}
\label{sec: model description}

Consider a service system with $M \in \N$ tasks present at time zero and new tasks arriving as a renewal process $\calR$. The tasks already present in the system at time zero are indexed by the nonpositive integers $i \in \{1 - M, \dots, 0\}$, whereas the new tasks are indexed by the positive integers $i \in \Z_+$. The arrival time of task $i$ is denoted by $a_i$ for all $i \geq 1$ and $a_i \defeq 0$ for all $i \leq 0$. Each task $i$ requires $b_i$ units of service time and cannot stay in the system for more than $c_i$ units of time. In addition, the random vectors $(b_i, c_i)$ are independent of the arrival process $\calR$, mutually independent over $i$ and identically distributed.

\begin{remark}
	\label{rem: codomains of arrival process, service requirements and patiences}
	The process $\calR$ takes values in the space of c\`adl\`ag functions from $\R_+$ to $\N$, has jumps of unit size and is nondecreasing. In particular, the number of arrivals in any finite interval is finite surely. Also, the random vectors $(b_i, c_i)$ take values in the positive orthant $(0, \infty)^2$. Thus, each task stays a positive amount of time in the system surely.
\end{remark}

The amounts of time that task $i$ has spent in service and in the system by time $t$ are denoted by $\bx_i(t)$ and $\by_i(t)$, respectively. The system behaves as a partial service queue, which means that tasks may abandon even while receiving service. More precisely, task $i$ leaves the system due to completion or abandonment. The former situation occurs when the attained service $\bx_i$ reaches the service time $b_i$ required by the task, whereas the latter situation happens when the time in system $\by_i$ reaches the patience $c_i$ of the task. Hence, $d_i \defeq \inf \setc{t \geq a_i}{\bx_i(t) \geq b_i\ \text{or}\ \by_i(t) \geq c_i}$ is the departure time of task $i$.

We regard tasks as particles moving over the two-dimensional orthant $\R_+^2$, with the coordinates $(\bx_i, \by_i)$ of task $i$ given by the current attained service and time in the system. Arriving tasks appear at the origin, since they have not received any service or spent any time in the system, and task $i$ disappears at time $d_i$, upon completion or abandonment. We describe the state of the system at time $t$ through the finite discrete measure
\begin{equation}
	\label{eq: state of the system}
	\bmu_t(dx, dy) \defeq \sum_{i = 1 - M}^{\calR(t)} \ind{t < d_i}\delta_{\left(\bx_i(t), \by_i(t)\right)}(dx, dy) \in \calM_F^+(\R_+^2).
\end{equation}

Naturally, the time in the system of each task $i$ always increases at unit rate, but the evolution of the attained service depends on the service policy. We assume that the system has $N$ servers, each capable of working at unit rate, and we consider the preemptive LCFS policy that always serves the $N$ tasks with the least time in the system. We denote the service rate of task $i$ at time $t$ by $r(\bmu_t, \by_i(t))$, where the function $r$ is defined as:
\begin{equation}
	\label{eq: service rate function}
	r\left(\mu, y\right) \defeq \ind{y \leq y_*(\mu)}, \quad \text{with} \quad y_*(\mu) \defeq \inf\asetc{y \geq 0}{\mu\left(\R_+ \times [0, y]\right) \geq N},
\end{equation}
for all $\mu \in \calM_F^+(\R_+^2)$ and $y \in \R_+$. The threshold $y_*(\bmu_t)$ is infinite if the system has fewer tasks than servers, and otherwise indicates the $N$th smallest time in the system across the present tasks. The preemptive LCFS policy assigns one of the $N$ unit-rate servers to each task $i$ with time in the system $\by_i$ smaller than or equal to the current threshold $y_*(\bmu)$. While we focus on the preemptive LCFS policy, other policies can be modeled by defining the service rate function $r$ in \eqref{eq: service rate function} suitably in terms of $\mu$, $x$ and $y$.

The arrival process $\calR$, the random vectors $(b_i, c_i)$ and the initial state $\bmu_0$ can be defined on a common probability space $(\Omega, \calF, P)$. The other random variables and processes are constructed in terms of these stochastic primitives in Appendix \ref{app: construction of sample paths}, such that:
\begin{equation}
	\label{eq: attained service and sojourn time evolution}
	\bx_i(t) = \bx_i(0) + \int_{a_i \wedge t}^{d_i \wedge t} r\left(\bmu_s, \by_i(s)\right)ds \quad \text{and} \quad \by_i(t) = \by_i(0) + d_i \wedge t - a_i \wedge t.
\end{equation}
We assume that $\bx_i(0) < b_i$ and $\by_i(0) < c_i$ for tasks with $i \leq 0$, which are already present in the system at time zero; otherwise, these tasks would leave the system immediately. For tasks with $i \geq 1$, which arrive after time zero, we let $\bx_i(0) \defeq 0 \eqdef \by_i(0)$ since these tasks have not received any service or spent any time in the system when they arrive. The time in the system of task $i$ is constant before the arrival time, increases at unit rate between the arrival and departure times and remains constant after the departure time. The attained service behaves similarly, except that increases at a possibly time-varying rate between the arrival and departure times of the task; this rate is the service rate $r(\bmu, \by_i)$ of task $i$.

Consider the $\sigma$-algebras defined as
\begin{equation*}
	\tilde{\calF}_t = \sigma\asetr{\bx_i(s), \by_i(s)}{i \geq 1, 0 \leq s \leq t} \quad \text{for all} \quad t \geq 0.
\end{equation*}
Fix the complete and right-continuous filtration $\setc{\calF_t}{t \geq 0}$ induced by $\setc{\tilde \calF_t}{t \geq 0}$. The right-derivatives $\partial_t^+\by_i$ are adapted by the right-continuity of the filtration, and therefore
\begin{align*}
		&a_i = \inf \asetc{t \geq 0}{\partial_t^+\by_i(t) > 0} \quad \text{and} \quad d_i = \inf \asetc{t \geq a_i}{\partial_t^+\by_i(t) = 0}
\end{align*}
are stopping times by \cite[Proposition 2.1.5]{ethier2009markov}. In particular, this implies that $\calR$ and $\bmu$ are adapted. Moreover, because the functions $\bx_i$ and $\by_i$ are continuous, the measure-valued process $\map{\bmu}{[0, \infty)}{\calM_F^+(\R_+^2)}$ is c\`adl\`ag with respect to the weak topology.

\subsection{Limiting regime and standing assumptions}
\label{sub: limiting regime and standing assumptions}

In order to obtain greater tractability, we consider a limiting regime where the number of servers goes to infinity and the arrival rate of tasks grows proportionally. Formally, we fix a sequence of stochastic primitives $(\calR^N, \setc{(b_i^N, c_i^N)}{i \in \Z}, \bmu_0^N)$, where the superscript $N \geq 1$ indicates the number of servers, and construct the corresponding measure-valued processes $\bmu^N$ as in Appendix \ref{app: construction of sample paths}. We assume that the stochastic primitives satisfy the assumptions listed below, which specify the limiting regime, and our goal is to derive and characterize the limit as $N \to \infty$ of the normalized processes $\bar\bmu^N \defeq \bmu^N / N$.

\begin{assumption}
	\label{ass: standing assumptions 1}
	The above sequence of stochastic primitives is defined on a common probability space $(\Omega, \calF, P)$ for all $N \geq 1$, and the following properties hold.
	\begin{enumerate}
		\item[(a)] There exists $\bar M \geq 0$ such that $E[M^N] \leq \bar M N$  for all $N \geq 1$, i.e., the ratio of the initial number of tasks to the number of servers is bounded on average.
		
		\item[(b)] The sequence of random measures $\setc{\bar\bmu_0^N}{N \geq 1}$ is tight in $\calM_F^+(\R_+^2)$.
		
		\item[(c)] There exists a set of probability one $\Theta_0 \in \calF$, such that for all $\omega \in \Theta_0$, there exists a nonnegative function $\theta_0(\omega) \in L_{loc}^1(\R_+)$ satisfying that
		\begin{equation*}
			\limsup_{N \to \infty} \bar\bmu_0^N\left(\omega, \R_+ \times \left[y_1, y_2\right]\right) \leq \int_{y_1}^{y_2} \theta_0(\omega, t)dt \quad \text{for all} \quad y_1, y_2 \in \R_+.
		\end{equation*}
		If the initial conditions $\bar\bmu_0^N$ converge as $N \to \infty$, then this implies that the $y$-marginal of the limit is almost surely an absolutely continuous measure.
		
		\item[(d)] There exists a renewal process $\calR_0$ such that $\calR^N(t) = \calR_0(Nt)$ for all $t \geq 0$, which implies that the arrival rate is proportional to the number of servers.
		
		\item[(e)] The holding time distribution of the renewal process $\calR_0$ has density $\map{f_0}{\R_+}{\R_+}$ and mean $1 / \lambda_0$ for some $\lambda_0 > 0$, which is the normalized arrival rate of tasks.
		
		\item[(f)] For all $i \in \Z$ and $N \geq 1$, we have $(b_i^N, c_i^N) = (b_i^1, c_i^1)$, i.e., these random vectors do not depend on $N$. Moreover, $(b_i^1, c_i^1)$ has density $\map{g}{\R_+^2}{\R_+}$ for all $i \in \Z$.
	\end{enumerate}
\end{assumption}

Assumptions (d) and (e) imply that the arrival rate of tasks scales with $N$. However, the sequence of arriving tasks remains the same by (f), in the sense that the $i$th arriving task has the same service requirement and patience for all indices $N$. Assumption (c) implies that the distribution of the time in the system across the tasks already present at time zero becomes absolutely continuous as $N \to \infty$. If the system is initially empty and we consider a new origin of time at $t = t_0$, then the measures $\bar\bmu_{t_0}^N$ satisfy (c) with the function defined almost surely as $\theta_0 (t) \defeq \lambda_0$ for all $t \geq 0$. Indeed, the times in the system $\by_i^N(t_0) = t_0 - a_i^N$ of tasks present in the system at time $t = t_0$ are a subset of the reflected and shifted arrival times $a_i^N$, which occur at rate $N \lambda_0$. Therefore, it is natural to assume that the $y$-marginals of the initial states $\bar \bmu_0^N$ satisfy assumption (c).

We assume that tasks arrive as a renewal process to simplify the exposition, but the fluid limit can be extended in a straightforward way to the situation where the arrival process is a time-inhomogeneous Poisson process; the analysis of the limiting transport equation, however, relies on the arrival rate of tasks being constant over time. Another assumption that simplifies the exposition is that the distributions of the interarrival time, service requirement and patience have densities, which still covers a large class of distributions of interest. As noted in Section \ref{sec: introduction}, we do not assume independence of the service requirement and patience of tasks, which has been a critical assumption in the literature.

Consider the complementary cumulative distribution function
\begin{equation*}
	\bar F_0(t) \defeq \int_t^\infty f_0(s)ds \quad \text{for all} \quad t \geq 0,
\end{equation*}
and let $\kappa_0 \defeq \sup \setc{t \geq 0}{\bar F_0(t) > 0}$, which may be infinite. Then the holding times of $\calR_0$ are strictly shorter than $\kappa_0$ with probability one. For the system with $N$ servers,
\begin{equation*}
	f^N(t) = Nf_0\left(N t\right) \quad \text{and} \quad \bar F^N(t) = \int_{Nt}^\infty f_0(s)ds
\end{equation*}
are the density and complementary cumulative distribution function of the interarrival time distribution by (d) of Assumption \ref{ass: standing assumptions 1}. In addition, the interarrival times are strictly shorter than $\kappa^N \defeq \kappa_0 / N$ with probability one and the arrival rate is $\lambda^N = N \lambda_0$.

The following definition introduces the survival function and hazard rate vector field associated with $g$. The latter appears in the transport equation arising from the fluid limit, and the former plays an important role in the analysis of this transport equation.

\begin{definition}
	\label{def: hazard rate}
	The survival function associated with $g$ is defined as
	\begin{equation*}
		S_g(x, y) \defeq \int_{[x, \infty) \times [y, \infty)} g(u, v)dudv \quad \text{for all} \quad x, y \in \R_+.
	\end{equation*}
	Consider the open set $U_g \defeq \setc{(x, y) \in \R_+^2}{S_g(x, y) > 0}$. If the survival function $S_g$ is differentiable, then we define the hazard rate vector field $\map{h_g}{\R_+^2}{\R_+^2}$ as:
	\begin{equation*}
		h_g(x, y) \defeq -\nabla \left[\log S_g\right](x, y) \quad \text{if} \quad (x, y) \in U_g \quad \text{and} \quad h_g(x, y) \defeq 0 \quad \text{otherwise}.
	\end{equation*}
	The first and second components of $h_g$ are denoted by $h_g^x$ and $h_g^y$, respectively.
\end{definition}

The hazard rate vector field is a natural two-dimensional generalization of the hazard rate function, introduced in \cite{johnson1975vector}. We impose the following regularity conditions.

\begin{assumption}
	\label{ass: standing assumptions 2}
	The following properties hold.
	\begin{enumerate}
		\item[(a)] The density $g$ is locally bounded in $U_g$, which means that for each $z \in U_g$ there exists a relatively open subset $U$ of $\R_+^2$ such that $z \in U \subset U_g$ and $g$ is bounded inside $U$. Also, $f_0$ is locally bounded in the interval $[0, \kappa_0)$, which is defined analogously.
		
		\item[(b)] The survival function $S_g$ is continuously differentiable. In particular, note that this implies that the hazard rate vector field $h_g$ is continuous in $U_g$.
	\end{enumerate}
\end{assumption}

It is possible that $g$ is discontinuous and unbounded and $h_g$ is unbounded. For example,
\begin{equation*}
	g(x, y) = \frac{1}{2\sqrt{1 - x}}\ind{x < 1, y < 1}, \quad h_g^x(x, y) = \frac{1}{2(1 - x)} \quad \text{and} \quad h_g^y = \frac{1}{1 - y}
\end{equation*}
have these properties; in this example, $U_g = [0, 1) \times [0, 1)$. Assuming that the survival function is continuously differentiable is important for the analysis of the transport equation that arises from the fluid limit. In addition, having a continuous hazard rate vector field simplifies some arguments used to prove the fluid limit.

\section{Statements of main results}
\label{sec: main results}

Consider the following version of \eqref{eq: service rate function} for normalized state descriptors:
\begin{equation*}
	\bar r\left(\mu, y\right) \defeq \ind{y \leq \bar y_*(\mu)} \quad \text{with} \quad \bar y_*(\mu) \defeq \inf\asetc{y \geq 0}{\mu\left(\R_+ \times [0, y]\right) \geq 1}.
\end{equation*}
Using the superscript $N$ to denote the threshold and rate function of the system with $N$ servers, we note that $\bar y_*(\bar \bmu^N) = y_*^N(\bmu^N)$ and $\bar r(\bar \bmu^N) = r^N(\bmu^N)$. Hence, the above definition allows to express the service rate of tasks in terms of the normalized state descriptor $\bar \bmu^N$. We now define the measure-valued functions that may arise as limits of $\bar\bmu^N$.

\begin{definition}
	\label{def: fluid problem}
	The set of solutions of the \emph{fluid problem} is denoted by $\calS^+$. We say that a measure-valued function $\bmu \in D_{\calM_F^+(\R_+^2)}[0, \infty)$ is a solution of the fluid problem, and thus belongs to $\calS^+$, if it satisfies the following conditions.
	\begin{enumerate}
		\item[(a)] We have $\bmu_t(U_g^c) = 0$ for all $t \geq 0$, where we use the notation $U_g^c \defeq \R_+^2 \setminus U_g$.
		
		\item[(b)] The $y$-marginal of $\bmu_t$ is absolutely continuous with respect to the Lebesgue measure for all $t \geq 0$; this is the measure $A \mapsto \bmu_t(\R_+ \times A)$ defined for Borel sets $A \subset \R_+$.
		
		\item[(c)] The following integral is finite for all $t \geq 0$:
		\begin{equation*}
			\int_0^t \angled{\bar r(\bmu_s)h_g^x + h_g^y}{\bmu_s}ds = \int_0^t\int_{U_g} \left[\bar r\left(\bmu_s, y\right)h_g^x(x, y) + h_g^y(x, y)\right]\bmu_s(dx, dy)ds < \infty.
		\end{equation*}
		
		\item[(d)] The following equation holds for all $t \geq 0$ and $\varphi \in C_c^1(\R_+^3)$:
		\begin{equation}
			\label{eq: transport equation}
			\begin{split}
				\angled{\varphi_t}{\bmu_t} &= \angled{\varphi_0}{\bmu_0} + \lambda_0 \int_0^t \varphi_s(0, 0)ds - \int_0^t \angled{\varphi_s\left[\bar r(\bmu_s)h_g^x + h_g^y\right]}{\bmu_s}ds \\
				&+ \int_0^t \angled{\partial_t\varphi_s + \bar r\left(\bmu_s\right)\partial_x\varphi_s + \partial_y\varphi_s}{\bmu_s}ds.
			\end{split}
		\end{equation}
	\end{enumerate}
	We remark that the values of the test function $\varphi$ are denoted by $\varphi_t(x, y)$, with the time variable as a subscript. Also, the integrals with respect to $\bmu_s$ only involve the variables $x$ and $y$ of the integrand and keep the time variable $s$ fixed. We establish in Appendix \ref{app: measurability properties} that the iterated integrals on the right-hand side of \eqref{eq: transport equation} are well-defined. Specifically, we show that the inner integrals are Borel measurable with respect to $s$.
\end{definition}

The measure $\bmu_t$ represents the distribution of the attained service and time in the system for a continuous population of tasks, and \eqref{eq: transport equation} describes how this distribution evolves over time; since $\bmu_t$ is singular for the solutions, as will be shown in Section \ref{sec: analysis of the fluid problem}, we adopt a weak formulation as in \cite{kaspi2011law}. The first term on the right-hand side of \eqref{eq: transport equation} captures the attained service times and times in the system at time zero, and the second term indicates that new tasks are injected at the origin at rate $\lambda_0$. The third term corresponds to tasks disappearing at a spatially varying rate; here $\bar r(\bmu_s, y) h_g^x(x, y)$ represents the rate at which tasks with attained service $x$ and time in the system $y$ complete service, whereas $h_g^y(x, y)$ represents the abandonment rate. The last term of \eqref{eq: transport equation} captures the evolution of attained service times and times in the system as a result of time passing and the service of tasks progressing. This movement is governed by the velocity field $\bv_t(x, y) \defeq [\bar r(\bmu_t, y)\ 1]^\intercal$.

For analyzing the FCFS many-server queue without abandonment, Kaspi and Ramanan considered a similar transport equation in \cite{kaspi2011law}. However, the variable in their equation is a one-dimensional measure instead of a two-dimensional measure as in \eqref{eq: transport equation}, and the term associated with departures depends linearly on this variable, whereas the nonlinear feedback term $\bar r(\bmu)$ appears in \eqref{eq: transport equation}. The transport equation in \cite{kaspi2011law} is also used by Kang and Ramanan in \cite{kang2010fluid} to analyze the same model with abandonment. There the fluid limit is given by two of these transport equations, coupled with each other and with several other equations which represent mass balance of auxiliary processes, such as the number of tasks in the system or queue. Our model admits the more streamlined characterization \eqref{eq: transport equation}.

\subsection{Measure-valued fluid limit}
\label{sub: measure-valued fluid limit}

We derive the limit as $N \to \infty$ of the processes $\bar\bmu^N$ in Section \ref{sec: proof of the fluid limit}, establishing that the preemptive LCFS partial service queue is characterized by the fluid problem in the limiting regime described in Section \ref{sub: limiting regime and standing assumptions}. Formally, we prove the following fluid limit.

\begin{theorem}
	\label{the: measure-valued fluid limit}
	Each subsequence of $\setc{\bar\bmu^N}{N \geq 1}$ has a further subsequence that converges weakly in $D_{\calM_F^+(\R_+^2)}[0, \infty)$ as $N \to \infty$ to a measure-valued process $\bar \bmu$ that solves the fluid problem introduced in Definition \ref{def: fluid problem} with probability one.
\end{theorem}

The preceding theorem proves the relative compactness of the sequence $\setc{\bar\bmu^N}{N \geq 1}$ and characterizes all subsequential limits in terms of the fluid problem. In principle, the subsequential limits could be nonunique since we do not show that solutions of the fluid problem are unique with respect to the initial condition. However, we formally establish in Theorem \ref{the: solution of transport equation for null initial condition} of Section \ref{sub: properties of the fluid problem} that the solution is unique when the initial condition is the null measure. As a result, Theorems \ref{the: measure-valued fluid limit} and \ref{the: solution of transport equation for null initial condition} yield the next corollary.

\begin{corollary}
	\label{cor: law of large numbers limit with null initial condition}
	If $\bar\bmu_0^N$ converges weakly to the null measure in $\calM_F^+(\R_+^2)$ as $N \to \infty$, then $\bar\bmu^N$ converges weakly in $D_{\calM_F^+(\R_+^2)}[0, \infty)$ to the unique measure-valued function $\bar\bmu$ that solves the fluid problem and satisfies that the initial condition $\bar\bmu_0$ is the null measure.
\end{corollary}

In particular, the above corollary provides the limit of a sequence of preemptive LCFS partial service queues that are initially empty, i.e., with no tasks at time zero.

\subsection{Limiting transport equation}
\label{sub: properties of the fluid problem}

We will analyze the fluid problem in Section \ref{sec: analysis of the fluid problem}. The existence of solutions to the fluid problem for a broad class of initial conditions will be proved using the fluid limit. As noted earlier, in the particular case where the initial condition is the null measure, we prove that the solution is unique and we obtain a closed-form expression.

\begin{theorem}
	\label{the: solution of transport equation for null initial condition}
	Consider the threshold value defined as follows:
	\begin{equation}
		\label{eq: equilibrium threshold}
		y_* \defeq \inf \asetc{y \geq 0}{\lambda_0 \int_0^y S_g(u, u)du \geq 1},
	\end{equation}
	where we let $y_* \defeq \infty$ if the integral is strictly less than one for all $y \geq 0$.
	There exists a unique $\bmu \in \calS^+$ such that $\bmu_0$ is the null measure. The threshold function of $\bmu$ is:
	\begin{equation*}
		\bar y_*\left(\bmu_t\right) =
		\begin{cases}
			\infty & \text{if} \quad 0 \leq t < y_*, \\
			y_* & \text{if} \quad y_* < \infty \quad \text{and} \quad t \geq y_*.
		\end{cases}
	\end{equation*}
	Moreover, $\bmu$ is determined by the following integrals:
	\begin{equation}
		\label{eq: solution with null initial condition}
		\angled{\phi}{\bmu_t} = \lambda_0 \int_0^t \phi\left(u \wedge y_*, u\right) S_g\left(u \wedge y_*, u\right)du \quad \text{for all} \quad t \geq 0 \quad \text{and} \quad \phi \in C_b\left(\R_+^2\right).
	\end{equation}
\end{theorem}

Remarkably, the threshold $\bar y_*(\bmu)$ starts at an infinite value and remains infinite forever or attains a finite equilibrium value $y_*$ in finite time, at $t = y_*$. The measures $\bmu_t$ are singular with support in the one-dimensional set $\setc{(u \wedge y_*, u) \in \R_+^2}{u \in \R_+}$ and mass distributed according to the survival function; the total mass depends on $\lambda_0$ as well.

Intuitively, the cases $y_* < \infty$ and $y_* = \infty$ correspond to overloaded and underloaded systems, respectively. Indeed, in the latter case, \eqref{eq: equilibrium threshold} implies that
\begin{equation*}
	\lambda_0\expect*{b \wedge c} = \lambda_0 \int_0^\infty \cprob*{b \geq u, c \geq u}du = \lambda_0 \int_0^\infty S_g(u, u)du \leq 1,
\end{equation*}
where $(b, c)$ is a random vector with joint probability density $g$. The expectation represents the mean amount of time that a task stays in the system if it is served uninterruptedly between its arrival and departure times, as if there were infinitely many servers. Thus, the left-hand side is the rate at which the total service requirement increases and the inequality states that this rate is smaller than the total service rate after normalization. This corresponds to an underloaded system where service capacity covers demand, whereas in the overloaded case $\lambda_0 E[b \wedge c] > 1$, the service capacity is insufficient to cover demand, the threshold stabilizes at a finite value $y^*$ and service curtailing occurs.

The interpretation of \eqref{eq: solution with null initial condition} is that, for large enough $N$, all but a negligible fraction of tasks are served from the moment of their arrival until they have been in the system $y_*$ units of time or have departed, due to abandonment or service completion; in underloaded systems, tasks are served throughout their stay. The main takeaway is that for all but a negligible fraction of tasks, if the service requirement and patience are $(b, c)$, then the task leaves the system with attained service $b \wedge c \wedge y_*$.

Proving that solutions are unique for general initial conditions is difficult due to the nonlinear feedback $\bar r(\bmu)$. However, we prove the following important structural result.

\begin{theorem}
	\label{the: structure of solutions}
	Suppose that $\bmu \in \calS^+$, let $\bnu$ be the unique solution of the fluid problem with initial condition the null measure and consider the measure-valued function $\bxi \defeq \bmu - \bnu$. Then $\bxi_t \in \calM_F^+(\R_+^2)$ and satisfies (a) and (b) of Definition \ref{def: fluid problem} for each $t \geq 0$. Also,
	\begin{equation}
		\label{eq: finite service rate for xi}
		\int_0^t \angled{\bar r(\bmu_s) h_g^x + h_g^y}{\bxi_s}ds < \infty \quad \text{for all} \quad t \geq 0,
	\end{equation}
	and for each $t \geq 0$ and $\varphi \in C_c^1(\R_+^3)$, we have:
	\begin{equation}
		\label{eq: transport equation without arrivals}
		\begin{split}
			\angled{\varphi_t}{\bxi_t} &= \angled{\varphi_0}{\bxi_0} - \int_0^t \angled{\varphi_s\left[\bar r(\bmu_s) h_g^x + h_g^y\right]}{\bxi_s}ds \\
			&+ \int_0^t \angled{\partial_t \varphi_s + \bar r\left(\bmu_s\right) \partial_x \varphi_s + \partial_y \varphi_s}{\bxi_s}ds.
		\end{split}
	\end{equation}
\end{theorem}

In other words, solutions $\bmu$ of the fluid problem can be expressed as the superposition of two measure-valued functions: one satisfies \eqref{eq: transport equation} with $\lambda_0 = 0$ and prescribed threshold~$\bar y_*(\bmu)$, the other one is the unique solution of the fluid problem with null initial condition. The former function represents the evolution of tasks already present in the system at time zero, whereas the latter function describes the evolution of tasks that arrive afterwards. In the stochastic partial service queue, tasks present at time zero have no impact on the service rate of tasks that arrive after time zero, due to the preemptive LCFS policy. The theorem establishes an analogous property for solutions of the fluid problem.

Using this property, we derive the long-term limit of solutions, without assuming uniqueness of solutions with respect to the initial condition. For the unique solution with null initial condition, the vague limit is a measure $\mu_\infty \in \calM^+(\R_+^2)$, given by 
\begin{equation}
	\label{eq: global attractor}
	\begin{split}
		\angled{\phi}{\mu_\infty} &\defeq \lambda_0 \int_0^\infty \phi\left(u \wedge y_*, u\right) S_g\left(u \wedge y_*, u\right)du \\
		&= \lim_{t \to \infty} \lambda_0 \int_0^t \phi\left(u \wedge y_*, u\right) S_g\left(u \wedge y_*, u\right)du
	\end{split}
\end{equation}
for all $\phi \in C_c(\R_+^2)$. Moreover, we prove that the measure-valued function that represents tasks already present at time zero always converges vaguely to the null measure. Hence, $\mu_\infty$ is the unique fixed point of the fluid problem and a global attractor with respect to the vague topology, and with respect to the weak topology under mild conditions. As noted earlier, the measure $\mu_\infty$ solves the stationary transport equation obtained from the fluid model postulated in \cite{ferragut2025last} by Ferragut, Goldsztajn and Paganini. Thus, the results proved in the present paper provide rigorous support to the conclusions obtained in \cite{ferragut2025last}.

\begin{theorem}
	\label{the: long-term behavior of solutions}
	Suppose that $\bmu \in \calS^+$ and define $\bnu$ and $\bxi$ as in Theorem \ref{the: structure of solutions}. Then the measure $\bxi_t$ converges vaguely to the null measure as $t \to \infty$ and the measures $\bmu_t$ and $\bnu_t$ converge vaguely to the measure $\mu_\infty$ defined in \eqref{eq: global attractor}. Further, $\bxi_t$ in fact converges weakly to the null measure if $y_* < \infty$. If in addition the survival function is such that
	\begin{equation}
		\label{eq: condition for fininte total mass in limit}
		\int_0^\infty S_g(u \wedge y_*, u)du < \infty,
	\end{equation}
	then $\mu_\infty \in \calM_F^+(\R_+^2)$ and the measures $\bmu_t$ and $\bnu_t$ converge weakly to $\mu_\infty$ as $t \to \infty$.
\end{theorem}

The weak convergence results can also be obtained in the underloaded case $y_* = \infty$ under mild additional assumptions, which are discussed at the end of Section \ref{sec: analysis of the fluid problem}.

The rest of the paper is devoted to proving the theorems stated in this section. We first prove the measure-valued fluid limit in Section \ref{sec: proof of the fluid limit} and then analyze the fluid problem in Section \ref{sec: analysis of the fluid problem}. However, these sections can be read independently and in any order.

\section{Proof of the fluid limit}
\label{sec: proof of the fluid limit}

We follow the same program as Kaspi and Ramanan in \cite{kaspi2011law} for deriving the fluid limit of the FCFS many-server queue without abandonment. This program can be summarized as follows. We first derive and analyze a stochastic equation for the pre-limit system dynamics. Then we prove tightness for several sequences of measure-valued processes involved in the pre-limit equation, including the process $\bar \bmu^N$ that describes the normalized state of the system. We conclude by characterizing the subsequential limits of these sequences, showing that the pre-limit equation turns into the transport equation \eqref{eq: transport equation} in the limit.

While the overall program is the same as in \cite{kaspi2011law}, the setting considered in this paper presents significant new challenges; also with respect to other papers in the literature, as already noted in Section \ref{sec: introduction}. The main one of these challenges is that the service rate of tasks depends on the measure-valued state of the system. As a result, processes that describe departures from the system and appear in the pre-limit equation are harder to analyze than in \cite{kaspi2011law}; namely, it is substantially more difficult to derive their compensators. In addition, the characterization of subsequential limits involves entirely new arguments for deriving the limits of processes driven by the service rate of tasks, associated with the last two terms of \eqref{eq: transport equation}. Several proofs are also more involved since tasks are regarded as particles in a two-dimensional space, instead of particles in one-dimensional spaces.

\subsection{Outline of the proof}
\label{sub: outline of the proof}

As observed earlier, the first step of the proof is to derive a stochastic equation for describing the pre-limit dynamics of the system. For this purpose, define
\begin{equation*}
	\bA_t^N(\varphi) \defeq \sum_{i = 1}^{\calR^N(t)} \varphi_{a_i^N}\left(0, 0\right) \quad \text{and} \quad \bD_t^N(\varphi) \defeq \sum_{i = 1 - M^N}^{\calR^N(t)} \ind{t \geq d_i^N}\varphi_{d_i^N}\left(\bx_i^N\left(d_i^N\right), \by_i^N\left(d_i^N\right)\right)
\end{equation*}
for all $t \geq 0$ and $\varphi \in C_b(\R_+^3)$. The processes $\bA^N(\varphi)$ and $\bD^N(\varphi)$ are called the arrival and departure processes associated with the test function $\varphi$, respectively; observe that $\calR^N$ is the arrival process associated with the function that is identically equal to one. We establish in Proposition \ref{prop: arrivals-flow-departures decomposition} of Section \ref{sub: pre-limit stochastic equation} that
\begin{equation}
	\label{eq: pre-limit equation}
	\begin{split}
		\angled{\varphi_t}{\bmu_t^N} &= \angled{\varphi_0}{\bmu_0^N} + \bA_t^N(\varphi) - \bD_t^N(\varphi) \\
		&+ \int_0^t \angled{\partial_t\varphi_s + r^N\left(\bmu_s^N\right)\partial_x\varphi_s + \partial_y\varphi_s}{\bmu_s^N}ds
	\end{split}
\end{equation} 
surely. This pre-limit equation has the same structure as the transport equation \eqref{eq: transport equation}, with the second and third terms on the right-hand side representing the effect of arrivals and departures, respectively, in the integral of a test function $\varphi_t$ with respect to $\bmu_t^N$. The last term reflects how this integral changes as time passes and the service of tasks progresses, and has the same form as the last term of the transport equation \eqref{eq: transport equation}.

For fixed $t \geq 0$ and $\varphi \in C_c^1(\R_+^3)$, a law of large numbers argument yields
\begin{equation}
	\label{eq: law of large numbers for arrivals}
	\lim_{N \to \infty} \bA_t^N(\varphi) = \lambda_0 \int_0^t \varphi_s(0, 0)ds 
\end{equation}
almost surely; see Lemma \ref{lem: limit of arrival process} in Section \ref{sub: characterization of subsequential limits}. Analyzing the departure process $\bD^N(\varphi)$ is significantly more challenging. We establish in Theorem \ref{the: compensator of departures-driven process} of Section \ref{sub: departure processes} that
\begin{equation*}
	\bD_t^{N, c}(\varphi) \defeq \int_0^t \angled{\varphi_s\left[r^N\left(\bmu_s^N\right)h_g^x + h_g^y\right]}{\bmu_s^N}ds
\end{equation*}
is the compensator of $\bD^N(\varphi)$, where the integrand should be interpreted as the intensity of departures at time $s$. We also prove that the local martingale $\bD^N(\varphi) - \bD^{N, c}(\varphi)$ is a locally square-integrable martingale, and Corollary \ref{cor: limit of departures martingale} bounds its predictable quadratic variation. We leverage this bound in Section \ref{sub: characterization of subsequential limits} to characterize the limit of the departure processes in terms of the limits of their compensators.

In Section \ref{sub: tightness results}, we normalize all terms of the pre-limit equation \eqref{eq: pre-limit equation}, and consider the measure-valued processes $\bar\bA^N$, $\bar\bD^N$ and $\bar\bD^{N, c}$ defined as:
\begin{equation*}
	\bar\bA_t^N(\varphi) \defeq \frac{\bA_t^N(\varphi)}{N}, \quad \bar\bD^N_t(\varphi) \defeq \frac{\bD_t^N(\varphi)}{N}, \quad \bar\bD_t^{N, c}(\varphi) \defeq \frac{\bD_t^{N, c}(\varphi)}{N}, \quad \text{with} \quad \varphi \in C_b\left(\R_+^3\right).
\end{equation*}
Using Jakubowski's criteria, we prove in Theorem \ref{the: relative compactness of measured-valued processes} that the sequences
\begin{equation*}
	\asetc{\bar\bmu^N}{N \geq 1}, \quad \asetc{\bar\bA^N}{N \geq 1}, \quad \asetc{\bar\bD^N}{N \geq 1} \quad \text{and} \quad \asetc{\bar\bD^{N, c}}{N \geq 1}
\end{equation*}
are tight; the former in $D_{\calM_F^+(\R_+^2)}[0, \infty)$ and the latter three in $D_{\calM_F^+(\R_+^3)}[0, \infty)$. Observe that \eqref{eq: pre-limit equation} can be expressed in terms of the above processes as follows:
\begin{equation}
	\label{eq: pre-limit equation in normalized notation}
	\begin{split}
		\angled{\varphi_t}{\bar\bmu_t^N} &= \angled{\varphi_0}{\bar\bmu_0^N} + \bar\bA_t^N(\varphi) - \bar\bD_t^N(\varphi) \\
		&+ \int_0^t \angled{\partial_t\varphi_s + \bar r\left(\bar\bmu_s^N\right)\partial_x\varphi_s + \partial_y\varphi_s}{\bar\bmu_s^N}ds.
	\end{split}
\end{equation}

Let $\bar\calR^N \defeq \calR^N / N$ denote the normalized arrival process. Then $\bar\calR^N$ converges weakly in the topology of uniform convergence over compacts sets to the function $t \mapsto \lambda_0 t$. This follows from Assumption \ref{ass: standing assumptions 1} and the functional law of large numbers for renewal processes; e.g., see \cite[Theorem 5.10]{chen2001fundamentals}. Therefore, the sequence formed by the processes
\begin{equation*}
	\bar\bX^N \defeq \left(\bar\bmu_0^N, \bar\bmu^N, \bar\bA^N, \bar\bD^N, \bar\bD^{N, c}, \bar\calR^N\right)
\end{equation*}
is tight in $\calM_F^+(\R_+^2) \times D_{\calM_F^+(\R_+^2)}[0, \infty) \times [D_{\calM_F^+(\R_+^3)}[0, \infty)]^3 \times D_\R[0, \infty)$. Prohorov's theorem implies that every subsequence has a further subsequence which converges weakly.

In Section \ref{sub: characterization of subsequential limits}, we fix a convergent subsequence and characterize its limit
\begin{equation*}
	\bar\bX \defeq \left(\bar\bmu_0, \bar\bmu, \bar\bA, \bar\bD, \bar\bD^{c}, \bar\calR\right).
\end{equation*}
Using \eqref{eq: law of large numbers for arrivals}, and leveraging Corollary \ref{cor: limit of departures martingale} to bound the local martingales $\bar\bD^N(\varphi) - \bar\bD^{N, c}(\varphi)$, we show in Proposition \ref{prop: set of probability one} that, almost surely for all $t \geq 0$ and $\varphi \in C_b(\R_+^3)$, we have:
\begin{equation}
	\label{eq: first results on subsequential limits}
	\bar\bA_t(\varphi) = \lambda_0 \int_0^t \varphi_s(0, 0)ds \quad \text{and} \quad \bar\bD_t(\varphi) = \bar\bD_t^c(\varphi).
\end{equation}
Moreover, we establish that the limit $\bar\bmu$ satisfies properties (a) and (b) of Definition \ref{def: fluid problem} in Propositions \ref{prop: support of fluid limit} and \ref{prop: limiting measure of y-intervals}, respectively. For finishing the proof of Theorem \ref{the: measure-valued fluid limit}, we must establish that $\bar\bmu$ satisfies property (c) of Definition \ref{def: fluid problem}, and that the remaining terms of \eqref{eq: pre-limit equation in normalized notation} converge to the corresponding terms of \eqref{eq: transport equation} along the fixed subsequence.

Specifically, in view of \eqref{eq: first results on subsequential limits}, we must consider the following measure-valued processes
\begin{align*}
	&\bar\bD_t^{N, c}(\varphi) \defeq \int_0^t \angled{\varphi_s\left[\bar r\left(\bar\bmu_s^N\right)h_g^x + h_g^y\right]}{\bar\bmu_s^N}ds, \\
	&\bar\bY_t^N(\varphi) \defeq \int_0^t \angled{\partial_t\varphi_s + \bar r\left(\bar\bmu_s^N\right)\partial_x\varphi_s + \partial_y\varphi_s}{\bar\bmu_s^N}ds,
\end{align*}
to obtain the limits of the third and fourth terms on the right-hand side of \eqref{eq: pre-limit equation in normalized notation}. For this purpose, we study integrals of the form $\langle\psi \bar r(\bar \bmu_t^N), \bar \bmu_t^N\rangle$ in Section \ref{subsub: limits involving service rates}, with $\psi$ a continuous and bounded function. While the limit of the threshold process $\bar y_*(\bar\bmu^N)$ could not exist, we derive the limits of the latter integrals in Proposition \ref{prop: limits of integrals involving r}. In Section \ref{subsub: limits involving hazard rates}, we consider integrals involving also the possibly unbounded hazard rates and prove that $\bar\bmu$ satisfies property (c) of Definition \ref{def: fluid problem}. Using the results of Sections \ref{subsub: limits involving service rates} and \ref{subsub: limits involving hazard rates}, we prove in Theorem \ref{the: characterization of subsequential limits} that $\bar\bmu$ satisfies property (d) of Definition \ref{def: fluid problem}, completing the proof.

\subsection{Pre-limit stochastic equation}
\label{sub: pre-limit stochastic equation}

In this section we derive the stochastic equation that describes the pre-limit dynamics of the system. For this purpose, we will fix the number of servers $N$ and omit it from the notation for brevity. The pre-limit equation is obtained in the following proposition.

\begin{proposition}
	\label{prop: arrivals-flow-departures decomposition}
	For each $t \geq 0$ and $\varphi \in C_c^1(\R_+^3)$, we have
	\begin{equation*}
		\angled{\varphi_t}{\bmu_t} = \angled{\varphi_0}{\bmu_0} + \bA_t(\varphi) - \bD_t(\varphi) + \int_0^t \angled{\partial_t\varphi_s + r(\bmu_s)\partial_x\varphi_s + \partial_y\varphi_s}{\bmu_s}ds
	\end{equation*} 
	surely. Recall that the arrival process $\bA(\varphi)$ and departure process $\bD(\varphi)$ are defined as:
	\begin{equation*}
		\bA_t(\varphi) \defeq \sum_{i = 1}^{\calR(t)} \varphi_{a_i}(0, 0) \quad \text{and} \quad \bD_t(\varphi) \defeq \sum_{i = 1 - M}^{\calR(t)} \ind{t \geq d_i}\varphi_{d_i}\left(\bx_i(d_i), \by_i(d_i)\right).
	\end{equation*}
\end{proposition}

\begin{proof}
	Consider the random times $\tau_k$ defined recursively as follows:
	\begin{equation*}
		\tau_0 \defeq 0 \quad \text{and} \quad \tau_{k + 1} \defeq \min\asetc{a_i, d_j}{i \geq 1, a_i > \tau_k, j \geq 1, d_j > \tau_k}.
	\end{equation*}
	If we refer to the arrivals and departures as the events of the system, then $\tau_k$ is the time of the $k$th event of the system, as in Appendix \ref{app: construction of sample paths}. Hence,
	\begin{equation*}
		\bu(t) \defeq \angled{\varphi_t}{\bmu_t} = \sum_{i = 1 - M}^{\calR(t)} \ind{t < d_i} \varphi_t\left(\bx_i(t), \by_i(t)\right)
	\end{equation*}
	is differentiable in $(\tau_k, \tau_{k + 1})$ as a function of $t$. Indeed, the number of summands remains fixed throughout this interval and $r(\bmu, \by_i)$ is constant in $(\tau_k, \tau_{k + 1})$ for each $i$. By \eqref{eq: attained service and sojourn time evolution},
	\begin{equation}
		\label{eq: time derivative of integral of test function}
		\partial_t \bu(t) = \angled{\partial_t\varphi_t + r(\bmu_t)\partial_x\varphi_t + \partial_y\varphi_t}{\bmu_t} \quad \text{for all} \quad t \in \left(\tau_k, \tau_{k + 1}\right).
	\end{equation}
	
	On the other hand, at each time $\tau_k$, we have:
	\begin{equation*}
		\bu\left(\tau_k\right) - \bu\left(\tau_k^-\right) = \sum_{i = 1}^\infty \ind{a_i = \tau_k}\varphi_{a_i}\left(0, 0\right) - \sum_{i = 1 - M}^\infty \ind{d_i = \tau_k} \varphi_{d_i}\left(\bx_i\left(d_i\right),  \by_i\left(d_i\right)\right).
	\end{equation*}
	We establish in Appendix \ref{sub: properties of the construction} that $\tau_k \to \infty$ as $k \to \infty$. Thus, $k_t \defeq \max\setc{k \geq 0}{\tau_k \leq t}$ is well-defined and finite. It follows that
	\begin{align*}
		\bu(t) - \bu(0) &= \bu(t) - \bu\left(\tau_{k_t}\right) + \sum_{k = 0}^{k_t - 1} \left[\bu\left(\tau_{k + 1}^-\right) - \bu\left(\tau_k\right)\right] + \sum_{k = 1}^{k_t} \left[\bu\left(\tau_k\right) - \bu\left(\tau_k^-\right)\right].
	\end{align*}
	
	Because $\varphi$ is uniformly continuous and $\bmu$ is right-continuous,  the following inequality, which holds for all $0 \leq t \leq s \leq t + 1$ implies that $\bu$ is right-continuous as well:
	\begin{align*}
		\left|\bu(s) - \bu(t)\right| &\leq \left|\angled{\varphi_s - \varphi_t}{\bmu_s}\right| + \left|\angled{\varphi_t}{\bmu_s} - \angled{\varphi_t}{\bmu_t}\right| \\
		&\leq \bmu_s(\R_+^2) \sup_{x, y \in \R_+} \left|\varphi_s(x, y) - \varphi_t(x, y)\right| + \left|\angled{\varphi_t}{\bmu_s} - \angled{\varphi_t}{\bmu_t}\right| \\
		&\leq \left[M + \calR(t + 1)\right] \sup_{x, y \in \R_+} \left|\varphi_s(x, y) - \varphi_t(x, y)\right| + \left|\angled{\varphi_t}{\bmu_s} - \angled{\varphi_t}{\bmu_t}\right|.
	\end{align*}
	Indeed, the right-hand side converges to zero as $s \downarrow t$ while $t$ remains fixed. As a result, recalling that $\bu$ is differentiable over $(\tau_k, \tau_{k + 1})$, it follows from \eqref{eq: time derivative of integral of test function} that
	\begin{equation*}
		\bu(t) - \bu\left(\tau_{k_t}\right) + \sum_{k = 0}^{k_t - 1} \left[\bu\left(\tau_{k + 1}^-\right) - \bu\left(\tau_k\right)\right] = \int_0^t \angled{\partial_t\varphi_s + r(\bmu_s)\partial_x\varphi_s + \partial_y\varphi_s}{\bmu_s}ds.
	\end{equation*}
	
	Recall that $k_t < \infty$ and note that the number of tasks that visit the system in $[0, t]$ is upper bounded by $M + \calR(t)$; in particular, the total number of departures is also bounded. Thus, the following iterated sums have finitely many nonzero terms surely, and we get:
	\begin{align*}
		\sum_{k = 1}^{k_t} \left[\bu\left(\tau_k\right) - \bu\left(\tau_k^-\right)\right] &= \sum_{k = 1}^{k_t}\sum_{i = 1}^\infty \ind{a_i = \tau_k}\varphi_{a_i}\left(0, 0\right) - \sum_{k = 1}^{k_t}\sum_{i = 1 - M}^\infty \ind{d_i = \tau_k} \varphi_{d_i}\left(\bx_i\left(d_i\right),  \by_i\left(d_i\right)\right) \\
		&= \sum_{i = 1}^\infty \ind{a_i \leq \tau_{k_t}}\varphi_{a_i}\left(0, 0\right) - \sum_{i = 1 - M}^\infty \ind{d_i \leq \tau_{k_t}} \varphi_{d_i}\left(\bx_i\left(d_i\right),  \by_i\left(d_i\right)\right)\\
		&= \sum_{i = 1}^\infty \ind{a_i \leq t}\varphi_{a_i}\left(0, 0\right) - \sum_{i = 1 - M}^\infty \ind{d_i \leq t} \varphi_{d_i}\left(\bx_i\left(d_i\right),  \by_i\left(d_i\right)\right).
	\end{align*}
	The right-hand side equals $\bA_t(\varphi) - \bD_t(\varphi)$, which completes the proof.
\end{proof}

As observed in Section \ref{sub: outline of the proof}, the first two terms of the pre-limit equation are arrival and departure processes, which capture the effects of arrivals and departures, respectively, in the integral of a test function with respect to $\bmu_t$. Both processes are pure jump processes having jumps precisely at the arrival and departure times, respectively. The last term of the pre-limit equation is continuous over time and describes the effect of time passing.

\subsection{Departure processes}
\label{sub: departure processes}

As noted in Section \ref{sub: outline of the proof}, the limits of the arrival processes as the number of servers goes to infinity are fairly easy to characterize. We postpone this analysis until Section \ref{sub: characterization of subsequential limits} and focus here on departure processes, which are significantly harder to analyze. We derive their compensators and bound the predictable quadratic variation of the compensated processes. In most of this section, the index $N$ is fixed and omitted from the notation.

\subsubsection{Compensators}
\label{subsub: compensators}

In this section we prove that the compensator of $\bD(\varphi)$ is
\begin{equation*}
	\bD_t^c(\varphi) \defeq \int_0^t \angled{\varphi_s\left[r\left(\bmu_s\right)h_g^x + h_g^y\right]}{\bmu_s}ds \quad \text{for all} \quad \varphi \in C_b(\R_+^3).
\end{equation*}
Before proving this formally, let us outline the general direction of the proof.

Suppose that $\varphi$ is nonnegative and thus $\bD(\varphi)$ nondecreasing; we will see that this can be assumed without any loss of generality. The above expression for the compensator will be obtained by computing $E[\bD_T(\varphi)]$ for an arbitrary stopping time $T$. The rationale is that $E[\bD_T(\varphi)] = E[\bD_T^c(\varphi)]$ for all stopping times $T$ if $\bD^c(\varphi)$ is the compensator of $\bD(\varphi)$; verifying this is equivalent to verifying the local martingale property.

In order to compute $E[\bD_T(\varphi)]$, we observe that Tonelli's theorem implies that
\begin{equation*}
	\expect*{\bD_T(\varphi)} = \expect*{\sum_{i = 1 - M}^\infty \expect*{\bD_T^i(\varphi) | M}}, \quad \text{where} \quad \bD_t^i(\varphi) \defeq \ind{d_i \leq t}\varphi_{d_i}\left(\bx_i(d_i), \by_i(d_i)\right).
\end{equation*}
Therefore, we will fix a stopping time $T$ and an index $i$ to compute $E[\bD_T^i(\varphi) | M]$. We will do this by partitioning $\R_+$ into intervals of equal size and taking the limit as the size of the intervals approaches zero. Let $\varphi_n^m$ and $\calF_n^m$ stand for $\varphi_t(\bx_i(t), \by_i(t))$ and $\calF_t$ when $t = m / n$, respectively. Our computation will be based on deriving the following limit:
\begin{equation}
	\label{eq: motivation for compensator}
	\begin{split}
		\expect*{\bD_T^i(\varphi) | M} &= \lim_{n \to \infty} \sum_{m = 0}^\infty \expect*{\ind{a_i \leq \frac{m}{n} < d_i \leq \frac{m + 1}{n}, \frac{m}{n} < T} \varphi_n^m | M} \\
		&= \lim_{n \to \infty} \sum_{m = 0}^\infty \expect*{\ind{\frac{m}{n} < T} \varphi_n^m \cprob*{a_i \leq \frac{m}{n} < d_i \leq \frac{m + 1}{n} | \calF_n^m} | M} \\
		&= \lim_{n \to \infty} \expect*{\sum_{m = 0}^\infty \ind{\frac{m}{n} < T} \varphi_n^m \cprob*{a_i \leq \frac{m}{n} < d_i \leq \frac{m + 1}{n} | \calF_n^m} | M}.
	\end{split}
\end{equation}

Computing the limit on the right-hand side requires a series of lemmas. The first of these lemmas will be used to characterize the limit of the conditional probability as $n \to \infty$, and it explains why the hazard rate vector field appears in the compensator.

\begin{lemma}
	\label{lem: hazard rate terms}
	Consider the following events:
	\begin{align*}
		&A_i^\alpha(s) \defeq \left\{a_i \leq s < d_i, r\left(\bmu_s, \by_i(s)\right) = \alpha\right\}, \\
		&B_i(s, t) \defeq \left\{a_j, d_k \notin (s, t]\ \text{for all}\ j \geq 1\ \text{and}\ 1 - M \leq k \neq i\right\},
	\end{align*}
	defined for all $i \in \Z$, $\alpha \in \{0, 1\}$ and $0 \leq s < t$. The former event corresponds to task $i$ being present at time $s$ with service rate $\alpha$, and the latter event corresponds to no tasks arriving in $(s, t]$ and no tasks leaving with the possible exception of task $i$. Let $I_i^\alpha(s)$ and $J_i(s, t)$ denote the indicator functions of the above events, respectively, and define the random variable
	\begin{equation}
		\label{eq: definition of H}
		H_i^\alpha(s, t) \defeq \frac{S_g\left(\bx_i(s), \by_i(s)\right) - S_g\left(\bx_i(s) + \alpha(t - s), \by_i(s) + t - s\right)}{S_g\left(\bx_i(s), \by_i(s)\right)}.
	\end{equation}
	Consider the $\sigma$-algebra $\calG_i(s, t) \defeq \sigma\left(\bx_i(s), \by_i(s), I_i^0(s), I_i^1(s), J_i(s, t)\right)$. Then
	\begin{equation}
		\label{eq: conditional probability equal to H}
		\cprob*{d_i \leq t | \calG_i(s, t)}I_i^\alpha(s)J_i(s, t) = H_i^\alpha(s, t) I_i^\alpha(s) J_i(s, t)
	\end{equation}
	with probability one for all $\alpha \in \{0, 1\}$ and $0 \leq s < t$. Furthermore,
	\begin{equation}
		\label{eq: derivative of H}
		\lim_{n \to \infty} \frac{H_i^\alpha(s_n, t_n)}{t_n - s_n} = \alpha h_g^x\left(\bx_i(t), \by_i(t)\right) + h_g^y\left(\bx_i(t), \by_i(t)\right)
	\end{equation}
	surely for all $t \geq 0$ and any two sequences such that $s_n \uparrow t$ and $t_n \downarrow t$ as $n \to \infty$.
\end{lemma}

\begin{proof}
	Clearly, the right-hand side of \eqref{eq: conditional probability equal to H} is zero when $I_i^\alpha(s) = 0$ or $J_i(s, t) = 0$. Otherwise, we obtain the following Radon-Nikod\'ym derivative:
	\begin{align*}
		\cprob*{d_i \leq t | \bx_i(s) = x, \by_i(s) = y, I_i^\alpha(s) = 1, I_i^{1 - \alpha}(s) = 0, J_i(s, t) = 1} \\
		=\cprob*{x + \alpha(t - s) \geq b_i\ \text{or}\ y + t - s \geq c_i | x < b_i\ \text{and}\ y < c_i}
	\end{align*}
	Indeed, $A_i^\alpha(s)$ implies that task $i$ is present at time $s$ with service rate $\alpha$. In particular, we must have $x < b_i$ and $y < c_i$. In addition, $B_i(s, t)$ implies that the service rate of task $i$ does not change in $[s, t]$ unless the task leaves the system. Thus, $\bx_i(t) = x + \alpha(t - s)$ if task $i$ has not left by time $t$, so $d_i \leq t$ is equivalent to $x + \alpha(t - s) \geq b_i$ or $y + t - s \geq c_i$. The above expression coincides with \eqref{eq: definition of H} when $\bx_i(s) = x$ and $\by_i(s) = y$, proving \eqref{eq: conditional probability equal to H}.
	
	For proving \eqref{eq: derivative of H}, fix $t \geq 0$ and a sample path of $(\bx_i, \by_i)$. Then observe that
	\begin{align*}
		H_i^\alpha(s_n, t_n) = \frac{1}{S_g\left(\bx_i(s_n), \by_i(s_n)\right)} \left[\int_{A_n} g(x, y)dxdy + \int_{B_n} g(x, y)dxdy - \int_{C_n} g(x, y)dxdy\right],
	\end{align*}
	where the above integration domains are defined as:
	\begin{align*}
		&A_n \defeq \left[\bx_i(s_n), \bx_i(s_n) + \alpha\left(t_n - s_n\right)\right] \times \left[\by_i(s_n), \infty\right), \\
		&B_n \defeq \left[\bx_i(s_n), \infty\right) \times \left[\by_i(s_n), \by_i(s_n) + t_n - s_n\right],  \\
		&C_n \defeq \left[\bx_i(s_n), \bx_i(s_n) + \alpha\left(t_n - s_n\right)\right] \times \left[\by_i(s_n), \by_i(s_n) + t_n - s_n\right].
	\end{align*}
	
	By Assumption \ref{ass: standing assumptions 2}, there exists $K \geq 0$ such that $g(x, y) \leq K$ for all $(x, y)$ in some open neighborhood of $(\bx_i(t), \by_i(t))$. Hence, the last integral is bounded by $\alpha K (t_n - s_n)^2$ for large enough $n$, which approaches zero faster than $t_n - s_n$. Moreover,
	\begin{align*}
		&\left|\int_{A_n} g(x, y)dxdy - \int_{\bx_i(s_n)}^{\bx_i(s_n) + \alpha (t_n - s_n)} \left[\int_{\by_i(t)}^\infty g(x, y)dy\right]dx\right| \leq \alpha K \left(t_n - s_n\right)\left[\by_i(t) - \by_i(s_n)\right], \\
		&\left|\int_{B_n} g(x, y)dxdy - \int_{\by_i(s_n)}^{\by_i(s_n) + t_n - s_n}\left[\int_{\bx_i(t)}^\infty g(x, y)dx\right]dy\right| \leq K \left(t_n - s_n\right)\left[\bx_i(t) - \bx_i(s_n)\right].
	\end{align*}
	It follows from \eqref{eq: attained service and sojourn time evolution} that $\bx_i$ and $\by_i$ are Lipschitz. Thus, the above inequalities imply that
	\begin{align*}
		&\lim_{n \to \infty} \frac{1}{t_n - s_n}\int_{A_n} g(x, y)dxdy = \lim_{n \to \infty} \frac{1}{t_n - s_n}\int_{\bx_i(s_n)}^{\bx_i(s_n) + \alpha (t_n - s_n)} \left[\int_{\by_i(t)}^\infty g(x, y)dy\right]dx \\
		&\lim_{n \to \infty} \frac{1}{t_n - s_n}\int_{B_n} g(x, y)dxdy = \lim_{n \to \infty} \frac{1}{t_n - s_n}\int_{\by_i(s_n)}^{\by_i(s_n) + t_n - s_n}\left[\int_{\bx_i(t)}^\infty g(x, y)dx\right]dy.
	\end{align*}
	The integrals inside the square brackets are equal to $-\partial_xS_g(x, \by_i(t))$ and $-\partial_yS_g(\bx_i(t), y)$, respectively. Since these partial derivatives are continuous by Assumption \ref{ass: standing assumptions 2}, we obtain:
	\begin{align*}
		&\lim_{n \to \infty} \frac{1}{t_n - s_n}\int_{A_n} g(x, y)dxdy = -\alpha\partial_xS_g\left(\bx_i(t), \by_i(t)\right), \\
		&\lim_{n \to \infty} \frac{1}{t_n - s_n}\int_{B_n} g(x, y)dxdy = -\partial_y S_g\left(\bx_i(t), \by_i(t)\right).
	\end{align*}
	Then \eqref{eq: derivative of H} follows from Definition \ref{def: hazard rate}.
\end{proof}

The following two lemmas will be used to bound the random variable $H_i^\alpha(s, t)$ that will appear on the right-hand side of \eqref{eq: motivation for compensator} after invoking Lemma \ref{lem: hazard rate terms}. The bounds provided in these lemmas will be needed to invoke the dominated convergence theorem. The first lemma is rather standard, so we defer the proof until Appendix \ref{app: proofs of auxiliary results}.

\begin{restatable}{lemma}{lemmafourthree}
	\label{lem: bound for H}
	Let $C \subset U_g$ be compact. The minimum of $S_g$ over $C$ exists and is positive. Also, there exist constants $L \geq 0$ and $\delta > 0$, which may depend on the set $C$, such that:
	\begin{equation*}
		\left|S_g(x, y) - S_g(x + \alpha \theta, y + \theta)\right| \leq L \theta \quad \text{for all} \quad \alpha \in \{0, 1\}, \quad \theta \in [0, \delta] \quad \text{and} \quad (x, y) \in C.
	\end{equation*}
\end{restatable}

Next we leverage the preceding lemma to bound the probability that the event $B_i(s, t)$, defined in Lemma \ref{lem: hazard rate terms}, does not occur given the history of the system up to time $s$. This bound will also be needed for a dominated convergence argument.

\begin{lemma}
	\label{lem: no but possibly one event in small interval}
	Consider the process defined as:
	\begin{equation*}
		\bw(t) \defeq t - \max \asetc{a_i}{i \geq 1, a_i \leq t} \quad \text{for all} \quad t \geq 0, 
	\end{equation*}
	which represents the amount of time elapsed since the last arrival at time $t$. For each $i \in \Z$ and $0 \leq s < t$, the following inequality holds almost surely:
	\begin{equation}
		\label{eq: bound for probability of B_i^c(s, t) given F_s}
		\cprob*{B_i^c(s, t) | \calF_s} \leq \frac{\bar F\left(\bw(s)\right) - \bar F\left(\bw(s) + t - s\right)}{\bar F\left(\bw(s)\right)} + \sum_{j = 1 - M}^{\calR(s)} H_j^1(s, t),
	\end{equation}
	where $B_i^c(s, t)$ is the complement of $B_i(s, t)$. Also, let us fix a realization and suppose that there exist a constant $\eta \in (0, \kappa)$, a compact set $(0, 0) \in C \subset U_g$ and a time $T \geq 0$ such that
	\begin{equation}
		\label{eq: conditions for bound for B_i^c(s, t)}
		\bw(t) \leq \eta \quad \text{and} \quad \left(\bx_j(t), \by_j(t)\right) \in C \quad \text{if} \quad 1 - M \leq j \leq \calR(T) \quad \text{and} \quad t \in [0, T].
	\end{equation}
	Then there exist two deterministic constants $K \geq 0$ and $\delta > 0$, which only depend on $\eta$ and $C$, such that if $t - \delta \leq s < t \leq T + \delta$, then we have:
	\begin{equation*}
		\frac{\bar F\left(\bw(s)\right) - \bar F\left(\bw(s) + t - s\right)}{\bar F\left(\bw(s)\right)} + \sum_{j = 1 - M}^{\calR(s)} H_j^1(s, t) \leq K\left[1 + M + \calR(s)\right](t - s).
	\end{equation*}
\end{lemma}

\begin{proof}
	The bound \eqref{eq: bound for probability of B_i^c(s, t) given F_s} follows from computing the probability that an arrival or departure occurs in $(s, t]$ if all tasks present in the system at time $s$ are served at unit rate, i.e., as if the system had infinitely many servers. Indeed, the first term on the right-hand side of \eqref{eq: bound for probability of B_i^c(s, t) given F_s} is the probability that an arrival occurs in $(s, t]$ given that $\bw(s)$ units of time have passed since the last arrival. On the other hand, if task $j$ is in the system at time $s$, then the $j$th term of the summation on the right-hand side of \eqref{eq: bound for probability of B_i^c(s, t) given F_s} is the probability that task $j$ leaves the system in $(s, t]$ provided that task $j$ is served at unit rate. Note that $j \leq \calR(s)$ for tasks in the system at time $s$ and $H_j^1(s, t) \geq 0$ even if task $j$ has already left by time $s$.
	
	In order to prove the second claim, choose some $\delta_0 > 0$ such that:
	\begin{equation*}
		\eta_0 \defeq \eta + \delta_0 < \kappa, \quad C_0 \defeq \asetc{(x + u, y + v) \in \R_+^2}{(x, y) \in C\ \text{and}\ u, v \in [0, \delta_0]} \subset U_g.
	\end{equation*}
	Note that $\delta_0$ can be selected using only the constant $\eta$ and the set $C$. The functions $\bx_j$ and $\by_j$ are Lipschitz with unit constant, and $\bw(t) - \bw(s) \leq t - s$ for all $0 \leq s < t$. Thus,
	\begin{equation*}
		\bw(t) \leq \eta_0 \quad \text{and} \quad \left(\bx_j(t), \by_j(t)\right) \in C_0 \quad \text{for} \quad 1 - M \leq j \leq \calR(T + \delta_0) \quad \text{and} \quad t \in [0, T + \delta_0];
	\end{equation*}
	the second condition holds for $\calR(T) < j \leq \calR(T + \delta_0)$ since $(\bx_j, \by_j) = (0, 0) \in C$ for a task $j$ that has just arrived to the system. Now define:
	\begin{equation*}
		\varepsilon \defeq \bar F (\eta_0) \wedge \min_{(x, y) \in C_0} S_g(x, y).
	\end{equation*}
	
	Observe that $\varepsilon$ is well-defined and positive by Lemma \ref{lem: bound for H} and the fact that $\eta_0 < \kappa$. Let $L_1 \geq 0$ and $\delta \in (0, \delta_0]$ be as in Lemma \ref{lem: bound for H} for the set $C_0$. Then \eqref{eq: definition of H} implies that:
	\begin{equation*}
		H_j^1(s, t) \leq \frac{L_1 (t - s)}{\varepsilon} \quad \text{if} \quad 1 - M \leq j \leq \calR(s) \quad \text{and} \quad t - \delta \leq s < t \leq T + \delta
	\end{equation*}
	By Assumption \ref{ass: standing assumptions 2}, there exists $L_2 \geq 0$ such that $f(t) \leq L_2$ for all $t \in [0, \eta_0 + \delta]$. Thus,
	\begin{equation*}
		\frac{\bar F\left(\bw(s)\right) - \bar F\left(\bw(s) + t - s\right)}{\bar F\left(\bw(s)\right)} \leq \frac{L_2(t - s)}{\varepsilon} \quad \text{if} \quad t - \delta \leq s < t \leq T + \delta. 
	\end{equation*}
	Combining the above inequalities, we conclude that:
	\begin{equation*}
		\frac{\bar F\left(\bw(s)\right) - \bar F\left(\bw(s) + t - s\right)}{\bar F\left(\bw(s)\right)} + \sum_{j = 1 - M}^{\calR(s)} H_j^1(s, t) \leq \frac{L_1 + L_2}{\varepsilon}\left[1 + M + \calR(s)\right](t - s)
	\end{equation*}
	if $t - \delta \leq s < t \leq T + \delta$, which completes the proof.
\end{proof}

The following lemma is a useful property of the set $U_g$ introduced in Definition \ref{def: hazard rate}. We provide the proof of this lemma in Appendix \ref{app: proofs of auxiliary results}.

\begin{restatable}{lemma}{lemmafourfive}
	\label{lem: approximating U_g by rectangles}
	A set $R \subset \R_+^2$ is called a closed rectangle if
	\begin{equation*}
		R = \asetc{(x, y) \in \R_+^2}{x \leq x_0\ \text{and}\ y \leq y_0} \quad \text{for some} \quad (x_0, y_0) \in \R_+^2,
	\end{equation*}
	and $R$ is an open rectangle when the above inequalities are strict. For each $\varepsilon > 0$, there exists a set $A \subset U_g$ such that $A$ is a finite union of closed rectangles and
	\begin{equation*}
		\int_A g(x, y)dxdy \geq 1 - \varepsilon.
	\end{equation*}
	Moreover, the same is true for open rectangles.
\end{restatable}

The last lemma that we need to obtain the compensator of $\bD(\varphi)$ constructs a sequence of stopping times with suitable properties that increase to a given stopping time. It will be used to replace the arbitrary stopping time in \eqref{eq: motivation for compensator} by one with convenient properties. In particular, the approximating stopping times are constructed to satisfy \eqref{eq: conditions for bound for B_i^c(s, t)}.

\begin{lemma}
	\label{lem: increasing sequence of stopping times}
	Given a stopping time $T$, there exist stopping times $\setc{T_k}{k \geq 0}$ such that $T_k \uparrow T$ as $k \to \infty$ almost surely and the following properties hold surely for each $k \geq 1$.
	\begin{enumerate}
		\item[(a)] There exists a constant $K \geq 0$ such that $T_k \leq K$.
		
		\item[(b)] There exists $\eta \in (0, \kappa)$ such that $a_i \wedge T_k - a_{i - 1} \wedge T_k \leq \eta$ for all $i \geq 1$.
		
		\item[(c)] There exists a compact set $(0, 0) \in C \subset U_g$ such that
		\begin{equation*}
			\left(\bx_i(t), \by_i(t)\right) \in C \quad \text{for all} \quad 1 - M \leq i \leq \calR\left(T_k\right) \quad \text{and} \quad t \in [0, T_k].
		\end{equation*}
	\end{enumerate}
\end{lemma}

\begin{proof}
	By Lemma \ref{lem: approximating U_g by rectangles}, there exists an increasing sequence of sets $\setc{U_k \subset U_g}{k \geq 1}$ such that $U_k$ is a finite union of open rectangles for each $k \geq 1$ and
	\begin{equation*}
		\sum_{k = 1}^\infty \int_{U_k^c} g(x, y)dxdy < \infty.
	\end{equation*}
	If $C_k$ is the closure of $U_k$, then $C_k$ is compact since it is a finite union of closed rectangles. Consider also any increasing sequence $\setc{\kappa_k \in (0, \kappa)}{k \geq 1}$ satisfying that:
	\begin{equation*}
		\sum_{k = 1}^\infty \bar F\left(\kappa_k\right) < \infty.
	\end{equation*}
	
	We now define two sequences of random times as follows:
	\begin{equation*}
		R_k \defeq \inf \asetc{t \geq 0}{\bmu_t\left(U_k^c\right) > 0} \quad \text{and} \quad S_k \defeq \inf \asetc{t \geq 0}{\max_{i \geq 1} \left\{a_i \wedge t - a_{i - 1} \wedge t\right\} > \kappa_k}.
	\end{equation*}
	 The processes $\bx_i$ and $\by_i$ are continuous and the departure times are stopping times. Hence, it follows that $t \mapsto \bmu_t(U_k^c)$ is adapted. It is also c\`adl\`ag, because $(\bx_i(s), \by_i(s)) \notin U_k$ implies that $(\bx_i(t), \by_i(t)) \notin U_k$ for all $t \geq s$ since $U_k$ is a union of rectangles and both $\bx_i$ and $\by_i$ are nondecreasing. Therefore, $R_k$ is a stopping time by \cite[Proposition 2.1.5]{ethier2009markov}. The same proposition implies that $S_k$ is a stopping time as well, as it is the first entrance time to an open set for a continuous and adapted process. We conclude that the random time
	\begin{equation*}
		T_k \defeq k \wedge R_k \wedge S_k \wedge T
	\end{equation*}
	is a stopping time for all $k \geq 1$. Moreover, it is clear that $T_k \leq T_{k + 1}$ surely for each $k \geq 1$, and that each $T_k$ satisfies properties (a), (b) and (c), with $K = k$, $\eta = \kappa_k$ and $C = C_k$.
	
	It remains to prove that $T_k \to T$ almost surely as $k \to \infty$. For this purpose, note that
	\begin{align*}
		\cprob*{R_k \leq t} &= \sum_{n = 0}^\infty \cprob*{R_k \leq t | M + \calR(t) = n} \cprob*{M + \calR(t) = n} \\
		&\leq \sum_{n = 0}^\infty n \left[\int_{U_k^c} g(x, y)dxdy\right] \cprob*{M + \calR(t) = n} \leq \expect*{M + \calR(t)}\int_{U_k^c} g(x, y)dxdy,
	\end{align*}
	where the latter expectation is finite by Assumption \ref{ass: standing assumptions 1}. Because the stopping times $R_k$ are increasing and the sum over $k \geq 1$ of the right-hand side is finite, we conclude from the Borel-Cantelli lemma that the probability of $\setc{R_k}{k \geq 1} \subset [0, t]$ is zero for each $t \geq 0$, and therefore $R_k \to \infty$ as $k \to \infty$ with probability one. On the other hand,
	\begin{align*}
		\cprob*{S_k \leq t} &= \sum_{n = 0}^\infty \cprob*{S_k \leq t | \calR(t) = n} \cprob*{\calR(t) = n} \\
		&\leq \sum_{n = 0}^\infty \left(n + 1\right) \bar F\left(\kappa_k\right) \cprob*{\calR(t) = n} \leq \expect*{\calR(t) + 1}\bar F\left(\kappa_k\right).
	\end{align*}
	Using similar arguments as above, we conclude that $S_k \to \infty$ as $k \to \infty$ almost surely, and then it follows that $T_k \to T$ as $k \to \infty$ with probability one.
\end{proof}

We are now ready to derive the compensator of the process $\bD(\varphi)$.

\begin{theorem}
	\label{the: compensator of departures-driven process}
	For each $\varphi \in C_b(\R_+^3)$, the process $\bD(\varphi)$ has integrable total variation over finite intervals of time. Further, its compensator is the process $\bD^c(\varphi)$ defined as:
	\begin{equation*}
		\bD_t^c(\varphi) \defeq \int_0^t \angled{\varphi_s\left[r(\bmu_s)h_g^x + h_g^y\right]}{\bmu_s}ds \quad \text{for all} \quad t \geq 0.
	\end{equation*}
\end{theorem}

\begin{proof}
	The total variation of $\bD(\varphi)$ over the interval $[0, t]$ is
	\begin{equation*}
		TV\left[\bD(\varphi)\right](t) = \sum_{i = 1 - M}^{\calR(t)} \ind{d_i \leq t} \left|\varphi_{d_i}\left(\bx_i(d_i), \by_i(d_i)\right)\right| \leq \left[M + \calR(t)\right] \norm{\varphi}_\infty.
	\end{equation*}
	The right-hand side has finite mean by Assumption \ref{ass: standing assumptions 1}, proving the first claim.
	
	In order to establish the second claim, let us consider the functions $\varphi^+ \defeq \max\{\varphi, 0\}$ and $\varphi^- \defeq \max\{-\varphi, 0\}$, which are nonnegative, continuous and bounded. Note that:
	\begin{equation*}
		\bD(\varphi) = \bD\left(\varphi^+\right) - \bD\left(\varphi^-\right) \quad \text{and} \quad \bD^c\left(\varphi\right) = \bD^c\left(\varphi^+\right) - \bD^c\left(\varphi^-\right),
	\end{equation*}
	As a result, we may assume without any loss of generality that $\varphi$ is nonnegative, and thus $\bD(\varphi)$ and $\bD^c(\varphi)$ are nondecreasing. Then it follows from \cite[Theorem I.3.17]{jacod2013limit} that $\bD^c(\varphi)$ is the compensator of $\bD(\varphi)$ if and only if $\bD^c(\varphi)$ is locally integrable, predictable and
	\begin{equation}
		\label{eq: condition for martingale property of compensator}
		\expect*{\bD_T(\varphi)} = \expect*{\bD_T^c(\varphi)} \quad \text{for all stopping times}\ T.
	\end{equation}
	Note that \eqref{eq: condition for martingale property of compensator} yields $E[\bD_t^c(\varphi)] = E[\bD_t(\varphi)] \leq E[M + \calR(t)]\norm{\varphi}_\infty < \infty$ for all $t \geq 0$. Also, $\bD^c(\varphi)$ is predictable since it is adapted and continuous, so it only remains to prove \eqref{eq: condition for martingale property of compensator}.
	
	Suppose then that $\varphi$ is nonnegative and fix an arbitrary stopping time $T$. In order to establish \eqref{eq: condition for martingale property of compensator}, we will compute $E[\bD_T(\varphi)]$. For this purpose, observe that
	\begin{equation*}
		\expect*{\bD_T(\varphi)} = \expect*{\sum_{i = 1 - M}^\infty \expect*{\bD_T^i(\varphi) | M}}, \quad \text{where} \quad \bD_t^i(\varphi) \defeq \ind{d_i \leq t}\varphi_{d_i}\left(\bx_i(d_i), \by_i(d_i)\right),
	\end{equation*}
	by Tonelli's theorem. Therefore, it suffices to fix $i$ and compute $E[\bD_T^i(\varphi) | M]$. Moreover, if $\setc{T_k}{k \geq 1}$ is a sequence of stopping times such that $T_k \uparrow T$ as $k \to \infty$ with probability one, then the monotone convergence theorem implies that
	\begin{equation*}
		\lim_{k \to \infty} \expect*{\bD_{T_k}(\varphi) | M} = \expect*{\bD_T(\varphi) | M} \quad \text{and} \quad \lim_{k \to \infty} \expect*{\bD_{T_k}^c(\varphi) | M} = \expect*{\bD_T^c(\varphi) | M}.
	\end{equation*}
	As a result, it follows from Lemma \ref{lem: increasing sequence of stopping times} that we may assume without any loss of generality that the stopping time $T$ satisfies the following properties surely:
	\begin{enumerate}
		\item[(a)] There exists $K \geq 0$ such that $T \leq K$.
		
		\item[(b)] There exists $\eta \in (0, \kappa)$ such that $a_j \wedge T - a_{j - 1} \wedge T \leq \eta$ for all $j \geq 1$.
		
		\item[(c)] There exists a compact set $(0, 0) \in C \subset U_g$ such that
		\begin{equation*}
			\left(\bx_j(t), \by_j(t)\right) \in C \quad \text{for all} \quad 1 - M \leq j \leq \calR\left(T\right) \quad \text{and} \quad t \in [0, T].
		\end{equation*}
	\end{enumerate}
	
	We proceed to compute $E[\bD_T^i(\varphi)]$ for a stopping time $T$ that satisfies the above three properties. For this purpose, we will partition $\R_+$ into subintervals of length $1 / n$ and we will make use of the following definitions and notation:
	\begin{align*}
		&\calF_n^m \defeq \calF_{\frac{m}{n}}, \quad r_n^m \defeq r\left(\bmu_{\frac{m}{n}}, \by_i\left(\frac{m}{n}\right)\right), \quad \varphi_n^m \defeq \varphi_{\frac{m}{n}}\left(\bx_i\left(\frac{m}{n}\right), \by_i\left(\frac{m}{n}\right)\right), \\
		&\text{and} \quad \bX_t^n \defeq \sum_{m = 0}^\infty \ind{a_i \leq \frac{m}{n} < d_i \leq \frac{m + 1}{n}, \frac{m}{n} < t} \varphi_n^m.
	\end{align*}
	Observe that $\bX_t^n \to \bD_t^i(\varphi)$ for all $t \geq 0$ as $n \to \infty$ surely. Also, $\bX_T^n \leq \norm{\varphi}_\infty$, so it follows from the bounded convergence theorem and Tonelli's theorem that:
	\begin{align*}
		\expect*{\bD_T^i(\varphi) | M} &= \lim_{n \to \infty} \expect*{\bX_T^n | M} \\
		&= \lim_{n \to \infty} \sum_{m = 0}^\infty \expect*{\ind{a_i \leq \frac{m}{n} < d_i \leq \frac{m + 1}{n}, \frac{m}{n} < T} \varphi_n^m | M} \\
		&= \lim_{n \to \infty} \sum_{m = 0}^\infty \expect*{\ind{\frac{m}{n} < T} \varphi_n^m \cprob*{a_i \leq \frac{m}{n} < d_i \leq \frac{m + 1}{n} | \calF_n^m} | M} \\
		&= \lim_{n \to \infty} \expect*{\sum_{m = 0}^\infty \ind{\frac{m}{n} < T} \varphi_n^m \cprob*{a_i \leq \frac{m}{n} < d_i \leq \frac{m + 1}{n} | \calF_n^m} | M}.
	\end{align*}
	
	Let us introduce some additional notation:
	\begin{equation*}
		I_n^m(\alpha) \defeq I_i^\alpha\left(\frac{m}{n}, \frac{m + 1}{n}\right), \quad J_n^m \defeq J_i\left(\frac{m}{n}, \frac{m + 1}{n}\right) \quad \text{and} \quad \calG_n^m \defeq \calG_i\left(\frac{m}{n}, \frac{m + 1}{n}\right),
	\end{equation*}
	where the right-hand sides are defined as in Lemma \ref{lem: hazard rate terms}. Observe that
	\begin{align*}
		\cprob*{a_i \leq \frac{m}{n} < d_i \leq \frac{m + 1}{n} | \calF_n^m} &= \sum_{\alpha \in \{0, 1\}} \expect*{I_n^m(\alpha)J_n^m\ind{d_i \leq \frac{m + 1}{n}} | \calF_n^m} \\
		&+ \sum_{\alpha \in \{0, 1\}} \expect*{I_n^m(\alpha)\left(1 - J_n^m\right)\ind{d_i \leq \frac{m + 1}{n}} | \calF_n^m} \\
		&= \sum_{\alpha \in \{0, 1\}} \expect*{I_n^m(\alpha)J_n^m\cprob*{d_i \leq \frac{m + 1}{n} | \calG_n^m} | \calF_n^m} \\
		&+ \sum_{\alpha \in \{0, 1\}} \expect*{I_n^m(\alpha)\left(1 - J_n^m\right)\ind{d_i \leq \frac{m + 1}{n}} | \calF_n^m}
	\end{align*}
	almost surely; the second equality follows by conditioning with respect to $\calF_n^m \vee \calG_n^m$, the $\sigma$-algebra generated by $\calF_n^m$ and $\calG_n^m$. Moreover, it follows from Lemma \ref{lem: hazard rate terms} that
	\begin{align*}
		\sum_{\alpha \in \{0, 1\}} \expect*{I_n^m(\alpha)J_n^m\cprob*{d_i \leq \frac{m + 1}{n} | \calG_n^m} | \calF_n^m} = \sum_{\alpha \in \{0, 1\}} I_n^m(\alpha)\expect*{J_n^m | \calF_n^m}H_n^m(\alpha) \\
		= \sum_{\alpha \in \{0, 1\}} \ind{a_i \leq \frac{m}{n} < d_i, r_n^m = \alpha}\expect*{J_n^m | \calF_n^m}H_n^m(\alpha),
	\end{align*}
	where $H_n^m(\alpha) \defeq H_i^\alpha(m / n, (m + 1) / n)$, using the notation of Lemma \ref{lem: hazard rate terms}. Further,
	\begin{align*}
		Y_n^m &\defeq \sum_{\alpha \in \{0, 1\}} \expect*{I_n^m(\alpha)\left(1 - J_n^m\right)\ind{d_i \leq \frac{m + 1}{n}} | \calF_n^m} \\
		&= \sum_{\alpha \in \{0, 1\}} \expect*{\expect*{I_n^m(\alpha)\left(1 - J_n^m\right)\ind{d_i \leq \frac{m + 1}{n}} | \calF_n^m \vee \calG_n^m} | \calF_n^m} \\
		&\leq \sum_{\alpha \in \{0, 1\}} \expect*{\expect*{I_n^m(\alpha)\left(1 - J_n^m\right)\ind{\bx_i\left(\frac{m}{n}\right) + \frac{1}{n} \geq b_i\ \text{or}\ \by_i\left(\frac{m}{n}\right) + \frac{1}{n} \geq c_i} | \calF_n^m \vee \calG_n^m} | \calF_n^m} \\
		&= \sum_{\alpha \in \{0, 1\}} \expect*{I_n^m(\alpha)\left(1 - J_n^m\right)\cprob*{\bx_i\left(\frac{m}{n}\right) + \frac{1}{n} \geq b_i\ \text{or}\ \by_i\left(\frac{m}{n}\right) + \frac{1}{n} \geq c_i | \calG_n^m} | \calF_n^m} \\
		&= \sum_{\alpha \in \{0, 1\}} \ind{a_i \leq \frac{m}{n} < d_i, r_n^m = \alpha} \expect*{1 - J_n^m | \calF_n^m} H_n^m(1) \\
		&\leq \sum_{\alpha \in \{0, 1\}} H_n^m(1) \expect*{1 - J_n^m | \calF_n^m} = 2H_n^m(1) \expect*{1 - J_n^m | \calF_n^m}
	\end{align*}
	almost surely. The bound in the third line corresponds to the situation where task $i$ is served at unit rate in $(m / n, (m + 1) / n)$. Also, the factor $H_n^m(1)$ in the fifth line is obtained as in Lemma \ref{lem: hazard rate terms}, noting that $\bx_i(m / n) < b_i$ and $\by_i(m / n) < c_i$ provided that $d_i \geq m / n$.
	
	Consider the processes defined as
	\begin{align*}
		&\bY_t^n \defeq \sum_{m = 0}^\infty \ind{\frac{m}{n} < T} nY_n^m\varphi_n^m\ind{\frac{m}{n} \leq t < \frac{m + 1}{n}}, \\
		&\bZ_t^n \defeq \sum_{m = 0}^\infty \sum_{\alpha \in \{0, 1\}} \ind{a_i \leq \frac{m}{n} < d_i \wedge T, r_n^m = \alpha}\expect*{J_n^m | \calF_n^m}nH_n^m(\alpha)\varphi_n^m\ind{\frac{m}{n} \leq t < \frac{m + 1}{n}},  
	\end{align*}
	for all $t \geq 0$. The right-hand sides are zero if $t > K + 1$ by (a). Hence,
	\begin{align*}
		\expect*{\bD_T^i(\varphi) | M} &= \lim_{n \to \infty} \expect*{\int_0^\infty\left[\bY_t^n + \bZ_t^n\right]dt | M} \\
		&= \lim_{n \to \infty} \expect*{\int_0^{K + 1} \left[\bY_t^n + \bZ_t^n\right]dt | M}. 
	\end{align*}
	For the first equality, note that the integral over $[0, \infty)$ of $n \ind{m / n \leq t < (m + 1) / n}$ is one.
	
	By (a), (b), (c) and Lemma \ref{lem: no but possibly one event in small interval}, there exists $L_1 \geq 0$ and $\delta_1 > 0$ such that
	\begin{equation}
		\label{eq: bound for 1 - J_n^m}
		\ind{\frac{m}{n} < T}\expect*{1 - J_n^m | \calF_n^m} \leq \ind{\frac{m}{n} < T}\frac{L_1}{n}\left[1 + M + \calR(K + 1)\right] \quad \text{for all} \quad n \geq \frac{1}{\delta_1}
	\end{equation}
	with probability one. Then, it follows that, for all $t \geq 0$ and $n \geq 1 / \delta_1$, 
	\begin{align*}
		\bY_t^n &\leq 2\sum_{m = 0}^\infty \ind{\frac{m}{n} < T}nH_n^m(1)\expect*{1 - J_n^m | \calF_n^m}\varphi_n^m\ind{\frac{m}{n} \leq t < \frac{m + 1}{n}} \\
		&\leq 2L_1\left[1 + M + \calR(K + 1)\right]\norm{\varphi}_\infty \sum_{m = 0}^\infty \ind{\frac{m}{n} < T} H_n^m(1)\ind{\frac{m}{n} \leq t < \frac{m + 1}{n}}
	\end{align*}
	almost surely. Further, let $\varepsilon > 0$ be the minimum of $S_g$ over the compact set $C$ and $L_2 \geq 0$ and $\delta_2 > 0$ be the constants in Lemma \ref{lem: bound for H}. Then \eqref{eq: definition of H}, (c) and Lemma \ref{lem: bound for H} yield:
	\begin{equation}
		\label{eq: bound for H_m^n}
		\ind{\frac{m}{n} < T}H_n^m(\alpha) \leq \ind{\frac{m}{n} < T}\frac{L_2}{\varepsilon n} \quad \text{for all} \quad n \geq \frac{1}{\delta_2} \quad \text{and} \quad \alpha \in \{0, 1\}. 
	\end{equation}
	Thus, $\bY^n$ is bounded by $2 L_1 L_2 \left[1 + M + \calR(K + 1)\right] \norm{\varphi}_\infty / \varepsilon$ for all $n \geq \max\{1 / \delta_1, 1 / \delta_2\}$, and converges pointwise to zero almost surely. By the bounded convergence theorem,
	\begin{equation*}
		\lim_{n \to \infty} \expect*{\int_0^{K + 1} \bY_t^ndt | M} = 0.
	\end{equation*}
	
	On the other hand, Lemma \ref{lem: hazard rate terms} and \eqref{eq: bound for 1 - J_n^m} imply that $\bZ_t^n$ converges as $n \to \infty$ to
	\begin{equation*}
		\bZ_t \defeq \varphi_t(\bx_i(t), \by_i(t)) \left[r\left(\bmu_t, \by_i(t)\right) h_g^x\left(\bx_i(t), \by_i(t)\right) + h_g^y\left(\bx_i(t), \by_i(t)\right)\right]
	\end{equation*}
	for all $t \in (a_i \wedge T, d_i \wedge T)$ almost surely, and $\bZ_t^n \to 0$ as $n \to \infty$ for all $t \notin [a_i \wedge T, d_i \wedge T]$, almost surely. Furthermore, $\bZ_t^n \leq L_2 \norm{\varphi}_\infty / \varepsilon$ if $n \geq 1 / \delta_2$ and $t \geq 0$ surely by \eqref{eq: bound for H_m^n}. Thus, the bounded convergence theorem yields
	\begin{equation*}
		\expect*{\bD_T^i(\varphi) | M} = \lim_{n \to \infty} \expect*{\int_0^{K + 1} \bZ_t^ndt | M} = \expect*{\int_{a_i \wedge T}^{d_i \wedge T} \bZ_tdt | M}.
	\end{equation*}
	
	Finally, it follows from the above equality that
	\begin{equation*}
		\expect*{\bD_T(\varphi)} = \expect*{\sum_{i = 1 - M}^\infty \expect*{\bD_T^i(\varphi) | M}} = \expect*{\sum_{i = 1 - M}^\infty \int_{a_i \wedge T}^{d_i \wedge T} \bZ_tdt} = \expect*{\bD_T^c(\varphi)},
	\end{equation*}
	which establishes \eqref{eq: condition for martingale property of compensator} and thereby completes the proof.
\end{proof}

\subsubsection{Predictable quadratic variation}
\label{subsub: predictable quadratic variation}

As mentioned in Section \ref{sub: outline of the proof}, bounds for the predictable quadratic variation of the local martingales $\bD(\varphi) - \bD^c(\varphi)$ will be used to characterize the limits of the departure processes as the number of servers approaches infinity. Specifically, we will show that the departure process $\bD(\varphi)$ can be replaced by its compensator in the limit and after normalization by the number of servers. The predictable quadratic variation is computed in the following proposition. Similar computations are carried out in \cite[Proposition II.2.29]{jacod2013limit}.

\begin{proposition}
	\label{prop: quadratic variations}
	The process $\bD(\varphi) - \bD^c(\varphi)$ is a locally square-integrable martingale for each $\varphi \in C_b(\R_+^3)$ and its predictable quadratic variation is $\langle \bD(\varphi) - \bD^c(\varphi) \rangle = \bD^c(\varphi^2)$.
\end{proposition}

\begin{proof}
	Recall that each realization of $\bD(\varphi)$ has finite variation over finite intervals; see the beginning of the proof of Theorem \ref{the: compensator of departures-driven process}. Also, $\bD^c(\varphi)$ has the same property because
	\begin{align*}
		TV\left[\bD^c(\varphi)\right](t) &\leq \int_0^t \angled{\left|\varphi_s\right|\left[h_g^x + h_g^y\right]}{\bmu_s}ds \\
		&\leq \norm{\varphi}_\infty \sum_{i = 1 - M}^{\calR(t)} \int_{a_i \wedge t}^{d_i \wedge t}\left[h_g^x\left(\bx_i(s), \by_i(s)\right) + h_g^y\left(\bx_i(s), \by_i(s)\right)\right]ds
	\end{align*}
	for all $t \geq 0$. The right-hand side is indeed finite since $\setc{(\bx_i(s), \by_i(s))}{s \in [a_i \wedge t, d_i \wedge t]}$ is compact and contained in $U_g$ for each $i$ surely, and $h_g$ is continuous in $U_g$ by Assumption~\ref{ass: standing assumptions 2}. It follows that the local martingale $\bD(\varphi) - \bD^c(\varphi)$ has finite variation and therefore is purely discontinuous by \cite[Lemma I.4.14]{jacod2013limit}. Furthermore, it is also a locally square-integrable martingale because it has bounded jumps and $\bD_0(\varphi) - \bD_0^c(\varphi) = 0$.
	
	Therefore, for all $t \geq 0$, we have:
	\begin{align*}
		\left\langle \bD(\varphi) - \bD^c(\varphi) \right\rangle_t &= \left[\bD(\varphi) - \bD^c(\varphi)\right]_t^c \\
		&= \left(\sum_{s \leq t} \left[\Delta\left(\bD_s(\varphi) - \bD_s^c(\varphi)\right)\right]^2\right)^c = \left(\sum_{s \leq t} \left[\Delta\bD_s(\varphi)\right]^2\right)^c = \bD_t^c(\varphi^2).
	\end{align*}
	Here we are using the following notation: the angled and squared brackets in the first line refer to the predictable quadratic variation and quadratic variation, respectively, the superscript $c$ refers to the compensator of the process below it and $\Delta \bX_t \defeq \bX_t - \bX_{t^-}$ is the jump of $\bX$ at time $t$. The first equality holds by \cite[Proposition I.4.50]{jacod2013limit}, and the second one holds by \cite[Theorem I.4.52]{jacod2013limit} since $\bD(\varphi) - \bD^c(\varphi)$ is a purely discontinuous local martingale. The third equality follows from the continuity $\bD^c(\varphi)$. Finally,
	\begin{equation*}
		\sum_{s \leq t} \left[\Delta\bD_s(\varphi)\right]^2 = \bD_t(\varphi^2) \quad \text{for all} \quad t \geq 0
	\end{equation*}
	surely by definition of $\bD(\varphi)$, which proves the last equality.
\end{proof}

For the last result of this section, we reintroduce the index $N$ into the notation. We prove that the normalized processes $\bar\bD^N(\varphi) - \bar\bD^{N, c}(\varphi)$ converge weakly to zero as $N \to \infty$. Hence, the departure processes can be replaced by their compensators in the limit.

\begin{corollary}
	\label{cor: limit of departures martingale}
	Fix $\varphi \in C_b(\R_+^3)$ and consider the normalized processes
	\begin{equation*}
		\bar\bD^N(\varphi) \defeq \frac{\bD^N(\varphi)}{N}, \quad \bar\bD^{N, c}(\varphi) \defeq \frac{\bD^{N, c}(\varphi)}{N} \quad \text{and} \quad \bX^N(\varphi) \defeq \bar\bD^N(\varphi) - \bar\bD^{N, c}(\varphi).
	\end{equation*}
	Then $\bX^N(\varphi) \Rightarrow 0$ in $D_\R[0, \infty)$ as $N \to \infty$. 
\end{corollary}

\begin{proof}
	The process $\bX^N(\varphi)$ is a locally square-integrable martingale by Proposition \ref{prop: quadratic variations}, so there exists a sequence of stopping times $T_k^N$ such that $T_k^N \uparrow \infty$ as $k \to \infty$ almost surely and $\bX_k^N(t) \defeq \bX_{t \wedge T_k^N}^N(\varphi)$ defines a square-integrable martingale. Then
	\begin{align*}
		\cprob*{\sup_{t \in [0, T]} \left|\bX_k^N(t)\right| \geq \varepsilon} &\leq \frac{1}{\varepsilon^2}\expect*{\sup_{t \in [0, T]} \left|\bX_k^N(t)\right|^2} \\
		&\leq \frac{4}{\varepsilon^2}\expect*{\left|\bX_k^N(T)\right|^2} = \frac{4}{\varepsilon^2}\expect*{\langle\bX_k^N\rangle_T} \quad \text{if} \quad \varepsilon > 0 \quad \text{and} \quad T \geq 0,
	\end{align*}
	by Markov's and Doob's inequalities. By the monotone convergence theorem,
	\begin{align*}
		\cprob*{\sup_{t \in [0, T]} \left|\bX_t^N(\varphi)\right| \geq \varepsilon} &= \lim_{k \to \infty} \cprob*{\sup_{t \in [0, T]} \left|\bX_k^N(t)\right| \geq \varepsilon} \\
		&\leq \lim_{k \to \infty} \frac{4}{\varepsilon^2}\expect*{\langle\bX_k^N\rangle_T} = \frac{4}{\varepsilon^2}\expect*{\langle\bX^N(\varphi)\rangle_T}.
	\end{align*}
	
	On the other hand, it follows from Proposition \ref{prop: quadratic variations} that
	\begin{align*}
		\expect*{\langle\bX^N(\varphi)\rangle_T} = \frac{1}{N^2} \expect*{\bD_T^{N, c}\left(\varphi^2\right)} = \frac{1}{N^2} \expect*{\bD_T^N\left(\varphi^2\right)} &\leq \frac{\norm{\varphi}_\infty^2}{N^2} \expect*{\bD_T^N\left(\indc\right)} \\
		&\leq \frac{\norm{\varphi}_\infty^2}{N^2} \expect*{M^N + \calR^N(T)},
	\end{align*}
	where $\indc$ is the function that is identically equal to one. Here the second equality follows from Theorem \ref{the: compensator of departures-driven process} and the last inequality follows from the fact that $\bD_T^N(\indc)$ is the number of departures in the interval $[0, T]$. We conclude that
	\begin{equation*}
		\lim_{N \to \infty} \cprob*{\sup_{t \in [0, T]} \left|\bX_t^N(\varphi)\right| \geq \varepsilon} \leq \lim_{N \to \infty} \frac{4\norm{\varphi}_\infty^2}{\varepsilon^2 N} \frac{\expect*{M^N + \calR^N(T)}}{N} = 0
	\end{equation*}
	by Assumption \ref{ass: standing assumptions 1} and the law of large numbers for renewal processes. This implies that $\bX^N(\varphi) \Rightarrow 0$ in $D_\R[0, T]$ for all $T \geq 0$, and therefore the limit also holds in $D_\R[0, \infty)$.
\end{proof}

\subsection{Tightness results}
\label{sub: tightness results}

Recall our notation for normalized processes:
\begin{equation*}
\bar\bA_t^N(\varphi) \defeq \frac{\bA_t^N(\varphi)}{N}, \quad \bar\bD_t^N(\varphi) \defeq \frac{\bD_t^N(\varphi)}{N} \quad \text{and} \quad \bar\bD_t^{N, c}(\varphi) \defeq \frac{\bD_t^{N, c}(\varphi)}{N}.
\end{equation*}
Fixing $t \geq 0$ and letting $\varphi \in C_b(\R_+^3)$ vary, we get random variables in $\calM_F^+(\R_+^3)$. Thus, letting time vary as well, we obtain measure-valued processes. These processes are c\`adl\`ag with respect to the weak topology of $\calM_F^+(\R_+^3)$, because for a fixed $\varphi \in C_b(\R_+^3)$, the above expressions define real-valued c\`adl\`ag processes as time varies.

In this section we establish the tightness of the following sequences:
\begin{equation*}
	\asetc{\bar\bmu^N}{N \geq 1}, \quad \asetc{\bar\bA^N}{N \geq 1}, \quad \asetc{\bar\bD^N}{N \geq 1} \quad \text{and} \quad \asetc{\bar\bD^{N, c}}{N \geq 1}.
\end{equation*}
The former consists of measure-valued processes in $D_{\calM_F^+(\R_+^2)}[0, \infty)$ and the latter three of measure-valued processes in $D_{\calM_F^+(\R_+^3)}[0, \infty)$. By Prohorov's theorem, the tightness of the first of these sequences implies the relative compactness claimed in Theorem \ref{the: measure-valued fluid limit}. The tightness of the other three sequences will be used together with the pre-limit stochastic equation to characterize the subsequential limits of the first sequence.

\subsubsection{Standard tightness criteria}
\label{subsub: standard tightness criteria}

In order to establish the tightness of the above sequences of measure-valued processes, we will use the following criteria; see \cite[Theorem 4.6]{jakubowski1986skorokhod} for a proof.

\begin{theorem}[Jakubowski]
	\label{the: Jakubowski's criteria}
	A sequence $\setc{\bnu^N}{N \geq 1}$ of processes in $D_{\calM_F^+(\R_+^d)}[0, \infty)$ is tight if and only if the following two conditions are satisfied.
	\begin{enumerate}
		\item[\normalfont{J1.}] For each $T > 0$ and $\varepsilon > 0$, there exists a compact set $C \subset \calM_F^+(\R_+^d)$ such that
		\begin{equation*}
			\liminf_{N \to \infty} \cprob*{\bnu_t^N \in C\ \text{for all}\ t \in [0, T]} \geq 1 - \varepsilon.
		\end{equation*}
		This is called the compact containment condition.
		
		\item[\normalfont{J2.}] There exists a family $\calH$ of continuous functions from $\calM_F^+(\R_+^d)$ into $\R$ that separates points in $\calM_F^+(\R_+^d)$, is closed under addition and satisfies that
		\begin{equation*}
			\asetc{t \mapsto h\left(\bnu_t^N\right)}{N \geq 1} \quad \text{is tight in}\ D_\R[0, \infty) \quad \text{for all} \quad h \in \calH.
		\end{equation*}
		If this property holds, then $\setc{\bnu^N}{N \geq 1}$ is called $\calH$-weakly tight.
	\end{enumerate}
\end{theorem}

Similar to the approach of Kaspi and Ramanan in \cite{kaspi2011law}, to prove condition J2, we will consider the following families of functions, given by integrals of test functions:
\begin{equation}
	\label{eq: family of functions for Jakubowski's criteria}
	\calH_2 \defeq \asetc{\mu \mapsto \angled{\phi}{\mu}}{\phi \in C_c^1(\R_+^2)} \quad \text{and} \quad \calH_3 \defeq \asetc{\mu \mapsto \angled{\varphi}{\mu}}{\varphi \in C_b(\R_+^3)}.
\end{equation}
Specifically, we will prove that the sequence of processes $\bar\bmu^N$ are $\calH_2$-weakly tight while the other three sequences of processes are $\calH_3$-weakly tight; in the latter case, observe that $\langle\varphi, \bar\bA_t^N\rangle = \bar\bA_t^N(\varphi)$, and analogously for the departure processes and their compensators. Clearly, the functions in $\calH_d$ are continuous with respect to the weak topology on $\calM_F^+(\R_+^d)$, the family $\calH_d$ is closed under addition and $\calH_d$ separates points in $\calM_F^+(\R_+^d)$.

In order to check the $\calH_d$-weak tightness conditions, we will use the following criteria; we refer to \cite[Theorem 3.8.6 and Remark 3.8.7]{ethier2009markov} for a proof.

\begin{theorem}[Kurtz]
	\label{the: Kurtz's criteria}
	Consider a sequence of processes $\setc{\bX^N}{N \geq 1}$ in $D_\R[0, \infty)$. If the following conditions are satisfied, then the sequence is tight.
	\begin{enumerate}
		\item[\normalfont{K1.}] For each rational $t \geq 0$, we have
		\begin{equation*}
			\lim_{K \to \infty} \limsup_{N \to \infty} P\left(\left|\bX_t^N\right| > K\right) = 0.
		\end{equation*}
		
		\item[\normalfont{K2.}] For each $t > 0$, there exists $\beta > 0$ such that
		\begin{equation*}
			\lim_{\delta \to 0} \limsup_{N \to \infty} E\left|\bX_{t + \delta}^N - \bX_t^N\right|^\beta = 0.
		\end{equation*}
	\end{enumerate}
\end{theorem}

For each sequence of measure-valued processes, we will use Theorem \ref{the: Kurtz's criteria} to verify the $\calH_d$-weak tightness condition, for the corresponding dimension $d$. Then we will check the compact containment condition and prove tightness by invoking Theorem \ref{the: Jakubowski's criteria}.

\subsubsection{Proofs of tightness results}
\label{subsub: tightness results results}

If we fix $\varphi \in C_b(\R_+^3)$ and $t \geq 0$, then
\begin{align*}
	&E\left|\bar\bD_{t + \delta}^N(\varphi) - \bar\bD_t^N(\varphi)\right| \leq E\left[\bar\bD_{t + \delta}^N(|\varphi|) - \bar\bD_t^N(|\varphi|)\right], \\
	&E\left|\bar\bD_{t + \delta}^{N, c}(\varphi) - \bar\bD_t^{N, c}(\varphi)\right| \leq E\left[\bar\bD_{t + \delta}^{N, c}(|\varphi|) - \bar\bD_t^{N, c}(|\varphi|)\right], 
\end{align*}
for all $\delta > 0$. Note that the right-hand sides are equal by Theorem \ref{the: compensator of departures-driven process}. Thus, to prove K2 for the departure process and its compensator, it suffices to show that the right-hand side of the last inequality vanishes as $\delta \to 0$.
This is not trivial since our expression for the compensator involves an integral where the integrand includes the possibly unbounded hazard rate vector field. The next lemma will be used to address this complication.

\begin{lemma}
	\label{lem: bounds for integrals of hazard rate}
	For each $\varepsilon > 0$, there exist $K_\varepsilon \geq 0$ and compact sets $A_\varepsilon \subset B_\varepsilon \subset U_g$ such that $A_\varepsilon$ is inside the interior in $\R_+^2$ of $B_\varepsilon$ and the following properties hold:
	\begin{equation*}
		\int_{A_\varepsilon} g(x, y)dxdy \geq 1 - \varepsilon \quad \text{and} \quad h_g^x(x, y) + h_g^y(x, y) \leq K_\varepsilon \quad \text{for all} \quad (x, y) \in B_\varepsilon.
	\end{equation*}
	Suppose that $\psi$ is continuous and $\ind{(x, y) \in A_\varepsilon} \leq \psi(x, y) \leq \ind{(x, y) \in B_\varepsilon}$ for all $x, y \in \R_+$, then
	\begin{equation*}
		\max\left\{E\left|\bar\bD_t^N((1 - \psi)\varphi)\right|, E\left|\bar\bD_t^{N, c}((1 - \psi)\varphi)\right|\right\} \leq \varepsilon \left(\bar M + \frac{\expect*{\calR^N(t)}}{N}\right) \norm{\varphi}_\infty.
	\end{equation*}
	for all $t \geq 0$ and $\varphi \in C_b(\R_+^3)$. This inequality in fact holds if $0 \leq \psi \leq 1$ and $\psi = 1$ in~$A_\varepsilon$. In particular, the condition involving $B_\varepsilon$ is not needed to prove the inequality, but will be used when we invoke the lemma in the sequel.
\end{lemma}

\begin{proof}
	By Lemma \ref{lem: approximating U_g by rectangles}, there exist closed rectangles $\setc{R_\varepsilon^i \subset U_g}{i = 1, \dots, k_\varepsilon}$ such that
	\begin{equation*}
		\int_{A_\varepsilon} g(x, y)dxdy \geq 1 - \varepsilon \quad \text{with} \quad A_\varepsilon \defeq \bigcup_{i = 1}^{k_\varepsilon} R_\varepsilon^i.
	\end{equation*}
	Considering slightly larger rectangles, we obtain another compact set $B_\varepsilon \subset U_g$ such that its interior contains $A_\varepsilon$. Because $B_\varepsilon \subset U_g$ is compact, it follows from Assumption \ref{ass: standing assumptions 2} that there exists a constant $K_\varepsilon \geq 0$ such that $h_g^x(x, y) + h_g^y(x, y) \leq K_\varepsilon$ for all $(x, y) \in B_\varepsilon$.
	
	For the second part of the lemma, we first note that:
	\begin{equation*}
		E\left|\bar\bD_t^N\left((1 - \psi)\varphi\right)\right| \leq E\left[\bar\bD_t^N\left((1 - \psi)|\varphi|\right)\right] = E\left[\bar\bD_t^{N, c}\left((1 - \psi)|\varphi|\right)\right] \geq E\left|\bar\bD_t^{N, c}\left((1 - \psi)\varphi\right)\right|
	\end{equation*}
	by Theorem \ref{the: compensator of departures-driven process}. Therefore, it suffices to bound the second expression from the left.
	
	Recall that $A_\varepsilon$ is a finite union of rectangles and $\bx_i^N(d_i^N) \leq b_i^N$ and $\by_i^N(d_i^N) \leq c_i^N$, so $(\bx_i^N(d_i^N), \by_i^N(d_i^N)) \notin A_\varepsilon$ implies that $(b_i^N, c_i^N) \notin A_\varepsilon$. It follows that
	\begin{align*}
		E\left[\bar\bD_t^N\left((1 - \psi)|\varphi|\right)\right] &\leq E\left[\sum_{i = 1 - M^N}^{\calR^N(t)} \ind{t \geq d_i^N, \left(\bx_i^N\left(d_i^N\right), \by_i^N\left(d_i^N\right)\right) \notin A_\varepsilon} \frac{\left|\varphi_{d_i^N}\left(\bx_i^N\left(d_i^N\right), \by_i^N\left(d_i^N\right)\right)\right|}{N}\right] \\
		&\leq \frac{1}{N}E\left[\sum_{i = 1 - M^N}^{\calR^N(t)} \ind{\left(b_i^N, c_i^N\right) \notin A_\varepsilon}\right] \norm{\varphi}_\infty \\
		&\leq \frac{1}{N}\expect*{\varepsilon \left[M^N + \calR^N(t)\right]} \norm{\varphi}_\infty \leq \varepsilon \left(\bar M + \frac{\expect*{\calR^N(t)}}{N}\right) \norm{\varphi}_\infty.
	\end{align*}
	For the third inequality, note that there are at most $M^N + \calR^N(t)$ tasks present at time~$t$, and that independently of $M^N$ and $\calR^N(t)$, for each task $i$, we have $(b_i^N, c_i^N) \notin A_\varepsilon$ with probability less than $\varepsilon$. The last inequality holds by Assumption \ref{ass: standing assumptions 1}
\end{proof}

The following lemma proves K2 for the arrival and departure processes associated with a function $\varphi \in C_b(\R_+^3)$, and for the compensator of this departure process.

\begin{lemma}
	\label{lem: proof of condition K2}
	For each $t \geq 0$ and $\varphi \in C_b(\R_+^3)$, we have:
	\begin{subequations}
		\begin{align}
			&\lim_{\delta \to 0} \limsup_{N \to \infty} E\left|\bar\bA_{t + \delta}^N(\varphi) - \bar\bA_t^N(\varphi)\right| = 0, \label{seq: short interval increment for arrivals} \\
			&\lim_{\delta \to 0} \limsup_{N \to \infty} E\left|\bar\bD_{t + \delta}^N(\varphi) - \bar\bD_t^N(\varphi)\right| = 0, \label{seq: short interval increment for departures} \\
			&\lim_{\delta \to 0} \limsup_{N \to \infty} E\left|\bar\bD_{t + \delta}^{N, c}(\varphi) - \bar\bD_t^{N, c}(\varphi)\right| = 0. \label{seq: short interval increment for departures compensator}
		\end{align}
	\end{subequations}
\end{lemma}

\begin{proof}
	For all $t \geq 0$ and $\delta > 0$, each $\varphi \in C_b(\R_+^3)$ satisfies that:
	\begin{equation*}
		E\left|\bar\bA_{t + \delta}^N(\varphi) - \bar\bA_t^N(\varphi)\right| \leq \frac{1}{N}\expect*{\sum_{i  = \calR^N(t) + 1}^{\calR^N(t + \delta)} \left|\varphi_{a_i^N}(0, 0)\right|} \leq \frac{\expect*{\calR^N(t + \delta) - \calR^N(t)}}{N} \norm{\varphi}_\infty.
	\end{equation*}
	By Assumption \ref{ass: standing assumptions 1} and the elementary renewal theorem, the right-hand side converges to $\lambda_0\delta\norm{\varphi}_\infty$ as $N \to \infty$, and therefore \eqref{seq: short interval increment for arrivals} follows. In addition, recall that
	\begin{equation}
		\label{eq: bound for short interval increments using absolut value}
		\begin{split}
			&E\left|\bar\bD_{t + \delta}^N(\varphi) - \bar\bD_t^N(\varphi)\right| \leq E\left[\bar\bD_{t + \delta}^N(|\varphi|) - \bar\bD_t^N(|\varphi|)\right], \\
			&E\left|\bar\bD_{t + \delta}^{N, c}(\varphi) - \bar\bD_t^{N, c}(\varphi)\right| \leq E\left[\bar\bD_{t + \delta}^{N, c}(|\varphi|) - \bar\bD_t^{N, c}(|\varphi|)\right].
		\end{split}
	\end{equation}
	Further, the right-hand sides of these inequalities are equal by Theorem \ref{the: compensator of departures-driven process}. As a result, both \eqref{seq: short interval increment for departures} and \eqref{seq: short interval increment for departures compensator} can be proved by showing that if $\varphi \in C_b(\R_+^3)$ is nonnegative, then
	\begin{equation}
		\label{eq: expected number of departures in a short interval}
		\lim_{\delta \to 0} \limsup_{N \to \infty} \expect*{\bar\bD_{t + \delta}^{N, c}(\varphi) - \bar\bD_t^{N, c}(\varphi)} = 0 \quad \text{for all} \quad t \geq 0.
	\end{equation}
	
	Given $\varepsilon > 0$, fix a constant $K_\varepsilon \geq 0$, compact sets $A_\varepsilon \subset B_\varepsilon \subset U_g$ and a function $\psi$ as in the statement of Lemma \ref{lem: bounds for integrals of hazard rate}. In addition, consider the following functions:
	\begin{equation*}
		\varphi^1(t, x, y) \defeq \psi(x, y)\varphi(t, x, y) \quad \text{and} \quad \varphi^2(t, x, y) \defeq [1 - \psi(x, y)]\varphi(t, x, y)
	\end{equation*}
	for all $t, x, y \in \R_+$. Since $\varphi = \varphi^1 + \varphi^2$ and $\varphi^2$ is nonnegative, we obtain:
	\begin{align*}
		\expect*{\bar\bD_{t + \delta}^{N, c}(\varphi) - \bar\bD_t^{N, c}(\varphi)} &= \expect*{\bar\bD_{t + \delta}^{N, c}(\varphi^1) - \bar\bD_t^{N, c}(\varphi^1)} + \expect*{\bar\bD_{t + \delta}^{N, c}(\varphi^2) - \bar\bD_{t}^{N, c}(\varphi^2)} \\
		&\leq \expect*{\bar\bD_{t + \delta}^{N, c}(\varphi^1) - \bar\bD_t^{N, c}(\varphi^1)} + \expect*{\bar\bD_{t + \delta}^{N, c}(\varphi^2)}.
	\end{align*}
	
	By Lemma \ref{lem: bounds for integrals of hazard rate} and the fact that $\varphi$ is nonnegative,
	\begin{equation*}
		\expect*{\bar\bD_{t + \delta}^{N, c}(\varphi^2)} = E\left|\bar\bD_{t + \delta}^{N, c}\left((1 - \psi)\varphi\right)\right| \leq \varepsilon\left(\bar M + \frac{\expect*{\calR^N(t + \delta)}}{N}\right)\norm{\varphi}_\infty.
	\end{equation*}
	Moreover, the fact that $\varphi$ is nonnegative and Lemma \ref{lem: bounds for integrals of hazard rate} also imply that:
	\begin{align*}
		\expect*{\bar\bD_{t + \delta}^{N, c}(\varphi^1) - \bar\bD_t^{N, c}(\varphi^1)} &\leq \frac{1}{N}\expect*{\int_t^{t + \delta}\int_{B_\varepsilon} \varphi_s(x, y)\left[h_g^x(x, y) + h_g^y(x, y)\right]\bmu_s^N(dx, dy)ds} \\
		&\leq \frac{\delta}{N} \expect*{M^N + \calR^N(t + \delta)} K_\varepsilon \norm{\varphi}_\infty \\
		&\leq \delta \left(\bar M + \frac{\expect*{\calR^N(t + \delta)}}{N}\right)K_\varepsilon \norm{\varphi}_\infty.
	\end{align*}
	For the first inequality, recall that $\psi = 0$, and thus $\varphi^1 = 0$, outside of $B_\varepsilon$. For the second inequality, note that $\bmu_s^N(\R_+^2) \leq M^N + \calR^N(t + \delta)$ for all $s \leq t + \delta$. In addition, observe that the third inequality follows from Assumption \ref{ass: standing assumptions 1}.
	
	By the elementary renewal theorem, $E[\calR^N(t + \delta)] / N \to \lambda_0(t + \delta)$ as $N \to \infty$. Thus,
	\begin{align*}
		0 \leq \lim_{\delta \to 0} \limsup_{N \to \infty} \expect*{\bar\bD_{t + \delta}^{N, c}(\varphi) - \bar\bD_t^{N, c}(\varphi)} &\leq \lim_{\delta \to 0} \left(\varepsilon + \delta K_\varepsilon\right) \left[\bar M + \lambda_0(t + \delta)\right] \norm{\varphi}_\infty \\
		&= \varepsilon \left[\bar M + \lambda_0t\right] \norm{\varphi}_\infty
	\end{align*}
	for all $\varepsilon > 0$. Hence, the iterated limit equals zero, proving \eqref{eq: expected number of departures in a short interval}.
\end{proof}

Using the previous result, the following lemma establishes the $\calH_2$ and $\calH_3$-weak tightness conditions required to invoke Theorem \ref{the: Jakubowski's criteria}.

\begin{lemma}
	\label{lem: relative compactness of projections by test functions}
	For each $\varphi \in C_b(\R_+^3)$, the sequences
	\begin{equation*}
		\asetc{\bar\bA^N(\varphi)}{N \geq 1}, \quad \asetc{\bar\bD^N(\varphi)}{N \geq 1} \quad \text{and} \quad \asetc{\bar\bD^{N, c}(\varphi)}{N \geq 1}
	\end{equation*}
	are tight in $D_\R[0, \infty)$. Therefore, the corresponding measure-valued processes are $\calH_3$-weakly tight. Also, the sequence $\setc{\langle\phi, \bar\bmu^N\rangle}{N \geq 1}$ is tight in $D_\R[0, \infty)$ for each $\phi \in C_c^1(\R_+^2)$, and thus the corresponding sequence of measure-valued processes is $\calH_2$-weakly tight.
\end{lemma}

\begin{proof}
	First, fix some $\varphi \in C_b(\R_+^3)$. For each $t \geq 0$, we have
	\begin{equation*}
		E\left|\bar\bA_t^N(\varphi)\right| \leq \frac{1}{N}E\left[\sum_{i = 1}^{\calR^N(t)} \left|\varphi_{a_i^N}\left(0 , 0\right)\right|\right] \leq \frac{\expect*{\calR^N(t)}}{N} \norm{\varphi}_\infty.
	\end{equation*}
	In addition, it follows from Theorem \ref{the: compensator of departures-driven process} that
	\begin{equation*}
		E\left|\bar\bD_t^N(\varphi)\right| \leq E\left[\bar\bD_t^N(|\varphi|)\right] = E\left[\bar\bD_t^{N, c}(|\varphi|)\right] \geq E\left|\bar\bD_t^{N, c}(\varphi)\right|.
	\end{equation*}
	Hence, by bounding the second expression from the left, we obtain:
	\begin{align*}
		\max\left\{E\left|\bar\bD_t^N(\varphi)\right|, E\left|\bar\bD_t^{N, c}(\varphi)\right|\right\} &\leq \frac{1}{N}E\left[\sum_{i = 1 - M^N}^{\calR^N(t)} \left|\varphi_{d_i^N}\left(\bx_i^N\left(d_i^N\right) , \by_i^N\left(d_i^N\right)\right)\right|\right] \\
		&\leq \frac{\expect*{M^N + \calR^N(t)}}{N} \norm{\varphi}_\infty.
	\end{align*}
	
	Recall that $E[M^N] \leq \bar M N$ by Assumption \ref{ass: standing assumptions 1} and observe that $E[\calR^N(t)] / N \to \lambda_0 t$ by the elementary renewal theorem. Therefore, it follows from the above inequalities and an application of Markov's inequality that the following sequences satisfy K1:
	\begin{equation*}
		\asetc{\bar\bA^N(\varphi)}{N \geq 1}, \quad \asetc{\bar\bD^N(\varphi)}{N \geq 1} \quad \text{and} \quad \asetc{\bar\bD^{N, c}(\varphi)}{N \geq 1}
	\end{equation*}
	Moreover, K2 holds by Lemma \ref{lem: proof of condition K2}, so Theorem \ref{the: Kurtz's criteria} yields tightness in $D_\R[0, \infty)$.
	
	Now fix $\phi \in C_c^1(\R_+^2)$ and observe that
	\begin{equation*}
		E\left|\angled{\phi}{\bar\bmu_t^N}\right| \leq \frac{1}{N}E\left[\sum_{i = 1 - M^N}^{\calR^N(t)} \left|\phi\left(\bx_i^N(t), \by_i^N(t)\right)\right|\right] \leq \frac{\expect*{M^N + \calR^N(t)}}{N} \norm{\phi}_\infty.
	\end{equation*}	
	Arguing as above, we conclude that the sequence $\setc{\langle\phi, \bar\bmu^N\rangle}{N \geq 1}$ satisfies K1. In order to prove that K2 holds as well, note that Proposition \ref{prop: arrivals-flow-departures decomposition} yields:
	\begin{align*}
		\angled{\phi}{\bar\bmu_t^N} &= \angled{\phi}{\bar\bmu^N_0} + \bar\bA_t^N(\phi) - \bar\bD_t^N(\phi) + \int_0^t \angled{r^N\left(\bmu_s^N\right)\partial_x\phi + \partial_y\phi}{\bar\bmu_s^N}ds \\
		&= \angled{\phi}{\bar\bmu^N_0} + \bar\bA_t^N(\phi) - \bar\bD_t^N(\phi) + \int_0^t \angled{\bar r\left(\bar \bmu_s^N\right)\partial_x\phi + \partial_y\phi}{\bar\bmu_s^N}ds \quad \text{for all} \quad t \geq 0.
	\end{align*}
	By Lemma \ref{lem: proof of condition K2}, in order to establish that K2 is satisfied by $\setc{\langle\phi, \bar\bmu^N\rangle}{N \geq 1}$, it suffices to prove that it is satisfied by the processes defined as
	\begin{equation*}
		\bX_t^N \defeq \int_0^t \angled{\bar r\left(\bar\bmu_s^N\right)\partial_x\phi + \partial_y\phi}{\bar\bmu_s^N}ds \quad \text{for all} \quad t \geq 0.
	\end{equation*}
	
	This is indeed the case, because
	\begin{align*}
		E\left|\bX_{t + \delta}^N - \bX_t^N\right| &\leq E\left[\int_t^{t + \delta} \angled{\left|\partial_x\phi\right| + \left|\partial_y\phi\right|}{\bar\bmu_s^N}ds\right] \\
		&\leq \delta\left[\norm{\partial_x\phi}_\infty + \norm{\partial_y\phi}_\infty\right]\expect*{\sup_{s \in [t, t + \delta]} \bar\bmu_s^N\left(\R_+^2\right)} \\
		&\leq \frac{\delta}{N} \left[\norm{\partial_x\phi}_\infty + \norm{\partial_y\phi}_\infty\right]\expect*{M^N + \calR^N(t + \delta)}.
	\end{align*}
	By Assumption \ref{ass: standing assumptions 1} and the elementary renewal theorem, the limit superior as $N \to \infty$ of the last expression is bounded by $\delta [\norm{\partial_x \phi}_\infty + \norm{\partial_y\phi}_\infty][\bar M + \lambda_0(t + \delta)]$. This bound vanishes as $\delta \to 0$, so K2 holds, and hence $\setc{\langle\phi, \bar\bmu^N\rangle}{N \geq 1}$ is tight in $D_\R[0, \infty)$.
\end{proof}

We are now ready to prove the main result of this section.

\begin{theorem}
	\label{the: relative compactness of measured-valued processes}
	The sequences
	\begin{equation*}
		\asetc{\bar\bmu^N}{N \geq 1}, \quad \asetc{\bar\bA^N}{N \geq 1}, \quad \asetc{\bar\bD^N}{N \geq 1} \quad \text{and} \quad \asetc{\bar\bD^{N, c}}{N \geq 1}
	\end{equation*}
	are tight, the former in $D_{\calM_F^+\left(\R_+^2\right)}[0, \infty)$ and the latter three in $D_{\calM_F^+\left(\R_+^3\right)}[0, \infty)$.
\end{theorem}

\begin{proof}
	As noted earlier, the proof consists of verifying the criteria of Theorem \ref{the: Jakubowski's criteria}. For the first sequence, we have already shown in Lemma \ref{lem: relative compactness of projections by test functions} that J2 holds, so it remains to prove that J1 holds as well. For this purpose, we must fix arbitrary $T > 0$ and $\varepsilon > 0$ and prove that there exists a compact set $C \subset \calM_F^+(\R_+^2)$ such that
	\begin{equation}
		\label{eq: compact containment condition for measure-valued state}
		\limsup_{N \to \infty} \cprob*{\bar\bmu_t^N \notin C\ \text{for some}\ t \in [0, T]} \leq \varepsilon.
	\end{equation}
	
	By Assumption \ref{ass: standing assumptions 1}, the sequence of measure-valued initial states $\setc{\bar\bmu_0^N}{N \geq 1}$ is tight in $\calM_F^+(\R_+^2)$. Then there exists a compact set $K \subset \calM_F^+(\R_+^2)$ such that
	\begin{equation}
		\label{eq: tightness of initial states}
		\sup_{N \geq 1} \cprob*{\bar\bmu_0^N \notin K} \leq \frac{\varepsilon}{2}.
	\end{equation}
	Observe that Prohorov's theorem implies that $K$ is tight in $\calM_F^+(\R_+^2)$, because it is compact. Thus, for each $\delta > 0$ there exists $T_\delta \geq 0$ such that
	\begin{equation*}
		\sup_{\mu \in K} \mu\left(\R_+^2 \setminus [0, T_\delta]^2\right) \leq \delta.
	\end{equation*}
	
	In order to prove \eqref{eq: compact containment condition for measure-valued state}, consider the following sets:
	\begin{equation*}
		A_L^\Delta(\delta) \defeq \asetc{\mu \in \calM_F^+\left(\R_+^2\right)}{\mu\left(\R_+^2\right) \leq L,\ \mu\left(\R_+^2 \setminus [0, T_\delta + \Delta]^2\right) \leq \delta}, 
	\end{equation*}
	and let $A_L^\Delta$ be the intersection of the sets $A_L^\Delta(\delta)$ over $\delta > 0$. Then $A_L^\Delta$ is uniformly bounded in the total variation norm and tight, hence relatively compact by Prohorov's theorem. We define $C_L^\Delta$ as the closure of $A_L^\Delta$ in $\calM_F^+(\R_+^2)$, which is a compact set.
	
	By Markov's inequality, Assumption \ref{ass: standing assumptions 1} and the elementary renewal theorem:
	\begin{equation}
		\label{eq: bound for probability of large total variation}
		\begin{split}
			\limsup_{N \to \infty} \cprob*{\sup_{t \in [0, T]} \bar\bmu_t^N\left(\R_+^2\right) > L} &\leq \limsup_{N \to \infty} \cprob*{\frac{M^N + \calR^N(T)}{N} > L} \\
			&\leq \limsup_{N \to \infty} \frac{\expect*{M^N + \calR^N(T)}}{LN} \leq \frac{\bar M + \lambda_0 T}{L},
		\end{split}
	\end{equation}
	where the right-hand side can be made smaller than $\varepsilon / 2$ by increasing $L$. In addition,
	\begin{equation}
		\label{eq: statement about compact containment}
		\bar\bmu_t^N\left(\R_+^2 \setminus [0, T_\delta + \Delta]^2\right) > \delta \quad \text{for some} \quad t \in [0, T]
	\end{equation}
	implies that $\bx_i^N(t) > T_\delta + \Delta$ or $\by_i^N(t) > T_\delta + \Delta$ for some time $t \in [0, T]$ and sufficiently many tasks. Note that $\bx_i^N(t) \leq \bx_i^N(0) + t$ and $\by_i^N(t) \leq \by_i^N(0) + t$, so if $\Delta > T$, then \eqref{eq: statement about compact containment} implies that sufficiently many tasks are present at time zero and satisfy:
	\begin{equation*}
		\bx_i^N(0) > T_\delta + \Delta - t \geq T_\delta + \Delta - T \quad \text{or} \quad \by_i^N(0) > T_\delta + \Delta - t \geq T_\delta + \Delta - T \quad \text{with} \quad t \in [0, T].
	\end{equation*}
	We conclude that \eqref{eq: statement about compact containment} implies that $\bar\bmu_0^N(\R_+^2 \setminus [0, T_\delta + \Delta - T]^2) > \delta$, which clearly implies that $\bar\bmu_0^N(\R_+^2 \setminus [0, T_\delta]^2) > \delta$ if $\Delta > T$. It follows from the above implications that
	\begin{equation}
		\label{eq: bound for probability of escaping mass}
		\begin{split}
			\cprob*{\bigcup_{\delta > 0} \left\{\sup_{t \in [0, T]} \bar\bmu_t^N\left(\R_+^2 \setminus [0, T_\delta + \Delta]^2\right) > \delta\right\}} &\leq \cprob*{\bigcup_{\delta > 0} \left\{\bar\bmu_0^N\left(\R_+^2 \setminus [0, T_\delta]^2\right) > \delta\right\}} \\
			&\leq \cprob*{\bar\bmu_0^N \notin K} \quad \text{for all} \quad N \geq 1.
		\end{split}
	\end{equation}
	
	Hence, by \eqref{eq: tightness of initial states}, \eqref{eq: bound for probability of large total variation} and \eqref{eq: bound for probability of escaping mass}, if $\Delta > T$, then we may select $L \geq 0$ such that:
	\begin{align*}
		\limsup_{N \to \infty} \cprob*{\bar\bmu_t^N \notin C_L^\Delta\ \text{for some}\ t \in [0, T]} \leq \frac{\bar M + \lambda_0 T}{L} + \sup_{N \geq 1} \cprob*{\bar\bmu_0^N \notin K} \leq \varepsilon,
	\end{align*}
	which proves \eqref{eq: compact containment condition for measure-valued state}. Therefore, the sequence $\setc{\bar\bmu^N}{N \geq 1}$ is tight in $D_{\calM_F^+(\R_+^2)}[0, \infty)$.
	
	In order to show that $\setc{\bar\bA^N}{N \geq 1}$ is tight, it is again sufficient to prove that condition J1 holds since Lemma \ref{lem: relative compactness of projections by test functions} implies that condition J2 holds. Thus, we fix arbitrary $T > 0$ and $\varepsilon > 0$ and proceed to establish J1. For this purpose, note that the set
	\begin{equation*}
		\asetc{\mu \in \calM_F^+\left(\R_+^3\right)}{\mu\left(\R_+^3\right) \leq L,\ \mu\left(\R_+^3 \setminus [0, T]^3\right) = 0}
	\end{equation*}
	is bounded in the total variation norm and tight, so its closure $C_L$ is compact in $\calM_F^+(\R_+^3)$ by Prohorov's theorem. The support of $\bar\bA_t^N$ is contained in $[0, T]^3$ for all $t \in [0, T]$. Also, if $\indc$ denotes the function that is identically equal to one, then
	\begin{equation*}
		\limsup_{N \to \infty} \cprob*{\sup_{t \in [0, T]} \bar\bA_t^N(\indc) > L} = \limsup_{N \to \infty} \cprob*{\frac{\calR^N(T)}{N} > L} \leq \limsup_{N \to \infty} \frac{\expect*{\calR^N(T)}}{LN} = \frac{\lambda_0 T}{L}
	\end{equation*}
	by the elementary renewal theorem. Hence, there exists $L \geq 0$ such that the right-hand side is strictly smaller than $\varepsilon$, and it follows that the probability that $\bar\bA_t^N \notin C_L$ for some $t \in [0, T]$ is less than $\varepsilon$ for all large enough $N$, which proves J1.
	
	In order to prove tightness of the remaining two sequences of processes, we proceed as before, fixing arbitrary $T > 0$ and $\varepsilon > 0$ and proving J1. Define
	\begin{equation*}
		\varepsilon_n \defeq \frac{\varepsilon}{n 2^{n + 1} \left(\bar M + \lambda_0 T\right)} \quad \text{for all} \quad n \geq 1.
	\end{equation*}
	Lemma \ref{lem: bounds for integrals of hazard rate} associates to each of the above constants $\varepsilon_n$ a compact set $B_n \subset U_g$ and a function $\psi_n$ with values in $[0, 1]$ vanishing outside $B_n$. Moreover, it implies that
	\begin{align*}
		\max \left\{\expect*{\bar\bD_T^N\left(\indc_{\R_+ \times B_n^c}\right)}, \expect*{\bar\bD_T^{N, c}\left(\indc_{\R_+ \times B_n^c}\right)}\right\} &\leq \\
		\max \left\{\expect*{\bar\bD_T^N\left(1 - \psi_n\right)}, \expect*{\bar\bD_T^{N, c}\left(1 - \psi_n\right)}\right\} &\leq \varepsilon_n\left(\bar M + \frac{\expect*{\calR^N(T)}}{N}\right).
	\end{align*}
	Because $\bar\bD_T^N$ and $\bar\bD_T^{N, c}$ are supported in $[0, T] \times \R_+^2$, we conclude that
	\begin{equation}
		\label{eq: bound on expectations for proving tightness of departure processes}
		\begin{split}
			&\expect*{\bar\bD_T^N\left(\ind{\R_+^3 \setminus \left([0, T] \times B_n\right)}\right)} = \expect*{\bar\bD_T^N\left(\indc_{\R_+ \times B_n^c}\right)} \leq \varepsilon_n\left(\bar M + \frac{\expect*{\calR^N(T)}}{N}\right), \\
			&\expect*{\bar\bD_T^{N, c}\left(\ind{\R_+^3 \setminus \left([0, T] \times B_n\right)}\right)} = \expect*{\bar\bD_T^{N, c}\left(\indc_{\R_+ \times B_n^c}\right)} \leq \varepsilon_n\left(\bar M + \frac{\expect*{\calR^N(T)}}{N}\right).
		\end{split}
	\end{equation}
	
	Given $L \geq 0$ and $n \geq 1$, define
	\begin{equation*}
		C_L^n \defeq \asetc{\mu \in \calM_F^+\left(\R_+^3\right)}{\mu\left(\R_+^3\right) \leq L,\ \mu\left(\R_+^3 \setminus \left([0, T] \times B_n\right)\right) \leq \frac{1}{n}}.
	\end{equation*}
	Then let $C_L$ be the intersection of the sets $C_L^n$ over $n \geq 1$. As before, the set $C_L$ is relatively compact by Prohorov's theorem, and therefore its closure $K_L$ is compact. Below we prove that $L \geq 0$ can be selected in such a way that
	\begin{equation}
		\label{eq: compact containment condition for departure process}
		\limsup_{N \to \infty} \cprob*{\bar\bD_t^N \notin K_L\ \text{for some}\ t \in [0, T]} \leq \varepsilon.
	\end{equation}
	This establishes that J1 holds for $\setc{\bar\bD^N}{N \geq 1}$, and thus that the sequence is tight. We only consider this sequence since the same arguments carry over for $\bar\bD^{N, c}$.
	
	First, note that $\bar\bD_t^N(\varphi)$ is nondecreasing in $t$ for all nonnegative functions $\varphi$. Thus,
	\begin{align*}
		\cprob*{\bigcup_{n \geq 0} \left\{\sup_{t \in [0, T]} \bar\bD_t^N\left(\ind{\R_+^3 \setminus \left([0, T] \times B_n\right)}\right) > \frac{1}{n}\right\}} &= \cprob*{\bigcup_{n \geq 0} \left\{\bar\bD_T^N\left(\ind{\R_+^3 \setminus \left([0, T] \times B_n\right)}\right) > \frac{1}{n}\right\}} \\
		&\leq \sum_{n = 1}^\infty n\expect*{\bar\bD_T^N\left(\ind{\R_+^3 \setminus \left([0, T] \times B_n\right)}\right)} \\
		&\leq \frac{\varepsilon}{2\left(\bar M + \lambda_0 T\right)} \left(\bar M + \frac{\expect*{\calR^N(T)}}{N}\right),
	\end{align*}
	where the first inequality holds by the union bound and Markov's inequality, and the second inequality holds by \eqref{eq: bound on expectations for proving tightness of departure processes}. It follows from the elementary renewal theorem that
	\begin{equation}
		\label{eq: first bound for tightness of departure process}
		\limsup_{N \to \infty} \cprob*{\bigcup_{n \geq 0} \left\{\sup_{t \in [0, T]} \bar\bD_t^N\left(\ind{\R_+^3 \setminus \left([0, T] \times B_n\right)}\right) > \frac{1}{n}\right\}} \leq \frac{\varepsilon}{2}.
	\end{equation}
	
	In addition, recalling that $\indc$ is the function that is identically equal to one, we obtain:
	\begin{equation*}
		\cprob*{\sup_{t \in [0, T]} \bar\bD_t^N(\indc) > L} = \cprob*{\bar\bD_T^N(\indc) > L} \leq \frac{\expect*{\bar\bD_T^N(\indc)}}{L} \leq \frac{\expect*{M^N + \calR^N(T)}}{LN}
	\end{equation*}
	The last inequality holds since the number of departures in $[0, T]$ is bounded by the initial number of tasks and the number of arrivals; note that the last inequality still holds if $\bar\bD^N$ is replaced by $\bar\bD^{N, c}$ since $E[\bar\bD_T^N(\indc)] = E[\bar\bD_T^{N, c}(\indc)]$ by Theorem \ref{the: compensator of departures-driven process}. Then Assumption~\ref{ass: standing assumptions 1} and the elementary renewal theorem imply that there exists $L \geq 0$ such that
	\begin{equation}
		\label{eq: second bound for tightness of departure process}
		\limsup_{N \to \infty} \cprob*{\sup_{t \in [0, T]} \bar\bD_t^N(\indc) > L} \leq \frac{\varepsilon}{2}.
	\end{equation}
	
	It follows from \eqref{eq: first bound for tightness of departure process} and \eqref{eq: second bound for tightness of departure process} that \eqref{eq: compact containment condition for departure process} holds, which completes the proof.
\end{proof}

\subsubsection{Convergence of integrals}
\label{subsub: convergence of integrals}

We conclude this section by establishing that convergence in $D_{\calM_F^+(\R_+^3)}[0, \infty)$ implies convergence of certain integrals in $D_\R[0, \infty)$, a property that will be useful in the next section. We provide the proof of the following lemma in Appendix \ref{app: proofs of auxiliary results}.

\begin{restatable}{lemma}{lemmafoursixteen}
	\label{lem: convergence measure-valued functions implies convergence of projections}
	Suppose that $\bY^N \to \bY$ in $D_{\calM_F^+(\R_+^3)}[0, \infty)$ as $N \to \infty$ and let $\varphi \in C_b(\R_+^3)$ be a Lipschitz function. For each $t \geq 0$, let us denote the integrals of $\varphi$ with respect to $\bY_t^N$ and $\bY_t$ by $\bY_t^N(\varphi)$ and $\bY_t(\varphi)$, respectively. Then $\bY^N(\varphi) \to \bY(\varphi)$ in $D_\R[0, \infty)$.
\end{restatable}

\subsection{Characterization of subsequential limits}
\label{sub: characterization of subsequential limits}

Consider the space $\calM_F^+(\R_+^2) \times D_{\calM_F^+(\R_+^2)}[0, \infty) \times [D_{\calM_F^+(\R_+^3)}[0, \infty)]^3 \times D_\R[0, \infty)$ endowed with a metric that is compatible with the product topology. Using the usual notation, let $\bar\calR^N \defeq \calR^N / N$ be the normalized arrival process. The sequence
\begin{equation*}
	\asetc{\bar\bX^N \defeq \left(\bar\bmu_0^N, \bar\bmu^N, \bar\bA^N, \bar\bD^N, \bar\bD^{N, c}, \bar\calR^N\right)}{N \geq 1}
\end{equation*}
is tight in the aforementioned space by Assumption \ref{ass: standing assumptions 1}, Theorem \ref{the: relative compactness of measured-valued processes} and the functional law of large numbers for renewal processes; see \cite[Theorem 5.10]{chen2001fundamentals}. By Prohorov's theorem, every subsequence has a further subsequence which converges weakly.

For the rest of this section, we fix an arbitrary convergent subsequence and keep using the index $N$ for conciseness and despite a slight abuse of notation. We denote the limit of this convergent subsequence by $\bar\bX = (\bar\bmu_0, \bar\bmu, \bar\bA, \bar\bD, \bar\bD^c, \bar\calR)$ and observe that $\bar\calR(t) = \lambda_0 t$ by Assumption \ref{ass: standing assumptions 1} and the functional law of large numbers for renewal processes.

By Skorohod's representation theorem, there exist versions of $\bar\bX^N$ and $\bar\bX$ defined on a common probability space $(\Omega, \calF, P)$ for all $N \geq 1$, and such that with probability one:
\begin{equation*}
	\lim_{N \to \infty} \bar\bX^N = \bar\bX \quad \text{in} \quad \calM_F^+\left(\R_+^2\right) \times D_{\calM_F^+\left(\R_+^2\right)}[0, \infty) \times \left[D_{\calM_F^+\left(\R_+^3\right)}[0, \infty)\right]^3 \times D_\R[0, \infty).
\end{equation*}
In the sequel, we consider these versions of $\bar\bX^N$ and $\bar\bX$ and characterize the limit $\bar\bX$ in a set of probability one, completing the program of Section \ref{sub: outline of the proof} for proving Theorem \ref{the: measure-valued fluid limit}.

Our first result in this direction is the following lemma.

\begin{lemma}
	\label{lem: departure process and compensator have same limit}
	The processes $\bar\bD$ and $\bar\bD^c$ are equal with probability one.
\end{lemma}

\begin{proof}
	If $\varphi \in C_b(\R_+^3)$ is Lipschitz, then Lemma \ref{lem: convergence measure-valued functions implies convergence of projections} implies that
	\begin{equation*}
		\lim_{N \to \infty} \bar\bD^N(\varphi) = \bar\bD(\varphi) \quad \text{and} \quad \lim_{N \to \infty} \bar\bD^{N, c}(\varphi) = \bar\bD^c(\varphi) \quad \text{in} \quad D_\R[0, \infty)
	\end{equation*}
	with probability one. Since $\bar\bD^N(\varphi) - \bar\bD^{N, c}(\varphi) \Rightarrow 0$ in $D_\R[0, \infty)$ by Corollary \ref{cor: limit of departures martingale}, we get:
	\begin{equation*}
		\cprob*{\bar\bD_t(\varphi) = \bar\bD_t^c(\varphi)\ \text{for all}\ t \geq 0} = 1. 
	\end{equation*}
	
	It remains to prove that the above equality holds simultaneously, and almost surely, for all $\varphi$ in a separating class of functions. Using the Stone-Weierstrass theorem, we obtain that the set of polynomials with rational coefficients is dense in $C_c(K)$ for each compact set $K \subset \R_+^3$, so there exists a countable set $\Psi \subset C_c^1(\R_+^3)$ that is dense in $C_c(\R_+^3)$, and
	\begin{equation*}
		\cprob*{\bar\bD_t(\psi) = \bar\bD_t^c(\psi)\ \text{for all}\ t \geq 0\ \text{and}\ \psi \in \Psi} = 1. 
	\end{equation*}
	We conclude that $\bar\bD_t = \bar\bD_t^c$ for all $t \geq 0$ with probability one, because the functions in $C_c(\R_+^3)$ separate the points of $\calM_F^+(\R_+^3)$ and the set $\Psi$ is dense in $C_c(\R_+^3)$.
\end{proof}

The following lemma characterizes the limit of the arrival processes $\bar\bA^N$. We defer the proof of this standard lemma until Appendix \ref{app: proofs of auxiliary results}.

\begin{restatable}{lemma}{lemmafoureighteen}
	\label{lem: limit of arrival process}
	The following property holds with probability one:
	\begin{equation}
		\label{eq: limit of arrival process}
		\bar\bA_t(\varphi) = \lambda_0 \int_0^t \varphi_s(0, 0)ds \quad \text{for all} \quad t \geq 0 \quad \text{and} \quad \varphi \in C_b(\R_+^3).
	\end{equation}
\end{restatable}

We summarize the above lemmas in the next proposition.

\begin{proposition}
	\label{prop: set of probability one}
	There exists a set of probability one $\Gamma \in \calF$ such that if $\omega \in \Gamma$, then
	\begin{equation*}
		\lim_{N \to \infty} \bar\bX^N(\omega) = \bar\bX(\omega), \quad \bar\bA_t(\omega, \varphi) = \lambda_0 \int_0^t \varphi_s(0, 0)ds \quad \text{and} \quad \bar\bD(\omega) = \bar\bD^c(\omega).
	\end{equation*}
	We may and will define $\Gamma$ such that property (c) of Assumption \ref{ass: standing assumptions 1} holds in $\Gamma$.
\end{proposition}

In the following sections we characterize $\bar\bmu$ and $\bar\bD$, often considering sample paths in the set $\Gamma$ of probability one. First, we show that $\bar\bmu$ satisfies (a) and (b) of Definition \ref{def: fluid problem} in Section \ref{subsub: support and absolute continuity}. Then, we prove (c) and (d) in Sections \ref{subsub: limits involving service rates}, \ref{subsub: limits involving hazard rates} and \ref{subsub: derivation of the transport equation}.

\subsubsection{Support and absolute continuity}
\label{subsub: support and absolute continuity}

The following random continuity set will be used to prove the results in this section.

\begin{definition}
	\label{def: continuity set}
	For each $\omega \in \Omega$, we let $\calC(\omega)$ denote the set of times $t \geq 0$ at which $\bar\bmu(\omega)$ is continuous in $\calM_F^+(\R_+^2)$, i.e., with respect to the weak topology.
\end{definition}

Note that $\calC$ has countable complement surely, because $t \mapsto \bar\bmu_t$ is a c\`adl\`ag function surely. In addition, if $\omega \in \Gamma$ and $t \in \calC(\omega)$, then $\bar\bmu_t^N(\omega) \to \bar\bmu_t(\omega)$ weakly as $N \to \infty$ since $\bar\bmu^N(\omega) \to \bar\bmu(\omega)$ as $N \to \infty$ in the Skorohod-$J_1$ topology.

The following proposition proves that $\bar \bmu$ satisfies (a) of Definition \ref{def: fluid problem}.

\begin{proposition}
	\label{prop: support of fluid limit}
	We have $\bar\bmu_t(U_g^c) = 0$ for all $t \geq 0$ almost surely.
\end{proposition}

\begin{proof}
	Define $C_\varepsilon \defeq \setc{z \in \R_+^2}{\mathrm{dist}(z, U_g^c) \leq \varepsilon}$ for each $\varepsilon > 0$, where $\mathrm{dist}(z, U_g^c)$ denotes the distance from the point $z$ to the closed set $U_g^c$ with respect to the standard Euclidean norm $\norm{\scdot}_2$. It follows from the Glivenko-Cantelli theorem that 
	\begin{equation}
		\label{eq: glivenko-cantelli}
		\lim_{n \to \infty} \frac{1}{n}\sum_{i = 1}^n \ind{\left(b_i^1, c_i^1\right) \in C_\varepsilon} = \lim_{n \to \infty} \frac{1}{n}\sum_{i = 1}^n \ind{\mathrm{dist}\left(\left(b_i^1, c_i^1\right), U_g^c\right) \leq \varepsilon} = \int_{C_\varepsilon} g(x, y)dxdy
	\end{equation}
	simultaneously for all $\varepsilon > 0$ with probability one. Similarly,
	\begin{equation*}
		\lim_{n \to \infty} \frac{1}{n}\sum_{i = 1 - n}^0 \ind{\left(b_i^1, c_i^1\right) \in C_\varepsilon} = \int_{C_\varepsilon} g(x, y)dxdy \quad \text{for all} \quad \varepsilon > 0 \quad \text{almost surely}.
	\end{equation*}
	
	Fix $\omega \in \Gamma$ in the set of probability one where the above limits hold; in the sequel, $\omega$ is omitted from the notation for conciseness. Also, recall that $\bar\bmu_t^N \to \bar\bmu_t$ weakly as $N \to \infty$ for all $t \in \calC$. Furthermore, we have
	\begin{equation}
		\label{eq: limit of initial number of tasks}
		\lim_{N \to \infty} \frac{M^N}{N} = \lim_{N \to \infty} \bar\bmu_0^N\left(\R_+^2\right) = \bar\bmu_0\left(\R_+^2\right).
	\end{equation}
	
	For each $\varepsilon > 0$, choose any $\phi^\varepsilon \in C_b(\R_+^2)$ such that:
	\begin{equation*}
		\ind{(x, y) \in U_g^c} \leq \phi^\varepsilon (x, y) \leq \ind{(x, y) \in C_\varepsilon} \quad \text{for all} \quad x, y \in \R_+.
	\end{equation*}
	If $(x_0, y_0) \in C_\varepsilon$, then there exists $(u_0, v_0) \in U_g^c$ such that $\norm{(x_0, y_0) - (u_0, v_0)}_2 \leq \varepsilon$. The definition of $U_g$ implies that $(u, v) \in U_g^c$ whenever $u \geq u_0$ and $v \geq v_0$. Suppose that $x_0 \leq b$ and $y_0 \leq c$ for some $(b, c) \in U_g$. Then $(b, c) \in C_\varepsilon$ since
	\begin{align*}
		\mathrm{dist}\left((b, c), U_g^c\right) &\leq \min\asetc{\sqrt{(b - u)^2 + (c - v)^2}}{u \geq u_0, v \geq v_0} \\
		&= \min\asetc{\sqrt{(x - u)^2 + (y - v)^2}}{x \leq b, y \leq c, u \geq u_0, v \geq v_0} \\
		&\leq \norm{(x_0, y_0) - (u_0, v_0)}_2 \leq \varepsilon.
	\end{align*}
	For the equality, observe that $(x, u)$ decouples from $(y, v)$ in the minimization, and that the minimum can always be attained with $x = b$ and $y = c$. Hence, $(\bx_i^N(t), \by_i^N(t)) \in C_\varepsilon$ implies that $(b_i^N, c_i^N) \in C_\varepsilon$ for all tasks and times.
	
	We conclude that
	\begin{equation*}
		\angled{\phi^\varepsilon}{\bar\bmu_t^N} \leq \frac{1}{N}\sum_{i = 1 - M^N}^{\calR^N(t)} \ind{t < d_i^N, \left(\bx_i^N(t), \by_i^N(t)\right) \in C_\varepsilon} \leq \frac{1}{N}\sum_{i = 1 - M^N}^{\calR^N(t)} \ind{\left(b_i^1, c_i^1\right) \in C_\varepsilon}
	\end{equation*}
	for all $t \geq 0$ and $N \geq 1$, where we recall that $(b_i^N, c_i^N) = (b_i^1, c_i^1)$ by Assumption \ref{ass: standing assumptions 1}. Then it follows from \eqref{eq: glivenko-cantelli} and \eqref{eq: limit of initial number of tasks} that for all $t \in \calC$ and $\varepsilon > 0$, we have: 
	\begin{align*}
		\angled{\phi^\varepsilon}{\bar\bmu_t} &= \lim_{N \to \infty} \angled{\phi^\varepsilon}{\bar\bmu_t^N} \\
		&\leq \lim_{N \to \infty} \left[\frac{M^N}{N} \frac{1}{M^N}\sum_{i = 1 - M^N}^{0} \ind{\left(b_i^1, c_i^1\right) \in C_\varepsilon} + \frac{\calR^N(t)}{N} \frac{1}{\calR^N(t)}\sum_{i = 1}^{\calR^N(t)} \ind{\left(b_i^1, c_i^1\right) \in C_\varepsilon}\right] \\
		&= \left[\bar\bmu_0\left(\R_+^2\right) + \lambda_0 t\right] \int_{C_\varepsilon} g(x, y)dx dy.
	\end{align*}
	
	Since $\calC$ has countable complement, for each $t \geq 0$ there exists $\setc{t_n}{n \geq 1} \subset \calC$ such that $t_n \downarrow t$ as $n \to \infty$. Therefore, we obtain:
	\begin{equation*}
		\angled{\phi^\varepsilon}{\bar\bmu_t} = \lim_{n \to \infty} \angled{\phi^\varepsilon}{\bar\bmu_{t_n}} \leq \left[\bar\bmu_0\left(\R_+^2\right) + \lambda_0 t\right] \int_{C_\varepsilon} g(x, y)dx dy \quad \text{for all} \quad \varepsilon > 0,
	\end{equation*}
	by the right-continuity of $\bar\bmu$. Because $\phi^\varepsilon \to \indc_{U_g^c}$ pointwise as $\varepsilon \to 0$, we conclude that
	\begin{equation*}
		\bar\bmu_t\left(U_g^c\right) = \lim_{\varepsilon \to 0} \angled{\phi^\varepsilon}{\bar\bmu_t} \leq \lim_{\varepsilon \to 0} \left[\bar\bmu_0\left(\R_+^2\right) + \lambda_0 t\right] \int_{C_\varepsilon} g(x, y)dx dy = 0
	\end{equation*}
	for all $t \geq 0$, and this completes the proof.
\end{proof}

Next we show that $\bar\bmu$ satisfies (b) of Definition \ref{def: fluid problem}. Consider the following $y$-marginals:
\begin{equation*}
	\bar\bnu_t^N(A) \defeq \bar\bmu_t^N(\R_+ \times A) \quad \text{and} \quad \bar\bnu_t(A) \defeq \bar\bmu_t(\R_+ \times A)
\end{equation*}
for all $t \geq 0$ and Borel $A \subset \R_+$. Let $\bar\by_*^N(t) \defeq y_*^N(\bmu_t^N) = \bar y_*(\bar\bmu_t^N)$ and $\bar\by_*(t) \defeq \bar y_*(\bar\bmu_t)$, i.e.,
\begin{align*}
	\bar\by_*^N(t) = \inf\asetc{y \geq 0}{\bar\bnu_t^N\left([0, y]\right) \geq 1} \quad \text{and} \quad \bar\by_*(t) = \inf\asetc{y \geq 0}{\bar\bnu_t\left([0, y]\right) \geq 1}.
\end{align*}
Furthermore, note that $r^N(\bmu_t^N, y) = \bar r(\bar\bmu_t^N, y) = \ind{y \leq \bar\by_*^N(t)}$ and $\bar r(\bar\bmu_t, y) = \ind{y \leq \bar\by_*(t)}$.

\begin{proposition}
	\label{prop: limiting measure of y-intervals}
	For each $\omega \in \Gamma$, there exists a nonnegative function $\map{\theta(\omega)}{\R}{\R_+}$ such that $\theta(\omega) \in L_{loc}^1(\R)$ and the following property holds:
	\begin{equation}
		\label{eq: absolute continuity with respect to lebesgue measure}
		\bar\bnu_t\left(\omega, [y_1, y_2]\right) = \lim_{N \to \infty} \bar\bnu_t^N\left(\omega, [y_1, y_2]\right) \leq \int_{y_1}^{y_2} \theta(\omega, s - t)ds
	\end{equation}
	for all $t \in \calC(\omega)$ and $0 \leq y_1 \leq y_2$. Moreover,
	\begin{equation*}
		\bar\bnu_t\left(\omega, [y_1, y_2]\right) \leq \int_{y_1}^{y_2} \theta(\omega, s - t)ds
	\end{equation*}
	for all $t \geq 0$ and $0 \leq y_1 \leq y_2$, i.e., even if $t \notin \calC(\omega)$. Hence, the measures $\bar\bnu_t$ are absolutely continuous with respect to the Lebesgue measure for all $t \geq 0$ with probability one.
\end{proposition}

\begin{proof}
	Fix $\omega \in \Gamma$ and recall that:
	\begin{equation*}
		\bar\bmu^N(\omega) \to \bar\bmu(\omega) \quad \text{in} \quad D_{\calM_F^+(\R_+^2)}[0, \infty) \quad \text{and} \quad \bar\calR^N(\omega) \to \bar\calR \quad \text{in} \quad D_\R[0, \infty) \quad \text{as} \quad N \to \infty.
	\end{equation*}
	The second limit holds uniformly over compact sets since $\bar\calR$ is continuous, and if $t \in \calC(\omega)$, then $\bar\bmu_t^N(\omega) \to \bar\bmu_t(\omega)$ weakly. For conciseness, $\omega$ will be omitted from the notation.
	
	Given $t \in \calC$, $0 \leq y_1 \leq y_2$ and $\varepsilon > 0$, there exists $\phi^\varepsilon \in C_b(\R_+^2)$ such that
	\begin{equation*}
		\ind{y_1 \leq y \leq y_2} \leq \phi^\varepsilon(x, y) \leq \ind{(y_1 - \varepsilon)^+ \leq y \leq y_2 + \varepsilon} \quad \text{for all} \quad x, y \in \R_+.
	\end{equation*}	
	A task present in the system at time $t$ has spent $y$ units of time in the system by time $t$ if and only if one of the following conditions holds: the task was present at time zero with elapsed time in the system $y - t$, or the task arrived at time $t - y$. Therefore,
	\begin{align*}
		\lim_{N \to \infty} \angled{\phi^\varepsilon}{\bar\bmu_t^N} &\leq \limsup_{N \to \infty} \bar\bnu_0^N\left(\left[\left((y_1 - \varepsilon)^+ - t\right)^+, (y_2 + \varepsilon - t)^+\right]\right)\\
		&+ \lim_{N \to \infty} \left[\bar\calR^N\left(\left(t - (y_1 - \varepsilon)^+\right)^+\right) - \bar\calR^N\left(\left(t - y_2 - \varepsilon\right)^+\right)\right] \\
		&\leq \int_{(y_1 - \varepsilon - t)^+}^{(y_2 + \varepsilon - t)^+} \theta_0(s)ds + \lambda_0\left[y_2 - y_1 + 2\varepsilon\right] = \int_{y_1 - \varepsilon}^{y_2 + \varepsilon} \theta(s - t)ds,
	\end{align*}
	where $\theta_0$ is the function introduced in Assumption \ref{ass: standing assumptions 1} and $\theta(s) \defeq \theta_0(s)\ind{s \geq 0} + \lambda_0$. The latter function is nonnegative and locally integrable since $\theta_0$ has these properties.
	
	If $t \in \calC$, then $\langle\phi^\varepsilon, \bar\bmu_t^N\rangle \to \langle\phi^\varepsilon, \bar\bmu_t\rangle \geq \bar\bnu_t([y_1, y_2])$ as $N \to \infty$ for all $\varepsilon > 0$, and we obtain:
	\begin{equation*}
		\max\left\{\bar\bnu_t\left([y_1, y_2]\right), \limsup_{N \to \infty} \bar\bnu_t^N\left([y_1, y_2]\right)\right\} \leq \inf_{\varepsilon > 0} \lim_{N \to \infty} \angled{\phi^\varepsilon}{\bar\bmu_t^N} \leq \int_{y_1}^{y_2} \theta(s - t)ds.
	\end{equation*}
	for all $t \in \calC$ and $0 \leq y_1 \leq y_2$, which proves the inequality in \eqref{eq: absolute continuity with respect to lebesgue measure}.
	
	Applying the above inequality to the intervals $[(y_1 - \varepsilon)^+, y_1]$ and $[y_2, y_2 + \varepsilon]$, we get:
	\begin{align*}
		\limsup_{N \to \infty} \left|\bar\bnu_t\left([y_1, y_2]\right) - \bar\bnu_t^N\left([y_1, y_2]\right)\right| &\leq \left|\bar\bnu_t\left([y_1, y_2]\right) - \angled{\phi^\varepsilon}{\bar\bmu_t}\right| \\
		&+ \limsup_{N \to \infty} \left|\angled{\phi^\varepsilon}{\bar\bmu_t^N} - \bar\bnu_t^N\left([y_1, y_2]\right)\right| \\
		&\leq \bar\bnu_t\left([(y_1 - \varepsilon)^+, y_1]\right) + \bar\bnu_t\left([y_2, y_2 + \varepsilon]\right) \\
		&+ \limsup_{N \to \infty} \bar\bnu_t^N\left([(y_1 - \varepsilon)^+, y_1]\right) \\
		&+ \limsup_{N \to \infty} \bar\bnu_t^N\left([y_2, y_2 + \varepsilon]\right) \\
		&\leq 2\int_{y_1 - \varepsilon}^{y_1}\theta(s - t)ds + 2\int_{y_2}^{y_2 + \varepsilon}\theta(s - t)ds.
	\end{align*}
	for all $t \in \calC$ and $0 \leq y_1 \leq y_2$. The first inequality holds since $\langle\phi^\varepsilon, \bar\bmu_t^N\rangle \to \langle\phi^\varepsilon, \bar\bmu_t\rangle$. Note that this completes the proof of \eqref{eq: absolute continuity with respect to lebesgue measure} because $\varepsilon > 0$ is arbitrary.
	
	Finally, observe that for any $t \geq 0$, there exists a sequence $\setc{t_n}{n \geq 1} \subset \calC$ such that $t_n \downarrow t$ as $n \to \infty$; in particular, this holds for times $t \notin \calC$. By right-continuity, $\bar\bnu_{t_n}$ converges weakly to $\bar\bnu_t$ as $n \to \infty$. Therefore, if $0 < y_1 \leq y_2$ and $\varepsilon < y_1$, then
	\begin{equation*}
		\bar\bnu_t\left((y_1 - \varepsilon, y_2 + \varepsilon)\right) \leq \liminf_{n \to \infty} \bar\bnu_{t_n}\left((y_1 - \varepsilon, y_2 + \varepsilon)\right) \leq \int_{y_1 - \varepsilon}^{y_2 + \varepsilon} \theta(s - t)ds.
	\end{equation*}
	by the Portmanteau theorem and \eqref{eq: absolute continuity with respect to lebesgue measure}. Letting $\varepsilon \to 0$, we get:
	\begin{equation*}
		\bar\bnu_t\left([y_1, y_2]\right) \leq \int_{y_1}^{y_2} \theta(s - t)ds \quad \text{for all} \quad 0 < y_1 \leq y_2.
	\end{equation*}
	An analogous argument applies when $y_1 = 0$, except that in this case the closed interval $[y_1, y_2] = [0, y_2]$ needs to be approximated by the intervals $[0, y_2 + \varepsilon)$. Because the latter intervals are open in $\R_+$, we can still invoke the Portmanteau theorem as above.
\end{proof}

\subsubsection{Limits involving service rates}
\label{subsub: limits involving service rates}

This and the following sections complete the characterization of $\bar\bmu$ by analyzing the limiting behavior of the pre-limit equation \eqref{eq: pre-limit equation in normalized notation}. In particular, it remains to characterize the limits of the third and fourth terms on the right-hand side of this equation. The fourth term is an integral involving the service rate function $\bar r^N(\bar \bmu)$. In this section we analyze how integrals of this type behave in the limit, whereas the following section concerns integrals involving hazard rates as well, as in the third term on the right-hand side of \eqref{eq: pre-limit equation in normalized notation}. As a first step, we derive some properties of the processes $\bar\by_*^N$ and $\bar\by_*$.
\begin{lemma}
	\label{lem: limiting threshold}
	For each $t \geq 0$, let
	\begin{equation*}
		\bar\by_*^-(t) \defeq \liminf_{N \to \infty} \bar\by_*^N(t) \quad \text{and} \quad \bar\by_*^+(t) \defeq \limsup_{N \to \infty} \bar\by_*^N(t).
	\end{equation*}
	The following property holds in the set $\Gamma$ of probability one:
	\begin{equation*}
		\bar\by_*(t) \leq \bar\by_*^-(t) \quad \text{and} \quad \bar\bnu_t\left(\left[\bar\by_*(t), \bar\by_*^+(t)\right)\right) = 0 \quad \text{for all} \quad t \in \calC.
	\end{equation*}
	If $\bar\by_*^+(t) < \infty$, then the latter equality also holds for the closed interval $[\bar\by_*(t), \bar\by_*^+(t)]$.
\end{lemma}

\begin{proof}
	Fix $\omega \in \Gamma$ and $t \in \calC(\omega)$; for conciseness, we will omit $\omega$ from the notation in the remainder of the proof. Using Proposition \ref{prop: limiting measure of y-intervals} and the definition of $\bar\by_*(t)$, we obtain: 
	\begin{equation*}
		\lim_{N \to \infty} \bar\bnu_t^N\left([0, y]\right) = \bar\bnu_t\left([0, y]\right) < 1 \quad \text{for all} \quad y < \bar\by_*(t).
	\end{equation*}
	We conclude that for all $y < \bar\by_*(t)$ we have $y \leq \bar\by_*^N(t)$ for all large enough $N$, which in turn yields $y \leq \bar\by_*^-(t)$. This proves the first claim of the lemma, i.e., $\bar\by_*(t) \leq \bar\by_*^-(t)$.
	
	It only remains to prove that $\bar\bnu_t([\bar\by_*(t), \bar\by_*^+(t))) = 0$. It is clear that this holds trivially if $\bar\by_*(t) = \infty$ or $\bar\by_*^+(t) = 0$. Hence, assume that neither of these is the case. If $y < \bar\by_*^+(t)$, then along an infinite subsequence $y < \bar\by_*^N(t)$, and thus $\bar\bnu_t^N([0, y]) < 1$ . Therefore,
	\begin{equation*}
		\bar\bnu_t\left(\left[0, \bar\by_*^+(t)\right)\right) = \sup_{y < \bar\by_*^+(t)} \bar\bnu_t\left([0, y]\right) = \sup_{y < \bar\by_*^+(t)} \lim_{N \to \infty} \bar\bnu_t^N\left([0, y]\right) \leq 1,
	\end{equation*}
	where the limit exists by Proposition \ref{prop: limiting measure of y-intervals}. Furthermore, the assumption that $\bar\by_*(t) < \infty$ and Proposition \ref{prop: limiting measure of y-intervals} yield $\bar\bnu_t([0, \bar\by_*(t)]) = \bar\bnu_t([0, \bar\by_*(t)))$ and we obtain:
	\begin{equation*}
		1 \leq \inf_{y > \bar\by_*(t)} \bar\bnu_t\left([0, y]\right) = \bar\bnu_t\left([0, \bar\by_*(t)]\right) = \bar\bnu_t\left([0, \bar\by_*(t))\right) = \sup_{y < \bar\by_*(t)} \bar\bnu_t\left([0, y]\right) \leq 1.
	\end{equation*}
	Hence, $\bar\bnu_t([0, \bar\by_*(t))) = 1$, and it follows that
	\begin{equation*}
		\bar\bnu_t\left(\left[\bar\by_*(t), \bar\by_*^+(t)\right)\right) = \bar\bnu_t\left(\left[0, \bar\by_*^+(t)\right)\right) - \bar\bnu_t\left(\left[0, \bar\by_*(t)\right)\right) \leq 0;
	\end{equation*}
	Since the left-hand side is nonnegative, $\bar\bnu_t([\bar\by_*(t), \bar\by_*^+(t))) = 0$. Moreover, Proposition \ref{prop: limiting measure of y-intervals} implies that the closed interval has measure zero as well whenever $\bar\by_*^+(t)$ is finite.
\end{proof}

Using Proposition \ref{prop: limiting measure of y-intervals} and the above lemma, we will prove the following result.

\begin{proposition}
	\label{prop: limits of integrals involving r}
	The following property holds in the set $\Gamma$ of probability one:
	\begin{equation*}
		\lim_{N \to \infty} \angled{\phi \bar r\left(\bar\bmu_t^N\right)}{\bar\bmu_t^N} = \angled{\phi \bar r\left(\bar\bmu_t\right)}{\bar\bmu_t} \quad \text{for all} \quad \phi \in C_b\left(\R_+^2\right) \quad \text{and} \quad t \in \calC.
	\end{equation*}
\end{proposition}

\begin{proof}
	Fix $\omega \in \Gamma$, $\phi \in C_b(\R_+^2)$ and $t \in \calC(\omega)$. For conciseness, we will omit $\omega$ from the notation. Given $\varepsilon > 0$, there exist $\psi^1, \psi^2 \in C_b(\R_+^2)$ such that, for all $x, y \in \R_+$:
	\begin{align*}
		&\ind{y < \left[\bar\by_*(t) - 2\varepsilon\right]^+} \leq \psi^1(x, y) \leq \ind{y < \left[\bar\by_*(t) - \varepsilon\right]^+}, \\
		&\ind{y > \bar\by_*^+(t) + 2\varepsilon} \leq \psi^2(x, y) \leq \ind{y > \bar\by_*^+(t) + \varepsilon}.
	\end{align*}
	By Lemma \ref{lem: limiting threshold}, we have $[\bar\by_*(t) - \varepsilon]^+ \leq \bar\by_*^N(t) \leq \bar\by_*^+(t) + \varepsilon$ for all large enough $N$. Then
	\begin{equation*}
		\angled{\psi^1\phi \bar r\left(\bar\bmu_t^N\right)}{\bar\bmu_t^N} = \angled{\psi^1\phi}{\bar\bmu_t^N} \quad \text{and} \quad \angled{\psi^2\phi \bar r\left(\bar\bmu_t^N\right)}{\bar\bmu_t^N} = 0
	\end{equation*}
	for all sufficiently large $N$, and since $t \in \calC$, we conclude that
	\begin{equation*}
		\lim_{N \to \infty} \angled{\psi^1\phi \bar r\left(\bar\bmu_t^N\right)}{\bar\bmu_t^N} = \angled{\psi^1 \phi}{\bar\bmu_t} \quad \text{and} \quad \lim_{N \to \infty} \angled{\psi^2\phi \bar r\left(\bar\bmu_t^N\right)}{\bar\bmu_t^N} = 0. 
	\end{equation*}
	Moreover, it follows from Proposition \ref{prop: limiting measure of y-intervals} and Lemma \ref{lem: limiting threshold} that
	\begin{align*}
		\limsup_{N \to \infty} \left|\angled{[1 - \psi^1 - \psi^2]\phi \bar r\left(\bar\bmu_t^N\right)}{\bar\bmu_t^N}\right| &\leq \lim_{N \to \infty} \bar\bnu_t^N\left(\left[\left(\bar\by_*(t) - 2\varepsilon\right)^+, \bar\by_*^+(t) + 2\varepsilon\right)\right) \norm{\phi}_\infty \\
		&= \bar\bnu_t\left(\left[\left(\bar\by_*(t) - 2\varepsilon\right)^+, \bar\by_*^+(t) + 2\varepsilon\right)\right) \norm{\phi}_\infty \\
		&\leq \norm{\phi}_\infty \int_{\bar\by_*(t) - 2\varepsilon}^{\bar\by_*(t)} \theta(s - t)ds \\
		&+ \norm{\phi}_\infty \int_{\bar\by_*^+(t)}^{\bar\by_*^+(t) + 2\varepsilon} \theta(s - t)ds.
	\end{align*}
	Combining the three limits, we conclude that
	\begin{equation*}
		\limsup_{N \to \infty} \left|\angled{\phi \bar r\left(\bar\bmu_t^N\right)}{\bar\bmu_t^N} - \angled{\psi^1 \phi}{\bar\bmu_t}\right| \leq \norm{\phi}_\infty \left[\int_{\bar\by_*(t) - 2\varepsilon}^{\bar\by_*(t)} \theta(s - t)ds + \int_{\bar\by_*^+(t)}^{\bar\by_*^+(t) + 2\varepsilon} \theta(s - t)ds\right].
	\end{equation*}
	
	In addition, observe that
	\begin{equation*}
		\left|\angled{\phi \bar r\left(\bar\bmu_t\right)}{\bar\bmu_t} - \angled{\psi^1\phi}{\bar\bmu_t}\right| \leq \bar\bnu_t\left(\left[\left(\bar\by_*(t) - 2\varepsilon\right)^+, \bar\by_*(t)\right)\right) \norm{\phi}_\infty \leq \norm{\phi}_\infty \int_{\bar\by_*(t) - 2\varepsilon}^{\bar\by_*(t)} \theta(s - t)ds
	\end{equation*}
	by Proposition \ref{prop: limiting measure of y-intervals}. As a result, we obtain:
	\begin{align*}
		\limsup_{N \to \infty} \left|\angled{\phi \bar r\left(\bar\bmu_t^N\right)}{\bar\bmu_t^N} - \angled{\phi \bar r\left(\bar\bmu_t\right)}{\bar\bmu_t}\right| &\leq 2 \norm{\phi}_\infty \int_{\bar\by_*(t) - 2\varepsilon}^{\bar\by_*(t)} \theta(s - t)ds \\
		&+ \norm{\phi}_\infty \int_{\bar\by_*^+(t)}^{\bar\by_*^+(t) + 2\varepsilon} \theta(s - t)ds,
	\end{align*}
	and since $\varepsilon > 0$ is arbitrary, it follows that the left-hand side is zero.
\end{proof}

The above proposition leads to the following corollary, which will allow to characterize the limit of the fourth term on the right-hand side of the pre-limit equation \eqref{eq: pre-limit equation in normalized notation}.

\begin{corollary}
	\label{cor: limit of partial derivatives term}
	The following property holds in the set $\Gamma$ of probability one:
	\begin{align*}
		\lim_{N \to \infty} \int_0^t \angled{\varphi_s + \bar r\left(\bar\bmu_s^N\right)\psi_s}{\bar\bmu_s^N}ds = \int_0^t \angled{\varphi_s + \bar r\left(\bar\bmu_s\right)\psi_s}{\bar\bmu_s}ds
	\end{align*}
	for all $\varphi, \psi \in C_b(\R_+^3)$ and all $t \geq 0$.
\end{corollary}

\begin{proof}
	Fix a sample path $\omega \in \Gamma$, functions $\varphi, \psi \in C_b(\R_+^3)$ and $t \geq 0$. For conciseness, we will omit $\omega$ from the notation in the sequel. It follows from Proposition \ref{prop: limits of integrals involving r} that
	\begin{equation*}
		\lim_{N \to \infty} \angled{\varphi_s + \bar r\left(\bar\bmu_s^N\right)\psi_s}{\bar\bmu_s^N} = \angled{\varphi_s + \bar r\left(\bar\bmu_s\right)\psi_s}{\bar\bmu_s} \quad \text{for all} \quad s \in \calC.
	\end{equation*}
	Furthermore, observe that $\bar\bmu_s^N(\R_+^2) \leq \bar\bmu_0^N(\R_+^2) + \bar\calR^N(t) \leq \bar\bmu_0(\R_+^2) + \lambda_0 t + 1$ for all $s \in [0, t]$ and all large enough $N$. As a result, we obtain:
	\begin{equation*}
		\left|\angled{\varphi_s + \bar r\left(\bar\bmu_s^N\right)\psi_s}{\bar\bmu_s^N}\right| \leq \left[\bar\bmu_0\left(\R_+^2\right) + \lambda_0 t + 1\right]\left(\norm{\varphi}_\infty + \norm{\psi}_\infty\right) \quad \text{for all} \quad s \in [0, t]
	\end{equation*}
	and all sufficiently large $N$. Because the complement of $\calC$ is countable, the claim follows from an application of the bounded convergence theorem.
\end{proof}

\subsubsection{Limits involving hazard rates}
\label{subsub: limits involving hazard rates}

While the previous section concerned integrals involving the service rate function $\bar r^N(\bar\bmu)$, this section considers integrals involving both the service rate function and the hazard rate vector field $h_g$. A characterization of how these integrals behave in the limit will allow to characterize the limit of $\bar\bD^{N, c}$, and thereby the limit of $\bar\bD^N$ by Proposition \ref{prop: set of probability one}. Then we will be able to obtain the limit of the right-hand side of the pre-limit equation \eqref{eq: pre-limit equation in normalized notation}.

First, we prove, in Appendix \ref{app: proofs of auxiliary results}, the following lemma, based on \cite[Theorem A.3.12]{dupuis2011weak}.

\begin{restatable}{lemma}{lemmafourtwentysix}
	\label{lem: limit inferior of integrals involving r}
	The following properties hold  in the set $\Gamma$ of probability one:
	\begin{equation*}
		\angled{\ell}{\bar\bmu_t} \leq \liminf_{N \to \infty} \angled{\ell}{\bar\bmu_t^N} \quad \text{and} \quad \angled{\ell \bar r\left(\bar\bmu_t\right)}{\bar\bmu_t} \leq \liminf_{N \to \infty} \angled{\ell \bar r\left(\bar\bmu_t^N\right)}{\bar\bmu_t^N}
	\end{equation*}
	for all $t \in \calC$ and all lower semicontinuous functions $\map{\ell}{\R_+^2}{[0, \infty)}$.
\end{restatable}

The next lemma proves that $\bar\bmu$ satisfies (c) of Definition \ref{def: fluid problem}.

\begin{lemma}
	\label{lem: boundedness of hazard rate integral in the limit}
	The following property holds with probability one:
	\begin{equation*}
		\int_0^t \angled{\bar r\left(\bar\bmu_s\right) h_g^x + h_g^y}{\bar\bmu_s} < \infty \quad \text{for all} \quad t \geq 0.
	\end{equation*}
\end{lemma}

\begin{proof}
	Choose $\varepsilon > 0$ and then fix a constant $K_\varepsilon \geq 0$, compact sets $A_\varepsilon \subset B_\varepsilon \subset U_g$ and a function $\psi$ as in the statement of Lemma \ref{lem: bounds for integrals of hazard rate}. Then observe that
	\begin{equation*}
		\expect*{\int_0^t \angled{\bar r\left(\bar\bmu_s^N\right)h_g^x + h_g^y}{\bar\bmu_s^N}ds} = \expect*{\bar\bD_t^{N, c}(\psi)} + \expect*{\bar\bD_t^{N, c}(1 - \psi)}.
	\end{equation*}
	By Lemma \ref{lem: bounds for integrals of hazard rate}, we have $E[\bar\bD_t^{N, c}(1 - \psi)] \leq \varepsilon (\bar M + E[\bar\calR^N(t)])$. Furthermore,
	\begin{align*}
		\expect*{\bar\bD_t^{N, c}(\psi)} &\leq \expect*{\int_0^t \int_{B_\varepsilon} \left[\bar r\left(\bar\bmu_s^N, y\right)h_g^x(x, y) + h_g^y(x, y)\right] \bar\bmu_s^N(dx, dy)ds} \\
		&\leq K_\varepsilon t \expect*{\sup_{s \in [0, t]} \bar\bmu_s^N\left(\R_+^2\right)} \leq K_\varepsilon t \frac{\expect*{M^N + \calR^N(t)}}{N} \leq K_\varepsilon t \left(\bar M + \expect*{\bar\calR^N(t)}\right).
	\end{align*}
	
	It follows from Assumption \ref{ass: standing assumptions 2}, Lemma \ref{lem: limit inferior of integrals involving r} and Fatou's lemma that
	\begin{align*}
		\int_0^t \angled{\bar r\left(\bar\bmu_s\right)h_g^x + h_g^y}{\bar\bmu_s}ds &\leq \int_0^t \liminf_{N \to \infty} \angled{\bar r\left(\bar\bmu_s^N\right)h_g^x + h_g^y}{\bar\bmu_s^N}ds \\
		&\leq \liminf_{N \to \infty} \int_0^t \angled{\bar r\left(\bar\bmu_s^N\right)h_g^x + h_g^y}{\bar\bmu_s^N}ds
	\end{align*}
	with probability one; note that the hazard rates $h_g^x$ and $h_g^y$ are nonnegative in $U_g$ and thus lower semicontinuous in $\R_+^2$. Another application of Fatou's lemma yields
	\begin{align*}
		\expect*{\int_0^t \angled{\bar r\left(\bar\bmu_s\right)h_g^x + h_g^y}{\bar\bmu_s}ds} \leq \liminf_{N \to \infty} \left(K_\varepsilon t + \varepsilon\right) \left(\bar M + \expect*{\bar\calR^N(t)}\right).
	\end{align*}
	By the elementary renewal theorem, the right-hand side equals $(K_\varepsilon t + \varepsilon) (\bar M + \lambda_0 t) < \infty$. Then the expression inside the expectation sign is almost surely finite.
\end{proof}

Next we characterize the limiting integrals $\bar\bD_t(\varphi)$ for a suitable class of test functions.

\begin{proposition}
	\label{prop: limit of hazard rate term}
	Given any Lipschitz function $\varphi \in C_b(\R_+^3)$, we have:
	\begin{equation*}
		\bar\bD_t(\varphi) = \int_0^t \angled{\varphi_s\left[\bar r\left(\bar\bmu_s\right) h_g^x + h_g^y\right]}{\bar\bmu_s}ds
	\end{equation*}
	for all $t \geq 0$ with probability one.
\end{proposition}

\begin{proof}
	Let $d_T^f$ denote the metric defined in the proof of Lemma \ref{lem: convergence measure-valued functions implies convergence of projections} for the Skorohod-$J_1$ topology in the space $D_\R[0, T]$, and for notational conciseness, let us write
	\begin{equation*}
		I_t\left(\varphi, h_g, \bar\bmu\right) \defeq \angled{\varphi_t\left[\bar r\left(\bar\bmu_t\right)h_g^x + h_g^y\right]}{\bar\bmu_t} \quad \text{for all} \quad t \geq 0,
	\end{equation*}
	with analogous definitions for other arguments than $\varphi$, $h_g$ and $\bar\bmu$. In addition, let
	\begin{equation*}
		\bY_t(\varphi) \defeq \int_0^t I_s\left(\varphi, h_g, \bar\bmu\right)ds \quad \text{for all} \quad t \geq 0.
	\end{equation*}
	
	We will prove that $E[d_T^f(\bar\bD(\varphi), \bY(\varphi))] = 0$ for all $T \geq 0$. Then $\bar\bD_t(\varphi) = \bY_t(\varphi)$ for all $t \in [0, T]$ with probability one, and thus the equality holds almost surely also for all $t \geq 0$. In order to establish the former property, observe that:
	\begin{align*}
		\expect*{d_T^f(\bar\bD(\varphi), \bY(\varphi))} &\leq \liminf_{N \to \infty} \expect*{d_T^f\left(\bar\bD^{N, c}(\varphi), \bY(\varphi)\right)} \\
		&\leq \liminf_{N \to \infty} \expect*{\sup_{t \in [0, T]} \left|\bar\bD_t^{N, c}(\varphi) - \bY_t(\varphi)\right|};
	\end{align*}
	the first inequality holds since $\bar\bD^{N, c} \to \bar\bD$ in $D_{\calM_F^+(\R_+^3)}[0, \infty)$ almost surely as $N \to \infty$, and follows from Lemma \ref{lem: convergence measure-valued functions implies convergence of projections} and Fatou's lemma. Therefore, it suffices to show that
	\begin{equation}
		\label{eq: liminf of expectation of supremum of difference of integrals}
		\liminf_{N \to \infty} \expect*{\sup_{t \in [0, T]} \left|\bar\bD_t^{N, c}(\varphi) - \bY_t(\varphi)\right|} = 0.
	\end{equation}
	
	Given $\varepsilon > 0$, fix $K_\varepsilon \geq 0$, compact sets $A_\varepsilon \subset B_\varepsilon \subset U_g$ and $\psi$ as in Lemma \ref{lem: bounds for integrals of hazard rate}. Then
	\begin{equation}
		\label{eq: bound for supremum over time of hazard rate integrals}
		\begin{split}
			\sup_{t \in [0, T]} \left|\bar\bD_t^{N, c}(\varphi) - \bY_t(\varphi)\right| &= \sup_{t \in [0, T]} \left|\int_0^t I_s\left(\varphi, h_g, \bar\bmu^N\right)ds - \int_0^t I_s\left(\varphi, h_g, \bar\bmu\right)ds\right| \\
			&\leq \int_0^T I_t\left((1 - \psi)|\varphi|, h_g, \bar\bmu^N\right)dt \\
			&+ \int_0^T I_t\left((1 - \psi)|\varphi|, h_g, \bar\bmu\right)dt \\
			&+ \int_0^T\left|I_t\left(\psi\varphi, h_g, \bar\bmu^N\right) - I_t\left(\psi\varphi, h_g, \bar\bmu\right)\right|dt.
		\end{split}
	\end{equation}
	We proceed to bound each of the terms on the right-hand side.
	
	The first term on the right-hand side satisfies
	\begin{equation*}
		E\left[\int_0^T I_t\left((1 - \psi)|\varphi|, h_g, \bar\bmu^N\right)dt\right] = E\left[\bar\bD_T^{N, c}\left((1 - \psi)|\varphi|\right)\right] \leq \varepsilon\left(\bar M + \expect*{\bar\calR^N(T)}\right) \norm{\varphi}_\infty
	\end{equation*}
	for all $N \geq 1$ by Lemma \ref{lem: bounds for integrals of hazard rate}. Moreover, the integrand of second term is such that
	\begin{equation*}
		I_t\left((1 - \psi)|\varphi|, h_g, \bar\bmu\right) \leq \liminf_{N \to \infty} I_t\left((1 - \psi)|\varphi|, h_g, \bar\bmu^N\right) \quad \text{for all} \quad t \in \calC 
	\end{equation*}
	almost surely, by Assumption \ref{ass: standing assumptions 2} and Lemma \ref{lem: limit inferior of integrals involving r}. As a result, Fatou's lemma yields
	\begin{equation*}
		E\left[\int_0^T I_t\left((1 - \psi)|\varphi|, h_g, \bar\bmu\right)dt\right] \leq \liminf_{N \to \infty} E\left[\int_0^T I_t\left((1 - \psi)|\varphi|, h_g, \bar\bmu^N\right)dt\right].
	\end{equation*}
	Therefore, the elementary renewal theorem implies that:
	\begin{equation*}
		\liminf_{N \to \infty} E\left[\int_0^T \left[I_t\left((1 - \psi)|\varphi|, h_g, \bar\bmu^N\right) + I_t\left((1 - \psi)|\varphi|, h_g, \bar\bmu\right)\right]dt\right] \leq 2\varepsilon\left(\bar M + \lambda_0 T\right)\norm{\varphi}_\infty.
	\end{equation*}
	
	For the last term on the right-hand side of \eqref{eq: bound for supremum over time of hazard rate integrals}, note that $\psi\varphi h_g^x, \psi\varphi h_g^y \in C_b(\R_+^3)$ since $h_g$ is continuous in $U_g$ by Assumption \ref{ass: standing assumptions 2} and $\psi$ is supported in $B_\varepsilon$. Then
	\begin{equation*}
		\lim_{N \to \infty} \left|I_t\left(\psi\varphi, h_g, \bar\bmu^N\right) - I_t\left(\psi\varphi, h_g, \bar \bmu\right)\right| = 0 \quad \text{for all} \quad t \in \calC
	\end{equation*}
	with probability one by Corollary \ref{cor: limit of partial derivatives term}. Since $\bar\bmu_t^N(\R_+^2) \leq \bar\bmu_0^N(\R_+^2) + \bar\calR^N(T)$ for all $t \in [0, T]$ surely, we have $\bar\bmu_t(\R_+^2) \leq \bar\bmu_0(\R_+^2) + \lambda_0 T$ for all $t \in [0, T] \cap \calC$ almost surely. Hence,
	\begin{align*}
		\left|I_t\left(\psi\varphi, h_g, \bar\bmu^N\right) - I_t\left(\psi\varphi, h_g, \bar \bmu\right)\right| &\leq K_\varepsilon \norm{\varphi}_\infty \left[\bar\bmu_t^N\left(\R_+^2\right) + \bar\bmu_t\left(\R_+^2\right)\right] \\
		&\leq K_\varepsilon \norm{\varphi}_\infty\left[\bar\bmu_0^N\left(\R_+^2\right) + \bar\calR^N(T) + \bar\bmu_0\left(\R_+^2\right) + \lambda_0 T\right] \\
		&\leq K_\varepsilon \norm{\varphi}_\infty\left[2\bar\bmu_0\left(\R_+^2\right) + 2\lambda_0 T + 1\right]
	\end{align*}
	for all $t \in [0, T] \cap \calC$ and all sufficiently large $N$ almost surely. Here the first inequality holds surely and the third inequality uses that $\bar\bmu_0^N(\R_+^2) + \bar\calR^N(T) \to \bar\bmu_0(\R_+^2) + \lambda_0 T$ as $N \to \infty$ almost surely. The right-hand side of the above inequality is integrable because
	\begin{align*}
		\expect*{\bar\bmu_0\left(\R_+^2\right)} &= \lim_{L \to \infty} \expect*{\bar\bmu_0\left(\R_+^2\right) \ind{\bar\bmu_0\left(\R_+^2\right) \leq L} + L \ind{\bar\bmu_0\left(\R_+^2\right) > L}} \\
		&= \lim_{L \to \infty} \lim_{N \to \infty} \expect*{\bar\bmu_0^N\left(\R_+^2\right) \ind{\bar\bmu_0^N\left(\R_+^2\right) \leq L} + L \ind{\bar\bmu_0^N\left(\R_+^2\right) > L}} \leq \bar M
	\end{align*}
	The limits use the monotone convergence theorem and the fact that $\bar\bmu_0^N(\R_+^2) \Rightarrow \bar\bmu_0(\R_+^2)$ as $N \to \infty$, whereas the inequality follows from $\bar\bmu_0^N(\R_+^2) = M^N / N$ and Assumption \ref{ass: standing assumptions 1}. Then the dominated convergence theorem implies that
	\begin{equation*}
		\lim_{N \to \infty} E\left[\int_0^T\left|I_t\left(\psi\varphi, h_g, \bar\bmu^N\right) - I_t\left(\psi\varphi, h_g, \bar\bmu\right)\right|dt\right] = 0.
	\end{equation*}
	
	Taking the expectation and then the limit inferior as $N \to \infty$ in \eqref{eq: bound for supremum over time of hazard rate integrals}, we obtain:
	\begin{equation*}
		\liminf_{N \to \infty} E\left[\sup_{t \in [0, T]} \left|\bar\bD_t^{N, c}(\varphi) - \bY_t(\varphi)\right|\right] \leq 2\varepsilon\left(\bar M + \lambda_0 T\right)\norm{\varphi}_\infty.
	\end{equation*}
	Because $\varepsilon > 0$ is arbitrary, the left-hand side equals zero, which proves \eqref{eq: liminf of expectation of supremum of difference of integrals}.
\end{proof}

As a corollary, we fully characterize the limit $\bar\bD$.

\begin{corollary}
	\label{cor: limit of hazard rate term}
	The following property holds with probability one:
	\begin{equation}
		\label{eq: limit of departure process}
		\bar \bD_t(\varphi) = \int_0^t \angled{\varphi_s\left[\bar r\left(\bar\bmu_s\right) h_g^x + h_g^y\right]}{\bar\bmu_s}ds
	\end{equation}
	for all $t \geq 0$ and $\varphi \in C_b(\R_+^3)$.
\end{corollary}

\begin{proof}
	
	As in the proof of Lemma \ref{lem: departure process and compensator have same limit}, we may construct a countable set $\Psi \subset C_c^1(\R_+^3)$ that is dense in $C_c(\R_+^3)$. It follows from Proposition \ref{prop: limit of hazard rate term} that there exists a set of probability one where Lemma \ref{lem: boundedness of hazard rate integral in the limit} holds and \eqref{eq: limit of departure process} holds simultaneously for all $\psi \in \Psi$ and the function $\indc$ that is identically equal to one. Fix any $\omega \in \Omega$ in this set of probability one; in the remainder of the proof we omit $\omega$ from the notation for conciseness.
	
	For each $\varphi \in C_c(\R_+^3)$, consider a sequence $\setc{\psi^k}{k \geq 1} \subset \Psi$ that converges to $\varphi$, and let $\bY_t(\varphi)$ denote the right-hand side of \eqref{eq: limit of departure process}, which is well-defined by Lemma \ref{lem: boundedness of hazard rate integral in the limit}. Then
	\begin{align*}
		&\left|\bar\bD_t(\varphi) - \bar\bD_t\left(\psi^k\right)\right| \leq \norm{\varphi - \psi^k}_\infty \bar\bD_t(\indc), \\
		&\left|\bY_t(\varphi) - \bY_t\left(\psi^k\right)\right| \leq \norm{\varphi - \psi^k}_\infty \bY_t(\indc),
	\end{align*}
	for all $t \geq 0$. Furthermore, $\bar\bD_t(\psi^k) = \bY_t(\psi^k)$ and $\bar\bD_t(\indc) = \bY_t(\indc) < \infty$ for all $t \geq 0$ by Proposition \ref{prop: limit of hazard rate term} and Lemma \ref{lem: boundedness of hazard rate integral in the limit}. Then \eqref{eq: limit of departure process} is obtained by taking the limit as $k \to \infty$. Further, this identity in fact holds for all $\varphi \in C_b(\R_+^3)$ because $C_c(\R_+^3)$ is separating.
\end{proof}

\subsubsection{Derivation of the transport equation}
\label{subsub: derivation of the transport equation}

Combining the results of Sections \ref{subsub: limits involving service rates} and \ref{subsub: limits involving hazard rates}, the following theorem derives the limit of the right-hand side of the pre-limit equation \eqref{eq: pre-limit equation in normalized notation}. In particular, it proves that $\bar\bmu$ satisfies (c) and (d) of Definition \ref{def: fluid problem}, thereby completing the proof of Theorem \ref{the: measure-valued fluid limit}.

\begin{theorem}
	\label{the: characterization of subsequential limits}
	The following property holds with probability one:
	\begin{equation*}
		\int_0^t \angled{\bar r\left(\bar\bmu_s\right) h_g^x + h_g^y}{\bar\bmu_s}ds < \infty \quad \text{for all} \quad t \geq 0.
	\end{equation*}
	Furthermore, for all $t \geq 0$ and $\varphi \in C_c^1(\R_+^3)$, we have:
	\begin{align*}
		\angled{\varphi_t}{\bar\bmu_t} &= \angled{\varphi_0}{\bar\bmu_0} + \lambda_0 \int_0^t \varphi_s(0, 0)ds - \int_0^t \angled{\varphi_s\left[\bar r(\bar \bmu_s)h_g^x + h_g^y\right]}{\bar\bmu_s}ds \\
		&+ \int_0^t \angled{\partial_t\varphi_s + \bar r\left(\bar \bmu_s\right)\partial_x\varphi_s + \partial_y\varphi_s}{\bar\bmu_s}ds.
	\end{align*}
\end{theorem}

\begin{proof}
	The first property was established in Lemma \ref{lem: boundedness of hazard rate integral in the limit}, so it only remains to prove the second property. Let us introduce the short-hand notation
	\begin{equation*}
		\bar\bY_t^N(\varphi) \defeq \int_0^t \angled{\partial_t\varphi_s + \bar r\left(\bar\bmu_s^N\right)\partial_x\varphi_s + \partial_y\varphi_s}{\bar\bmu_s^N}ds
	\end{equation*}
	for all $t \geq 0$ and $\varphi \in C_c^1(\R_+^3)$. Fix $\omega \in \Omega$ in the set of probability one where Proposition~\ref{prop: set of probability one} and Corollaries \ref{cor: limit of partial derivatives term} and \ref{cor: limit of hazard rate term} hold and fix $\varphi \in C_c^1(\R_+^3)$; we will omit $\omega$ from the notation.
	
	Using the notation $\bar\bY_t^N(\varphi)$, it follows from Proposition \ref{prop: arrivals-flow-departures decomposition} that
	\begin{equation*}
		\angled{\varphi_t}{\bar\bmu_t^N} = \angled{\varphi_0}{\bar\bmu_0^N} + \bar\bA_t^N(\varphi) - \bar\bD_t^N(\varphi) + \bar\bY_t^N(\varphi) \quad \text{for all} \quad t \geq 0 \quad \text{and} \quad N \geq 1.
	\end{equation*}
	Lemma \ref{lem: convergence measure-valued functions implies convergence of projections} and Proposition \ref{prop: set of probability one} imply that the term on the left-hand side and the first three terms on the right-hand side converge in $D_\R[0, \infty)$. Therefore, the processes $\bar\bY^N(\varphi)$ converge in $D_\R[0, \infty)$ to some process $\bar\bY(\varphi)$, and it follows from Corollary \ref{cor: limit of partial derivatives term} that
	\begin{equation*}
		\bar\bY_t(\varphi) = \int_0^t \angled{\partial_t\varphi_s + \bar r\left(\bar\bmu_s\right)\partial_x\varphi_s + \partial_y\varphi_s}{\bar\bmu_s}ds
	\end{equation*}
	for all continuity points of $\bar\bY(\varphi)$. Recall that $\bar\bmu_t(\R_+^2) \leq \bar\bmu_0(\R_+^2) + \lambda_0 t$ for all but countably many $t \geq 0$. As a result, the inner integral on the right-hand is bounded, which implies that the right-hand side is continuous as a function of $t$. Because $\bar\bY(\varphi)$ is c\`adl\`ag, we conclude that the equality in fact holds for all $t \geq 0$.
	
	Using Lemma \ref{lem: convergence measure-valued functions implies convergence of projections} and Proposition \ref{prop: set of probability one} again, we get:
	\begin{equation*}
		\angled{\varphi_t}{\bar\bmu_t} = \angled{\varphi_0}{\bar\bmu_0} + \lambda_0\int_0^t \varphi_s(0, 0)ds - \bar\bD_t(\varphi) + \bar\bY_t(\varphi)
	\end{equation*}
	for each $t \geq 0$ at which the process on left-hand side and each of the processes on the right-hand side is continuous. Since all these processes are c\`adl\`ag, the equality holds for all times $t \geq 0$. The limit $\bar\bD(\varphi)$ is given in Corollary \ref{cor: limit of hazard rate term}, so this completes the proof.
\end{proof}

\section{Analysis of the fluid problem}
\label{sec: analysis of the fluid problem}

In this section we study the space $\calS^+$ of measure-valued functions that solve the fluid problem in Definition \ref{def: fluid problem}, proving the theorems stated in Section \ref{sub: properties of the fluid problem}. By Theorem \ref{the: measure-valued fluid limit}, the fluid limits of partial service queues with the preemptive LCFS policy are solutions of the fluid problem. Hence, the interpretation of $\bmu \in \calS^+$ is that $\bmu_t$ represents the distribution of the time in the system and attained service for a continuous population of tasks.

Recall that if $\bmu \in \calS^+$, then $\bmu_t(U_g^c) = 0$ and its $y$-marginal is absolutely continuous with respect to the Lebesgue measure for all times $t \geq 0$. In addition, $\bmu$ solves the transport equation \eqref{eq: transport equation}. The term $\bar r(\bmu_t, y)$ on the right-hand side of \eqref{eq: transport equation} is discontinuous in the spatial variable $y$ and possibly in the time variable $t$ as well. Furthermore, $\bar r(\bmu)$ constitutes a nonlinear feedback term in \eqref{eq: transport equation}, as it depends on the unknown measure-valued function $\bmu$. The theory started by DiPerna and Lions in \cite{diperna1989ordinary}, which was further developed by Ambrosio in \cite{ambrosio2004transport}, can be used to analyze certain transport equations with discontinuous velocity fields. Nevertheless, this theory applies to partial differential equations and not directly to integral measure-valued equations like \eqref{eq: transport equation}, where in particular, solutions will be singular with respect to Lebesgue. Moreover, we lack a framework for analyzing transport equations with nonlinear feedback terms, and basic properties, such as the existence and uniqueness of solutions, are difficult to prove or may even not hold.

\subsection{Outline of the analysis}
\label{sub: outline of the analysis of the transport equation}

Despite the aforementioned challenges, we develop in what follows a special-purpose analysis of the fluid problem that enables substantive conclusions.

In Section \ref{sub: existence of solutions} we cover existence of solutions for a broad class of initial conditions, a direct consequence of the fluid limit proved in Theorem \ref{the: measure-valued fluid limit}. In particular, we show that a broad class of initial conditions $\bmu_0$ are the limits of measures describing the normalized initial states $\bar\bmu_0^N$ of partial service queues with a growing number of servers $N$. Then the fluid limit result for the associated measure-valued processes yields the existence of a solution to the fluid problem which has initial condition $\bmu_0$.

While we do not prove that solutions of the fluid problem are unique in general, we do provide a fairly complete characterization of solutions. Section \ref{sub: solution with null initial condition} establishes uniqueness for the null initial condition and gives a closed-form expression, proving Theorem \ref{the: solution of transport equation for null initial condition}. For general initial conditions, Section \ref{sub: structure of solutions} proves Theorem \ref{the: structure of solutions}, decomposing the general solution as the superposition of the unique solution with null initial condition and a measure-valued function representing tasks initially present; this decomposition is nontrivial given the underlying nonlinearity. Finally, Section \ref{sub: long-term behavior of solutions} shows that the latter measure-valued function converges to zero over time, and thereby establishes that the fluid problem has a global attractor which admits a closed-form expression, proving Theorem \ref{the: long-term behavior of solutions}.

We now describe from a high-level the strategy underlying the analysis that sustains these results. Section \ref{sub: elementary properties} covers some elementary technical properties of the fluid solution. Then a key step is taken in Section \ref{sub: prescribed thresholds}, where we consider a modified fluid problem in which the feedback loop is opened and replaced by an exogenous function. Specifically, the term $\bar r(\bmu_t, y) = \ind{y \leq \bar y_*(\bmu_t)}$ is replaced by $\ind{y \leq \bar \by_t}$ for some prescribed threshold function $\bar \by$. For such a modified fluid problem, we first carry out in Proposition \ref{prop: spatial rescaling} a spatial rescaling that simplifies the right-hand side of the transport equation \eqref{eq: transport equation}, and leads to a solution strategy based on a partial differential equation that may be solved using characteristic curves. Through this technique, Proposition \ref{prop: integrals below the threshold} obtains explicit expressions for the integral $\langle\phi, \bmu_t\rangle$ when $\phi$ is supported below the threshold function. Proposition \ref{prop: transport equation with constant threshold} provides expressions for this integral when $\phi$ is general and the threshold function is constant over time. These formulas are key enabling tools in the later proofs of the main theorems.

In particular, Section \ref{sub: solution with null initial condition} shows that if $\bmu \in \calS^+$ and $\bmu_0$ is the null measure, then
\begin{equation*}
	\bar y_*(\bmu_t) = \infty \quad \text{for all} \quad t < y_* \quad \text{and} \quad \bar y_*(\bmu_t) = y_* \quad \text{for all} \quad t \geq y_*,
\end{equation*}
where $y_*$ is the possibly infinite constant in \eqref{eq: equilibrium threshold}. The proof is based on the fact that the initially empty system has an infinite threshold for some positive amount of time. Hence, we can use the formulas $\langle\phi, \bmu_t\rangle$ obtained in Section \ref{sub: prescribed thresholds}, for functions $\phi$ supported below the threshold, to establish that $y_*$ is the first time at which $\bmu_t(\R_+^2) = 1$. Additional arguments are required to show that the threshold is constant over the time interval $[y^*, \infty)$. Once this is done, the expressions of Proposition \ref{prop: transport equation with constant threshold} for solutions with a constant threshold are used to obtain the full solution of the fluid problem, proving Theorem \ref{the: solution of transport equation for null initial condition}.

In the case of general initial conditions, we show in Section \ref{sub: structure of solutions} that if $\bmu \in \calS^+$, then
\begin{equation*}
	\bar y_*(\bmu_t) \geq t \wedge y_* \quad \text{for all} \quad t < y_* \quad \text{and} \quad \bar y_*(\bmu_t) = y_* \quad \text{for all} \quad t \geq y_*;
\end{equation*}
this again requires a refined analysis based on the explicit formulas of Section \ref{sub: prescribed thresholds}.

Given any solution $\bmu \in \calS^+$, we let $\bnu \in \calS^+$ be the solution with null initial condition and let $\bxi \defeq \bmu - \bnu$, as in Theorem \ref{the: structure of solutions}. Using the above characterizations of the threshold functions $\bar y_*(\bmu)$ and $\bar y_*(\bnu)$, we show that the nonlinear feedback terms $\bar r(\bmu_t)$ and $\bar r(\bnu_t)$ are equal almost everywhere with respect to $\bnu_t$ for all $t \geq 0$. Hence, it is possible to replace $\bar r(\bnu)$ by $\bar r(\bmu)$ in the transport equation \eqref{eq: transport equation} for $\bnu$. By doing this and subtracting the equation for $\bnu$ from the equation for $\bmu$, we get an equation similar to \eqref{eq: transport equation} for $\bxi$; this equation depends on $\bar r(\bmu)$ and has arrival rate equal to zero. The interpretation is that $\bxi$ describes the times in the system and attained service times of tasks that were already present at time zero. In principle, $\bxi = \bmu - \bnu$ could be a signed measure, which would not be consistent with the latter interpretation. However, using results from Section \ref{sub: prescribed thresholds}, we show that $\bxi$ is nonnegative, completing the proof of Theorem \ref{the: structure of solutions}. We note that for this reason, Section \ref{sub: prescribed thresholds} considers the more general setting of possibly signed measures.

We conclude our analysis of the fluid problem in Section \ref{sub: long-term behavior of solutions}, where we prove that the measure-valued function $\bxi$ defined above converges vaguely to the null measure over time, and also weakly under some mild additional assumptions. It follows that the long-term limit of any solution coincides with the long-term limit of the solution with null initial condition. Using the closed-form expression for the solution with the null initial condition obtained in Theorem \ref{the: solution of transport equation for null initial condition}, we compute this limit explicitly, proving Theorem \ref{the: long-term behavior of solutions}.

\subsection{Existence of solutions}
\label{sub: existence of solutions}

The fact that \eqref{eq: transport equation} arises from the fluid limit in Theorem \ref{the: measure-valued fluid limit} implies the existence of solutions to the fluid problem for a broad class of initial conditions.

\begin{proposition}
	\label{prop: existence of solutions}
	Let $\rho \in \calM_F^+(\R_+^2)$ be a probability measure that is stochastically smaller than the probability measure with density $g$. Specifically, 
	\begin{equation*}
		\angled{\phi}{\rho} \leq \int_{\R_+^2} \phi(x, y)g(x, y)dxdy
	\end{equation*}
	for each Borel $\map{\phi}{\R_+^2}{\R}$ that is nondecreasing in the componentwise order and bounded. Also, suppose that $\rho$ satisfies (b) of Definition \ref{def: fluid problem}. Then, for each constant $\alpha \geq 0$, there exists a solution $\bmu$ of the fluid problem with initial condition $\bmu_0 = \alpha \rho$. If the variables $x$ and $y$ represent the attained service and time spent in the system for tasks present at time zero, respectively, then the interpretation of the probability measure $\rho$ is that $\rho(A)$ represents the fraction of tasks present at time zero for which $(x, y) \in A$.
\end{proposition}

\begin{proof}
	By Theorem \ref{the: measure-valued fluid limit}, the claim will follow if we construct a sequence of preemptive LCFS partial service queues $\bmu^N$ satisfying Assumptions \ref{ass: standing assumptions 1} and \ref{ass: standing assumptions 2}, such that the normalized initial states $\bar\bmu_0^N$ converge weakly to the random measure almost surely equal to $\alpha \rho$.
	
	For carrying out this construction, we invoke Strassen's theorem; see \cite{strassen1965existence,lindvall1999strassen}. It gives random vectors $(b_i, c_i)$ and $(x_i, y_i)$ which are distributed according to $g$ and $\rho$, respectively, and in addition satisfy that $x_i \leq b_i$ and $y_i \leq c_i$ surely. We consider independent copies of the vectors $(b_i, c_i, x_i, y_i)$ defined on the same probability space for all $i \leq 0$.
	
	We interpret the random vector $(x_i, y_i)$ as the attained service and time spent in the system of a task already present at time zero, with $1 - \floor{\alpha N} \leq i \leq 0$ the index of the task and $M^N = \floor{\alpha N}$ the number of tasks present at time zero. In addition, we let $(b_i, c_i)$ be the service requirement and patience of task $i$. By the construction based on Strassen's theorem, these quantities are larger than the attained service and time spent in the system, as assumed in Section \ref{sec: model description}. The initial state of the system is given by:
	\begin{equation*}
		\bmu_0^N(dx, dy) = \sum_{i = 1 - M^N}^0 \delta_{(x_i, y_i)}(dx, dy).
	\end{equation*}
	
	Next we check those conditions in Assumption \ref{ass: standing assumptions 1} which pertain to the initial state of the system for the initial states $\bmu_0^N$ defined in this way. Assumptions \ref{ass: standing assumptions 1} and \ref{ass: standing assumptions 2} also include conditions involving the joint distribution of the service requirement and patience, which are standing assumptions, and conditions involving the arrival process, which are satisfied, for instance, if the renewal process $\calR_0$ is a Poisson process of intensity $\lambda_0$.
	
	Condition (a) of Assumption \ref{ass: standing assumptions 1} holds trivially because $M^N \leq \alpha N$ surely. For (c), note that the  Glivenko-Cantelli theorem implies that
	\begin{equation*}
		\lim_{N \to \infty} \bar\bmu_0^N(\R_+ \times [a, b]) = \lim_{N \to \infty} \frac{M^N}{N}\frac{1}{M^N}\sum_{i = 1 - M^N}^0 \ind{y_i \in [a, b]} = \alpha \rho(\R_+ \times [a, b])
	\end{equation*}
	simultaneously for all $0 \leq a < b$ almost surely. Letting $\theta_0$ be the Radon-Nikod\'ym derivative of the $y$-marginal of $\rho$ with respect to the Lebesgue measure, we obtain (c).
	
	Next we prove (b) of Assumption \ref{ass: standing assumptions 1}, i.e., that $\setc{\bar\bmu_0^N}{N \geq 1}$ is tight in $\calM_F^+(\R_+^2)$. For this purpose, fix $\varepsilon > 0$ and choose compact sets $C_n \subset \R_+^2$ and constants $\delta_n \geq 0$ such that:
	\begin{equation*}
		\sum_{n = 1}^\infty \delta_n \leq \frac{\varepsilon}{\alpha} \quad \text{and} \quad \rho\left(C_n^c\right) \leq \frac{\delta_n}{n}.
	\end{equation*}
	By Prohorov's theorem for finite measures, the following family of measures is tight:
	\begin{equation*}
		A_\varepsilon \defeq \bigcap_{n \geq 1} \asetc{\mu \in \calM_F^+(\R_+^2)}{\mu(\R_+^2) \leq \alpha\ \text{and}\ \mu\left(C_n^c\right) < \frac{1}{n}}.
	\end{equation*}
	It follows that the closure $K_\varepsilon$ of $A_\varepsilon$ is compact in $\calM_F^+(\R_+^2)$. In addition,
	\begin{align*}
		\cprob*{\bar\bmu_0^N \notin K_\varepsilon} &= \cprob*{\bigcup_{n \geq 1} \left\{\bar\bmu_0^N\left(C_n^c\right) \geq \frac{1}{n}\right\}} \\
		&\leq \sum_{n = 1}^\infty \cprob*{\bar\bmu_0^N\left(C_n^c\right) \geq \frac{1}{n}} \leq \sum_{n = 1}^\infty n \expect*{\bar\bmu_0^N\left(C_n^c\right)} \leq \sum_{n = 1}^\infty \frac{nM^N}{N} \rho\left(C_n^c\right) \leq \varepsilon
	\end{align*}
	for all $N \geq 1$. Observe that the equality holds since $\bar\bmu_0^N(\R_+^2) \leq \alpha$ surely, and the second inequality follows from Markov's inequality. Since $\varepsilon$ is arbitrary, this proves tightness.
	
	It only remains to show that $\bar\bmu_0^N \Rightarrow \alpha \rho$ as $N \to \infty$. If $\phi \in C_c(\R_+^2)$, then
	\begin{equation*}
		\lim_{N \to \infty} \angled{\phi}{\bar\bmu_0^N} = \lim_{N \to \infty} \frac{M^N}{N}\frac{1}{M^N}\sum_{i = 1 - M^N}^0 \phi(x_i, y_i) = \alpha \angled{\phi}{\rho} = \angled{\phi}{\alpha \rho}
	\end{equation*}
	almost surely, by the strong law of large numbers. Moreover, we may consider a countable dense set $\Phi \subset C_c(\R_+^2)$ so that the limit holds simultaneously for all $\phi \in \Phi$ with probability one. Therefore, the limit $\bar\bmu_0$ of any convergent subsequence satisfies that $\langle\phi, \bar\bmu_0\rangle = \langle\phi, \alpha \rho\rangle$ for all $\phi \in \Phi$ almost surely, and hence for all $\phi \in C_c(\R_+^2)$. We conclude that the limit of any convergent subsequence equals $\alpha \rho$ almost surely, and thus $\bar\bmu_0^N \Rightarrow \alpha \rho$ as $N \to \infty$.
\end{proof}

\subsection{Elementary properties}
\label{sub: elementary properties}

Next we derive elementary properties of solutions that will be useful in the sequel. The first one is an immediate consequence of property (a) of Definition \ref{def: fluid problem}. It implies that for studying the fluid problem, we may focus on test functions with support in $U_g$. We provide the proof of this standard lemma in Appendix \ref{app: proofs of auxiliary results}.

\begin{restatable}{lemma}{lemmafivetwo}
	\label{lem: solutions are determined by functions supported in Ug}
	For each $\phi \in C_c(\R_+^2)$ there exist functions $\psi^n \in C_c^1(U_g)$ such that:
	\begin{equation*}
		\sup_{n \geq 1} \norm{\psi^n}_\infty < \infty, \quad \lim_{n \to \infty} \psi^n(x, y) = \phi(x, y) \ind{(x, y) \in U_g} \quad \text{for all} \quad x, y \in \R_+.
	\end{equation*}
	If $\bmu \in \calS^+$, then the bounded convergence theorem and (a) of Definition \ref{def: fluid problem} imply that
	\begin{equation*}
		\angled{\phi}{\bmu_t} = \lim_{n \to \infty} \angled{\psi^n}{\bmu_t} \quad \text{for all} \quad t \geq 0 \quad \text{and} \quad \phi \in C_c\left(\R_+^2\right).
	\end{equation*}
	As a result, the measure $\bmu_t$ is fully determined by the integrals of functions in $C_c^1(U_g)$.
\end{restatable}

The following property allows to apply the transport equation \eqref{eq: transport equation} of Definition \ref{def: fluid problem} to functions $\varphi \in C_b^1(\R_+^3)$ which need not have compact support but must be bounded and have bounded partial derivatives. The standard proof is provided in Appendix \ref{app: proofs of auxiliary results}.

\begin{restatable}{lemma}{lemmafivethree}
	\label{lem: transport equation for functions with support not compact in t}
	If $\bmu \in \calS^+$ and $\varphi \in C_b^1(\R_+^3)$ has bounded partial derivatives, then \eqref{eq: transport equation} holds.
\end{restatable}

The last property of solutions that we prove here is used repeatedly in Sections \ref{sub: solution with null initial condition} and \ref{sub: structure of solutions}. It states that the rate at which mass grows inside any infinite horizontal band containing the origin is upper bounded by $\lambda_0$. Intuitively, this property holds because mass is injected at rate $\lambda_0$ at the origin and never moves downwards in the $y$-direction.

\begin{lemma}
	\label{lem: mass growth rate}
	If $\bmu \in \calS^+$, then we have:
	\begin{equation*}
		\bmu_t\left(\R_+ \times [0, y]\right) - \bmu_s\left(\R_+ \times [0, y]\right) \leq \lambda_0(t - s) \quad \text{for all} \quad 0 \leq s \leq t \quad \text{and} \quad y \in \R_+.
	\end{equation*}
	Furthermore, $\bmu_t(\R_+^2) - \bmu_s(\R_+^2) \leq \lambda_0(t - s)$ for all $0 \leq s \leq t$.
\end{lemma}

\begin{proof}
	Suppose that $\phi \in C_b^1(\R_+^2)$ is a nonnegative function such that the partial derivatives $\partial_x\phi$ and $\partial_y\phi$ are nonpositive and bounded. With a slight abuse of notation, we identify functions in $C_b^1(\R_+^2)$ with functions in $C_b^1(\R_+^3)$ which are independent of the variable $t$. With this identification, it follows from Lemma \ref{lem: transport equation for functions with support not compact in t} that, for all $0 \leq s \leq t$, we have:
	\begin{align*}
		\angled{\phi}{\bmu_t} - \angled{\phi}{\bmu_s} &= \lambda_0 \int_s^t \phi(0, 0)d\tau - \int_s^t \angled{\phi\left[\bar r\left(\bmu_\tau\right) h_g^x + h_g^y\right]}{\bmu_\tau}d\tau \\
		&+ \int_s^t \angled{\bar r\left(\bmu_\tau\right)\partial_x\phi + \partial_y\phi}{\bmu_\tau}d\tau \leq \lambda_0 \phi(0, 0) (t - s).
	\end{align*}
	
	Given $y_0 \in \R_+$, there exist nonnegative functions $\phi^n \in C_b^1(\R_+^2)$ such that
	\begin{equation*}
		\ind{(x, y) \in \R_+ \times [0, y_0]} \leq \phi^n(x, y) \leq \ind{(x, y) \in \R_+ \times [0, y_0 + 1 / n]}
	\end{equation*}
	for all $x, y \in \R_+$ and $n \geq 1$. Furthermore, it is possible to define the functions $\phi^n$ so that the partial derivatives of $\phi^n$ are nonpositive and bounded in $\R_+^2$ for each $n \geq 1$. Hence,
	\begin{equation*}
		\angled{\phi^n}{\bmu_t} - \angled{\phi^n}{\bmu_s} \leq \lambda_0 \phi^n(0, 0) (t - s) = \lambda_0 (t - s) \quad \text{for all} \quad 0 \leq s \leq t.
	\end{equation*}
	
	It follows from the bounded convergence theorem that
	\begin{equation*}
		\bmu_t\left(\R_+ \times [0, y_0]\right) - \bmu_s\left(\R_+ \times [0, y_0]\right) = \lim_{n \to \infty} \left[\angled{\phi^n}{\bmu_t} - \angled{\phi^n}{\bmu_s}\right] \leq \lambda_0 (t - s).
	\end{equation*}
	Moreover, letting $y_0 \uparrow \infty$, we obtain $\bmu_t(\R_+^2) - \bmu_s(\R_+^2) \leq \lambda_0(t - s)$.
\end{proof}

\subsection{Prescribed thresholds}
\label{sub: prescribed thresholds}

In this section we consider transport equations where the threshold function $t \mapsto \bar y_*(\bmu_t)$ is prescribed, removing the nonlinear feedback in \eqref{eq: transport equation}. As noted in Section \ref{sub: outline of the analysis of the transport equation}, the measures $\bmu_t$ are allowed to be \emph{signed} in order to enable some of the later analysis of Section \ref{sub: structure of solutions}. Under certain assumptions, we derive closed-form expressions for $\langle\phi, \bmu_t\rangle$ when $\phi$ is a test function and $\bmu$ solves a transport equation with prescribed threshold. These expressions play an important role in Sections \ref{sub: solution with null initial condition}, \ref{sub: structure of solutions} and \ref{sub: long-term behavior of solutions}. The equations that we will consider in this section are formally defined in the following definition.

\begin{definition}
	\label{def: modified transport equation}
	Fix $[a, b] \subset \R_+$ and a Borel function $\map{\bar \by}{[a, b]}{[0, \infty]}$. Suppose that $\bmu \in D_{\calM_F(\R_+^2)}[a, b]$ is such that $|\bmu|$ satisfies (a) and (b) of Definition \ref{def: fluid problem} and the integrals
	\begin{equation*}
		\int_a^t\int_{U_g} \left[\ind{y \leq \bar \by_s}h_g^x(x, y) + h_g^y(x, y)\right]\bmu_s(dx, dy)ds
	\end{equation*}
	are finite for all $t \in [a, b]$. Let $\bar \br_s(y) \defeq \ind{y \leq \bar \by_s}$ for all $s \in [a, b]$ and $y \in \R_+$. We say that $\bmu$ solves the transport equation \eqref{eq: transport equation with given threshold} with prescribed threshold $\bar \by$ in the interval $[a, b]$ if
	\begin{equation}
		\label{eq: transport equation with given threshold}
		\begin{split}
			\angled{\varphi_t}{\bmu_t} &= \angled{\varphi_a}{\bmu_a} + \lambda_0 \int_a^t \varphi_s(0, 0)ds - \int_a^t \angled{\varphi_s\left[\bar \br_s h_g^x + h_g^y\right]}{\bmu_s} ds \\
			&+ \int_a^t \angled{\partial_t\varphi_s + \bar \br_s\partial_x\varphi_s + \partial_y\varphi_s}{\bmu_s}ds
		\end{split}
	\end{equation}
	for all $t \in [a, b]$ and $\varphi \in C_c^1(\R_+^3)$, and we write $\bmu \in \calS_{\bar \by}[a, b]$.
\end{definition}

In the above definition $D_{\calM_F(\R_+^2)}[a, b]$ is the space of c\`adl\`ag functions from $[a, b]$ to $\calM_F(\R_+^2)$ when the latter space has the weak topology, with convergence characterized by integrals of continuous and bounded functions; the weak topology is not metrizable for finite signed measures, but we do not need metrizability here. Also, note that Remark \ref{rem: measurability for signed measures} in Appendix \ref{app: measurability properties} implies that the iterated integrals in \eqref{eq: transport equation with given threshold} are well-defined if $\bar \by$ is Borel.

The following definition introduces a spatial rescaling that simplifies the transport equation \eqref{eq: transport equation with given threshold} by removing the term that depends on the hazard rate vector field.

\begin{definition}
	\label{def: spatial rescaling}
	If $\mu \in \calM_F(\R_+^2)$, then we define $\map{L_\mu}{C_c(U_g)}{\R}$ as follows:
	\begin{equation*}
		L_\mu(\phi) \defeq \angled{\phi S_g^{-1}}{\mu} = \int_{U_g} \frac{\phi(x, y)}{S_g(x, y)}\mu(dx, dy) \quad \text{for all} \quad \phi \in C_c(U_g),
	\end{equation*}
	where we adopt the convention that $\phi S_g^{-1}$ equals zero outside $U_g$.
\end{definition}

The above integrals are well-defined and finite because $S_g$ is bounded away from zero in any compact subset of $U_g$. If $\mu$ is nonnegative, then these integrals define a nonnegative measure on $U_g$ which may not be finite since $S_g$ vanishes in the boundary of $U_g$. For the same reason, the integrals may not be well-defined for functions in $C_b(U_g)$ when $\mu$ is signed, and therefore need not define a measure in that case.

The following proposition describes the result of applying the spatial rescaling to a solution of the transport equation \eqref{eq: transport equation with given threshold} with a prescribed threshold.

\begin{proposition}
	\label{prop: spatial rescaling}
	If $\bmu \in \calS_{\bar \by}[a, b]$, then for all $t \in [a, b]$ and $\varphi \in C_c^1(\R_+ \times U_g)$, we have:
	\begin{equation}
		\label{eq: transport equation after spatial rescaling}
		L_{\bmu_t}\left(\varphi_t\right) = L_{\bmu_a}\left(\varphi_a\right) + \lambda_0 \int_a^t \varphi_s(0, 0)ds + \int_a^t L_{\bmu_s}\left(\partial_t\varphi_s + \bar \br_s \partial_x \varphi_s + \partial_y\varphi_s\right)ds,
	\end{equation}
	where $\bar \br_s(y) = \ind{y \leq \bar \by_s}$ for all $s \in [a, b]$ and $y \in \R_+$.
\end{proposition}

\begin{proof}
	Fix $\varphi \in C_c^1(\R_+ \times U_g)$ and define
	\begin{equation*}
		\eta_t(x, y) \defeq\begin{cases}
			\displaystyle\frac{\varphi_t(x, y)}{S_g(x, y)} = \varphi_t(x, y)\e^{-\log S_g(x, y)} & \text{if} \quad (t, x, y) \in \R_+ \times U_g, \\
			0 & \text{otherwise}.
		\end{cases}
	\end{equation*}
	Recall that $S_g$ is continuously differentiable by Assumption \ref{ass: standing assumptions 2}, and thus $\eta \in C_c^1(\R_+ \times U_g)$. In addition, for all $(t, x, y) \in [a, b] \times U_g$, we have:
	\begin{align*}
		&\partial_t \eta_t(x, y) = \frac{\partial_t\varphi_t(x, y)}{S_g(x, y)}, \\
		&\partial_x \eta_t(x, y) = \frac{\partial_x\varphi_t(x, y)}{S_g(x, y)} + \eta_t(x, y) h_g^x(x, y), \\
		&\partial_y\eta_t(x, y) = \frac{\partial_y\varphi_t(x, y)}{S_g(x, y)} + \eta_t(x, y) h_g^y(x, y).
	\end{align*}
	
	Since $L_{\bmu_t}(\varphi_t) = \angled{\eta_t}{\bmu_t}$, it follows from \eqref{eq: transport equation with given threshold} that:
	\begin{align*}
		L_{\bmu_t}(\varphi_t) &= \angled{\eta_a}{\bmu_a} + \lambda_0 \int_a^t \eta_s(0, 0)ds - \int_a^t \angled{\eta_s\left[\bar \br_s h_g^x + h_g^y\right]}{\bmu_s}ds \\
		&+ \int_a^t \angled{\partial_t\eta_s + \bar \br_s\partial_x\eta_s + \partial_y\eta_s}{\bmu_s}ds \\
		&= L_{\bmu_a}(\varphi_a) + \lambda_0 \int_a^t \frac{\varphi_s(0, 0)}{S_g(0, 0)}ds + \int_a^t \angled{\left[\partial_t\varphi_s + \bar \br_s\partial_x\varphi_s + \partial_y\varphi_s\right]S_g^{-1}}{\bmu_s}ds \\
		&= L_{\bmu_a}(\varphi_a) + \lambda_0 \int_a^t \varphi_s(0, 0)ds + \int_a^t L_{\bmu_s}\left(\partial_t\varphi_s + \bar \br_s\partial_x\varphi_s + \partial_y\varphi_s\right)ds
	\end{align*}
	for all $t \in [a, b]$ because $S_g(0, 0) = 1$.
\end{proof}

Recall that when $\bmu_t$ is nonnegative, the operator $L_{\bmu_t}$ represents the integral with respect to some nonnegative measure $\tilde \bmu_t$ . In this case, the transport equation \eqref{eq: transport equation after spatial rescaling} reads:
\begin{equation*}
	\angled{\varphi_t}{\tilde \bmu_t} = \angled{\varphi_a}{\tilde \bmu_a} + \lambda_0 \int_a^t \varphi_s(0, 0)ds + \int_a^t \angled{\partial_t\varphi_s + \bar \br_s \partial_x \varphi_s + \partial_y\varphi_s}{\tilde \bmu_s} ds.
\end{equation*}
This represents the evolution of mass distributed as $\tilde \bmu_t$ when new mass is injected at the origin at rate $\lambda_0$ and mass moves along the velocity field $\bv_t(x, y) \defeq [\bar \br_t\ 1]^\intercal$. Therefore, \eqref{eq: transport equation after spatial rescaling} is simpler than the transport equation \eqref{eq: transport equation with given threshold}, which has an extra term representing the disappearance of mass at a rate given by the hazard rate and velocity vector fields.

Suppose that $\bmu \in \calS^+$ and $\bar \by_t \defeq \bar y_*(\bmu_t)$ is known for all $t \in [a, b]$. Assume also that for each $t \in [a, b]$ and $\phi \in C_c^1(U_g)$ we can solve the following partial differential equation:
\begin{equation}
	\label{eq: pde for duality method}
	\psi_t^t(x, y) = \phi(x, y) \quad \text{and} \quad \partial_s \psi_s^t(x, y) + \ind{y \leq \bar \by_s} \partial_x \psi_s^t(x, y) + \partial_y \psi_s^t(x, y) = 0.
\end{equation}
Here $\psi_s^t(x, y)$ is a function of the variables $s, x, y \in \R_+$ and the superscript $t$ refers to a fixed terminal time. The first equation imposes a terminal condition at $s = t$, whereas the second equation must hold for all $s \in [a, t]$ and $x, y \in \R_+$; a solution $\psi^t$ can often be obtained through the method of characteristics. If the solution is in $C_c^1(\R_+ \times U_g)$, then it satisfies \eqref{eq: transport equation after spatial rescaling} with the third term on the right-hand side being zero. It follows that
\begin{equation}
	\label{eq: application of duality method}
	L_{\bmu_t}\left(\phi\right) = L_{\bmu_t}\left(\psi_t^t\right) = L_{\bmu_a}\left(\psi_a^t\right) + \lambda_0\int_a^t \psi_s^t(0, 0)ds.
\end{equation}
The solution $\psi^t$ of \eqref{eq: pde for duality method} only depends on $\phi$ and $\bar \by$. As a result, the right-hand side of \eqref{eq: application of duality method} only depends on $\phi$, $\bmu_a$ and $\bar \by$, besides the fixed problem data $\lambda_0$ and $S_g$.

Characterizing the solution measure $\bmu_t$ requires calculating $\langle\phi, \bmu_t\rangle$ for a general test function $\phi \in C_c^1(U_g)$. Since $S_g$ is continuously differentiable by Assumption \ref{ass: standing assumptions 2}, we can apply the preceding method with terminal condition $\phi S_g \in C_c^1(U_g)$. Now \eqref{eq: application of duality method} provides an expression for $\langle\phi, \bmu_t\rangle = L_{\bmu_t}(\phi S_g)$ which again only depends on $\phi$, $\bmu_a$ and $\bar \by$.

In the rest of this section, we formalize and develop this idea. The results obtained here will be particularly useful in Section \ref{sub: solution with null initial condition} for solving the fluid problem with null initial condition, for which the threshold function will be found explicitly.

\subsubsection{Behavior below the threshold}
\label{subsub: behavior below the threshold}

If the threshold $\bar \by$ is lower bounded by some constant $y_0$ and the function $\phi$ is supported inside the band $\R_+ \times [0, y_0)$, then it is possible to find solutions $\psi^t$ of \eqref{eq: pde for duality method} for characterizing the integrals of $\phi$ with respect to $\bmu \in \calS_{\bar \by}[a, b]$. The following lemma, which is proved in Appendix \ref{app: proofs of auxiliary results}, will be used to construct solutions $\psi^t$ of \eqref{eq: pde for duality method} with compact support in $\R_+ \times U_g$.

\begin{restatable}{lemma}{lemmafiveeight}
	\label{lem: support of psi^t}
	If $A \subset U_g$ is compact, then the set
	\begin{equation*}
		B \defeq \asetc{(x, y) \in \R_+^2}{x \leq x_0\ \text{and}\ y \leq y_0\ \text{for some}\ (x_0, y_0) \in A}
	\end{equation*}
	is compact and $A \subset B \subset U_g$. Let $d \defeq \mathrm{dist}(B, U_g^c)$ be the distance from the compact set $B$ to the closed set $U_g^c$ with respect to the maximum norm. Fix a time $t \geq 0$, a constant $0 \leq \varepsilon < d$ and a Borel function $\map{\alpha}{\R_+^3}{[0, 1]}$. Then the set
	\begin{equation*}
		C \defeq \asetc{(s, x, y) \in \R_+^3}{0 \leq s \leq t + \varepsilon\ \text{and}\ \left(x + \int_s^t \alpha_\tau(x, y)d\tau, y + t - s\right) \in A}
	\end{equation*}
	is contained in the following compact set
	\begin{equation*}
		D \defeq [0, t + \varepsilon] \times \asetc{(x, y) \in U_g}{\mathrm{dist}\left((x, y), B\right) \leq \varepsilon} \subset \R_+ \times U_g.
	\end{equation*}
\end{restatable}

Below we characterize the integrals $\langle\phi, \bmu_t\rangle$ when $\bmu$ solves the transport equation \eqref{eq: transport equation with given threshold} with a suitably lower bounded threshold and $\phi$ is supported below the lower bound.

\begin{proposition}
	\label{prop: integrals below the threshold}
	Suppose that $\bmu \in \calS_{\bar \by}[0, b]$ for a given threshold function $\bar \by$.
	\begin{enumerate}
		\item[(a)] Assume that there exists $y_0 \in (0, \infty]$ such that $\bar \by_t \geq y_0$ for all $t \in [0, b]$ and $\bmu_0$ is the null measure. If $t \in [0, b]$ and $\phi \in C_c^1(\R_+ \times [0, y_0))$, then
		\begin{equation}
			\label{eq: integrals below threshold a}
			\angled{\phi}{\bmu_t} =  \lambda_0 \int_0^t \phi(u, u)S_g(u, u)du.
		\end{equation}
		Furthermore, if $t \in [0, b]$ and $y \in [0, y_0)$, then
		\begin{equation*}
			\bmu_t\left(\R_+ \times [0, y]\right) = \lambda_0 \int_0^{t \wedge y} S_g(u, u)du,
		\end{equation*}
		and this equation holds for $y = y_0$ when $y_0 < \infty$.
		
		\item[(b)] Suppose that there exists $y_0 \in (0, \infty]$ such that $\bar \by_t \geq t \wedge y_0$ for all $t \in [0, b]$; here $\bmu_0$ need not be null. If $t \in [0, b]$ and $\phi \in C_c^1(\R_+ \times [0, t \wedge y_0))$, then
		\begin{equation}
			\label{eq: integrals below threshold b}
			\angled{\phi}{\bmu_t} = \lambda_0 \int_0^t \phi(u, u) S_g(u, u)du.
		\end{equation}
		Moreover, if $t \in [0, b]$ and $y \in [0, t \wedge y_0]$, then
		\begin{equation*}
			\bmu_t\left(\R_+ \times [0, y]\right) = \lambda_0 \int_0^y S_g(u, u)du.
		\end{equation*}
	\end{enumerate}
\end{proposition}

\begin{proof}
	If $\phi \in C_c^1(\R_+ \times [0, y_0))$, then it is possible to approximate the function $\phi\indc_{U_g}$ by a sequence of functions $\psi^n \in C_c^1((\R_+ \times [0, y_0)) \cap U_g)$ which can be constructed in the same way as in the proof of Lemma \ref{lem: solutions are determined by functions supported in Ug} but replacing $U_g$ by $(\R_+ \times [0, y_0)) \cap U_g$. If \eqref{eq: integrals below threshold a} holds for functions in $C_c^1((\R_+ \times [0, y_0)) \cap U_g)$, then we obtain:
	\begin{align*}
		\angled{\phi}{\bmu_t} = \angled{\phi \indc_{U_g}}{\bmu_t} &= \lim_{n \to \infty} \angled{\psi^n}{\bmu_t} \\
		&= \lim_{n \to \infty} \lambda_0 \int_0^t \psi^n(u, u) S_g(u, u)du = \lambda_0 \int_0^t \phi(u, u)S_g(u, u)du
	\end{align*}
	for all $t \in [0, b]$. The first equality holds since $\bmu$ satisfies (a) of Definition \ref{def: fluid problem}, while the second and last equalities hold by the bounded convergence theorem; for the last equality we also use that $S_g$ vanishes outside of $U_g$. Therefore, we need only prove \eqref{eq: integrals below threshold a} for functions $\phi \in C_c^1((\R_+ \times [0, y_0)) \cap U_g)$, and the same applies for \eqref{eq: integrals below threshold b}.
	
	Hence, to prove \eqref{eq: integrals below threshold a}, we fix $t \in [0, b]$ and $\phi \in C_c^1((\R_+ \times [0, y_0)) \cap U_g)$. Let $A \subset U_g$ denote the support of $\phi$ and choose a constant $\varepsilon > 0$ as in the statement of Lemma \ref{lem: support of psi^t}. Also, fix $\tilde\phi \in C_c^1(\R^2)$ and $\beta \in C_c^1(\R_+)$ satisfying that $\tilde\phi|_{\R_+^2} = \phi$ and
	\begin{equation*}
		\ind{s \in [0, t]} \leq \beta_s \leq \ind{s \in [0, t + \varepsilon]} \quad \text{for all} \quad s \in \R_+.
	\end{equation*}
	
	Consider now the function $\map{\psi^t}{\R_+^3}{\R}$ defined as:
	\begin{equation*}
		\psi^t_s(x, y) \defeq \beta_s\tilde\phi(x + t - s, y + t - s) \quad \text{for all} \quad s, x, y \in \R_+,
	\end{equation*}
	If the function $\alpha$ in the statement of Lemma \ref{lem: support of psi^t} is identically equal to one, then the set $C$ in that lemma contains the support of $\psi^t$. It follows that $\psi^t \in C_c^1(\R_+ \times U_g)$ and thus satisfies \eqref{eq: transport equation after spatial rescaling} by Proposition \ref{prop: spatial rescaling}. In order to show that \eqref{eq: pde for duality method} holds as well, observe that:
	\begin{align*}
		&\partial_s\psi_s^t(x, y) = -\partial_x\phi(x + t - s, y + t - s) - \partial_y\phi(x + t - s, y + t - s), \\
		&\partial_x\psi_s^t(x, y) = \partial_x\phi(x + t - s, y + t - s), \\
		&\partial_y\psi_s^t(x, y) = \partial_y\phi(x + t - s, y + t - s),
	\end{align*}
	for all $s \in [0, t]$ and $x, y \in \R_+$. Then, for all $s \in [0, t]$ and $x, y \in\R_+$, we have:
	\begin{equation*}
		\partial_s\psi_s^t(x, y) + \ind{y \leq \bar \by_s} \partial_x\psi_s^t(x, y) + \partial_y\psi_s^t(x, y) = \left[\ind{y \leq \bar \by_s} - 1\right] \partial_x\phi(x + t - s, y + t - s) = 0
	\end{equation*}
	For the last equality, note that $\partial_x \phi(x + t - s, y + t - s) = 0$ if $y > y_0$ since $t - s \geq 0$ and $\phi$ is supported in $\R_+ \times [0, y_0)$, while if $y \leq y_0$, then the indicator is active because $\bar \by_s \geq y_0$ for all $s \in [0, t]$. It is clear that $\psi_t^t(x, y) = \phi(x, y)$ for all $x, y \in \R_+$, so \eqref{eq: pde for duality method} holds.
	
	By \eqref{eq: application of duality method} and the fact that $\bmu_0$ is the null measure, we obtain:
	\begin{equation*}
		L_{\bmu_t}(\phi) = \lambda_0 \int_0^t \phi(t - s, t - s)ds = \lambda_0 \int_0^t \phi(u, u)du.
	\end{equation*}
	The survival function $S_g$ is continuously differentiable by Assumption \ref{ass: standing assumptions 2}, which implies that $\phi S_g \in C_c^1((\R_+ \times [0, y_0)) \cap U_g)$. Hence, we can construct a function $\psi^t \in C_c^1(\R_+ \times U_g)$ which solves \eqref{eq: pde for duality method} with boundary condition $\phi S_g$. Then \eqref{eq: integrals below threshold a} follows:
	\begin{equation*}
		\angled{\phi}{\bmu_t} = L_{\bmu_t}(\phi S_g) = \lambda_0 \int_0^t \phi(u, u)S_g(u, u)du.
	\end{equation*}
	
	Given $y \in [0, y_0)$, it is possible to approximate the indicator of $\R_+ \times [0, y]$ by a uniformly bounded sequence of functions in $C_c^1(\R_+ \times [0, y_0))$. By the bounded convergence theorem,
	\begin{equation*}
		\bmu_t\left(\R_+ \times [0, y]\right) = \lambda_0 \int_0^t \ind{u \leq y}S_g(u, u)du = \lambda_0 \int_0^{t \wedge y} S_g(u, u)du.
	\end{equation*}
	The continuity of the measure $\bmu_t$ and property (b) of Definition \ref{def: fluid problem} further imply that
	\begin{equation*}
		\bmu_t\left(\R_+ \times [0, y_0]\right) = \bmu_t\left(\R_+ \times [0, y_0)\right) = \lim_{y \uparrow y_0} \bmu_t\left(\R_+ \times [0, y]\right) = \lambda_0 \int_0^{t \wedge y_0} S_g(u, u)du
	\end{equation*}
	when $y_0 < \infty$, which completes the proof of property (a).
	
	The proof of (b) is similar, and hence we will focus on the differences with (a). If we fix $t \in [0, b]$ and $\phi \in C_c^1((\R_+ \times [0, t \wedge y_0) \cap U_g)$, then we can define $\psi^t \in C_c^1(\R_+ \times U_g)$ as in the proof of (a), with the same expressions for its partial derivatives. Once more,
	\begin{equation*}
		\partial_s\psi_s^t(x, y) + \ind{y \leq \bar \by_s} \partial_x\psi_s^t(x, y) + \partial_y\psi_s^t(x, y) = \left[\ind{y \leq \bar \by_s} - 1\right] \partial_x\phi(x + t - s, y + t - s) = 0
	\end{equation*} 
	for all $s \in [0, t]$ and $x, y \in \R_+$. Here $y > s \wedge [y_0 - (t - s)]$ implies that $y + t - s > t \wedge y_0$ and thus $\partial_x\phi(x + t - s, y + t - s) = 0$ since the support of $\phi$ is contained in $\R_+ \times [0, t \wedge y_0)$. Otherwise, we have $y \leq s \wedge [y_0 - (t - s)]$ and then $y \leq s \wedge y_0 \leq t \wedge y_0 \leq \bar \by_s$.
	
	The proof of \eqref{eq: integrals below threshold b} is completed invoking \eqref{eq: pde for duality method} and \eqref{eq: application of duality method} as in the proof of \eqref{eq: integrals below threshold a}. The term in \eqref{eq: application of duality method} associated with $\bmu_0$ vanishes since $\phi$ is supported in $\R_+ \times [0, t \wedge y_0)$ and
	\begin{equation*}
		L_{\bmu_0}\left(\psi_0^t\right) = \int_{U_g} \frac{\psi_0^t(x, y)}{S_g(x, y)}\bmu_0(dx, dy) = \int_{U_g} \frac{\phi(x + t, y + t)}{S_g(x, y)}\bmu_0(dx, dy) = 0.
	\end{equation*}
	
	Then the expression for $\bmu_t(\R_+ \times [0, y])$ follows again from approximating the indicator function of the set $\R_+ \times [0, y]$ by a sequence of functions in $C_c^1(\R_+ \times [0, t \wedge y_0])$, and then invoking property (b) of Definition \ref{def: fluid problem} for the boundary case $y = t \wedge y_0$.
\end{proof}

\subsubsection{Solutions with constant threshold}
\label{subsub: solutions with constatn threshold}

Recall our solution method for prescribed thresholds: we are leveraging a solution to the partial differential equation \eqref{eq: pde for duality method}, with boundary condition $\phi S_g$, to compute $\langle\phi, \bmu_t\rangle$ for suitable test functions $\phi$. The preceding Proposition \ref{prop: integrals below the threshold} applies to functions $\phi$ that vanish above the threshold, an essential requirement for the construction of a \emph{smooth} solution $\psi^t$ to \eqref{eq: pde for duality method}. For a general terminal condition $\phi$, this is typically not possible, but we can define a solution that is smooth outside a set of Lebesgue measure zero. However, since the modified transport equation \eqref{eq: transport equation after spatial rescaling} is defined through smooth test functions, we first need to establish its validity for test functions that are not everywhere differentiable. This is the purpose of the following two lemmas, which are proved in Appendix \ref{app: proofs of auxiliary results}.

\begin{restatable}{lemma}{lemmafiveten}
	\label{lem: approximation using mollifiers}
	Suppose that $\varphi \in C_c(\R_+ \times U_g)$ has the following properties.
	\begin{enumerate}
		\item[(a)] There exists a closed set $E \subset \R_+^3$ with Lebesgue measure zero and such that the partial derivatives of $\varphi$ exist and are continuous outside of $E$.
		
		\item[(b)] The lateral partial derivatives of $\varphi$ exist in $\R_+^3$ and for some $L \geq 0$ we have:
		\begin{equation*}
			\max\left\{\left|\partial_i^- \varphi(z)\right|, \left|\partial_i^+ \varphi(z)\right|\right\} \leq L \quad \text{for all} \quad z \in \R_+^3 \quad \text{and} \quad 1 \leq i \leq 3.
		\end{equation*} 
	\end{enumerate}
	Then there exist a compact set $K$, such that $\supp(\varphi) \subset K \subset \R_+ \times U_g$, and a sequence of functions $\varphi^n \in C_c^\infty(\R_+^3)$, with the following properties:
	\begin{equation*}
		\supp\left(\varphi^n\right) \subset K \quad \text{for all} \quad n \geq 1, \quad \lim_{n \to \infty} \norm{\varphi - \varphi^n}_\infty = 0 \quad \text{and} \quad \lim_{n \to \infty} \partial_i\varphi^n(z) = \partial_i\varphi(z)
	\end{equation*}
	for all $z \notin E$ and $1 \leq i \leq 3$. Further, there exists a constant $\tilde L \geq 0$ such that:
	\begin{equation*}
		\left|\partial_i \varphi^n(z)\right| \leq \tilde L \quad \text{for all} \quad z \in \R_+^3, \quad 1 \leq i \leq 3 \quad \text{and} \quad n \geq 1.
	\end{equation*}
\end{restatable}

\begin{restatable}{lemma}{lemmafiveeleven}
	\label{lem: transport equation for nonsmooth functions}
	Suppose that $\varphi \in C_c(\R_+ \times U_g)$ and there exists a closed set $F \subset \R_+^2$ of Lebesgue measure zero such that the partial derivatives of $\varphi$ exist and are continuous at each $(t, x, y) \in \R_+ \times U_g$ such that $(t, y) \notin F$. Also, property (b) of Lemma \ref{lem: approximation using mollifiers} holds. Given a threshold function $\bar \by$, assume that $\bmu \in \calS_{\bar \by}[a, b]$ and
	\begin{equation}
		\label{eq: integrable total variation}
		\int_a^b|\bmu_s|\left(\R_+^2\right)ds < \infty.
	\end{equation}
	Then the modified transport equation \eqref{eq: transport equation after spatial rescaling} holds for $\bmu$ and the test function $\varphi$.
\end{restatable}

We now tackle the solution of the transport equation \eqref{eq: transport equation with given threshold} when the threshold function is constant over time. For this purpose, we consider any test function $\phi \in C_c^1(U_g)$ and invoke Lemma \ref{lem: transport equation for nonsmooth functions} to apply \eqref{eq: pde for duality method} and \eqref{eq: application of duality method} although $\psi^t$ is not everywhere differentiable.

\begin{proposition}
	\label{prop: transport equation with constant threshold}
	Fix a constant $y_* \in [0, \infty)$ and let $\bar \by_t \defeq y_*$ for all $t \in [a, b]$. In addition, fix $t \in [a, b]$ and $\phi \in C_c^1(U_g)$, and consider the function defined as:
	\begin{equation*}
		\psi_s^t(x, y) \defeq \phi\left(x + \int_s^t \ind{y + \tau - s \leq y_*}d\tau, y + t - s\right)
	\end{equation*}
	for all $s \in [a, t]$ and $x, y \in \R_+$. If $\bmu \in \calS_{\bar \by}[a, b]$ satisfies \eqref{eq: integrable total variation}, then \eqref{eq: application of duality method} holds:
	\begin{equation*}
		L_{\bmu_t}(\phi) = L_{\bmu_a}(\psi_{a}^t) + \lambda_0 \int_{a}^t \psi_s^t(0, 0)ds.
	\end{equation*}
\end{proposition}

\begin{proof}
	Let $A \subset U_g$ be the support of $\phi$ and select $\varepsilon > 0$ as in Lemma \ref{lem: support of psi^t}. Then fix any two functions $\tilde \phi \in C_c^1(\R^2)$ and $\beta \in C_c^1(\R_+)$ such that $\tilde \phi|_{\R_+^2} = \phi$ and
	\begin{equation*}
		\ind{s \in [0, t]} \leq \beta_s \leq \ind{s \in [0, t + \varepsilon]} \quad \text{for all} \quad s \in \R_+.
	\end{equation*}
	In addition, define $\map{\psi^t}{\R_+^3}{\R}$ such that
	\begin{equation*}
		\psi_s^t(x, y) \defeq \beta_s \tilde \phi\left(x + \int_s^t \ind{y + \tau - s \leq y_*}d\tau, y + t - s\right) \quad \text{for all} \quad s, x, y \in \R_+;
	\end{equation*}
	this expression is obtained by applying the method of characteristics to \eqref{eq: pde for duality method}. Note that $\beta_s = 1$ and $\tilde \phi$ can be replaced by $\phi$ if $s \in [0, t]$. For conciseness, we will write $Z_s^t(x, y)$ to denote the vector-valued argument of $\tilde \phi$ in the above expression.
	
	The support of $\psi^t$ is compact and contained in $\R_+ \times U_g$ by Lemma \ref{lem: support of psi^t}. Also,
	\begin{equation*}
		\gamma(s, y) \defeq \int_s^t \ind{y + \tau - s \leq y_*}d\tau =
		\begin{cases}
			(t - s) \wedge (y_* - y)^+ & \text{if} \quad s \leq t, \\
			\left[t - s + (y - y_*)^+\right] \ind{y - y_* < s - t} & \text{if} \quad s > t. 
		\end{cases}
	\end{equation*}
	It follows that the lateral partial derivatives of $\psi^t$ exist in $\R_+^3$ and are bounded. Moreover, the set of points where the partial derivatives do not exist is contained in
	\begin{equation*}
		F \defeq \asetc{(s, y) \in \R^2}{s = t\ \text{or}\ y = y_*\ \text{or}\ y - s = y_* - t},
	\end{equation*}
	a union of lines, which is closed and has Lebesgue measure zero. In addition,
	\begin{equation*}
		\partial_s \gamma(s, y) = - \ind{y + t - s \leq y_*} \quad \text{and} \quad \partial_y \gamma(s, y) = -\ind{y_* < y + t - s,\ y \leq y_*} = \ind{y + t - s \leq y_*} - \ind{y \leq y_*}
	\end{equation*}
	for all $(s, y) \in [0, t] \times \R_+$ such that $(s, y) \notin F$.
	
	We obtain the following expressions for the partial derivatives of $\psi^t$:
	\begin{align*}
		&\partial_s \psi_s^t(x, y) = -\ind{y + t - s \leq y_*} \partial_x \phi\left(Z_s^t(x, y)\right) - \partial_y\phi\left(Z_s^t(x, y)\right), \\
		&\partial_x\psi_s^t(x, y) = \partial_x \phi\left(Z_s^t(x, y)\right), \\
		&\partial_y\psi_s^t(x, y) = \left[\ind{y + t - s \leq y_*} - \ind{y \leq y_*}\right] \partial_x \phi\left(Z_s^t(x, y)\right) + \partial_y \phi\left(Z_s^t(x, y)\right),
	\end{align*}
	for all $(s, x, y) \in [0, t] \times \R_+^2$ such that $(s, y) \notin F$. It follows that $\psi^t$ satisfies the second equation in \eqref{eq: pde for duality method} for all $(s, x, y) \in [a, t] \times \R_+^2$ such that $(s, y) \notin F$. Indeed,
	\begin{equation*}
		\partial_s \psi_s^t(x, y) + \ind{y \leq \bar \by_s} \partial_x \psi_s^t(x, y) + \partial_y \psi_s^t(x, y) = \left[\ind{y \leq \bar \by_s} - \ind{y \leq y_*}\right] \partial_x \phi(Z_s^t(x, y)) = 0
	\end{equation*}
	for all $(s, x, y) \in [a, t] \times \R_+^2$ such that $(s, y) \notin F$ since $\bar \by = y_*$ in $[a, t]$. Clearly, $\psi_t^t = \phi$, and thus $\psi^t$ also satisfies the boundary condition of \eqref{eq: pde for duality method}. Invoking Lemma \ref{lem: transport equation for nonsmooth functions}, we conclude that \eqref{eq: transport equation after spatial rescaling} holds for the test function $\psi^t$, and therefore \eqref{eq: application of duality method} holds as well.
\end{proof}

As noted earlier, this proposition will be used in the following section to solve the fluid problem when the initial condition is the null measure.

\subsection{Solution with null initial condition}
\label{sub: solution with null initial condition}

By Proposition \ref{prop: existence of solutions}, there exists $\bmu \in \calS^+$ such that $\bmu_0$ is the null measure; note that in this section $\bmu_t$ is again a nonnegative measure for all $t \geq 0$. In principle, there could exist multiple solutions to the fluid problem with null initial condition, but we will show that the solution is unique in this case, and will even provide a closed-form expression. We begin by characterizing the function $t \mapsto \bar y_*(\bmu_t)$, proving the first part of Theorem \ref{the: solution of transport equation for null initial condition}.

\begin{lemma}
	\label{lem: threshold for null initial condition}
	Recall from \eqref{eq: equilibrium threshold} that
	\begin{equation*}
		y_* \defeq \inf\asetc{y \geq 0}{\lambda_0\int_0^y S_g(u, u)du \geq 1},
	\end{equation*}
	where $y_* \defeq \infty$ if the integral is strictly less than one for all $y \geq 0$. Suppose that $\bmu \in \calS^+$ has initial condition the null measure. Then
	\begin{equation*}
		\bar y_*(\bmu_t) = \begin{cases}
			\infty 	& \text{if} \quad 0 \leq t < y_*, \\
			y_* 	& \text{if} \quad y_* < \infty \quad \text{and} \quad t \geq y_*.	
		\end{cases}
	\end{equation*}
\end{lemma}

\begin{proof}
	Since $\bmu_0$ is the null measure, Lemma \ref{lem: mass growth rate} implies that
	\begin{equation*}
		\tau_* \defeq \inf\asetc{t \geq 0}{\bmu_t\left(\R_+^2\right) \geq 1} > 0.
	\end{equation*}
	If $t \in [0, \tau_*)$, then $\bar y_*(\bmu_s) = \infty$ for all $s \in [0, t]$. By (a) of Proposition \ref{prop: integrals below the threshold}, we have:
	\begin{equation}
		\label{eq: total mass for infinite threshold}
		\bmu_t\left(\R_+^2\right) = \lambda_0 \int_0^t S_g(u, u)du \quad \text{for all} \quad t \in [0, \tau_*).
	\end{equation}
	We conclude that $\tau_* \leq y_*$. In particular, $\tau_* = \infty$ implies that $y_* = \infty$ and $\bmu_t(\R_+^2) < 1$ for all $t \geq 0$. Hence, $\bar y_*(\bmu_t) = \infty = y_*$ for all $t \geq 0$, proving the lemma when $\tau_* = \infty$. Since this case has been covered, in the following we assume that $\tau_* < \infty$.
	
	Clearly, $\bmu_t(\R_+^2) < 1$ for all $t \in [0, \tau_*)$, and $\bmu_{\tau_*}(\R_+^2) \geq 1$ by the right-continuity of $\bmu$, so
	\begin{equation*}
		\lim_{t \uparrow \tau_*} \left|\bmu_{\tau_*}\left(\R_+^2\right) - \bmu_t\left(\R_+^2\right)\right| = \lim_{t \uparrow \tau_*} \left[\bmu_{\tau_*}\left(\R_+^2\right) - \bmu_t\left(\R_+^2\right)\right] \leq \lim_{t \uparrow \tau_*} \lambda_0\left(\tau_* - t\right) = 0
	\end{equation*}
	by Lemma \ref{lem: mass growth rate}. It follows from \eqref{eq: total mass for infinite threshold} and the latter limit that
	\begin{equation*}
		\lambda_0 \int_0^{\tau_*} S_g(u, u)du = \lim_{t \uparrow \tau_*} \bmu_t\left(\R_+^2\right) = \bmu_{\tau_*}\left(\R_+^2\right) \geq 1.
	\end{equation*}
	Recall that $\tau_* \leq y_*$. Therefore, the above inequality establishes that in fact $\tau_* = y_*$, and in particular this implies that $\bar y_*(\bmu_t) = \infty$ for all $t \in [0, y_*)$.
	
	Fix any $y_0 \in (0, y_*)$. By (a) of Proposition \ref{prop: integrals below the threshold} and $\bar y_*(\bmu) = \infty$ in $[0, y_*)$, we have:
	\begin{equation}
		\label{eq: definition of m0}
		\bmu_s\left(\R_+ \times [0, y_0]\right) = \lambda_0 \int_0^{s} S_g(u, u)du \leq m_0 \defeq \lambda_0 \int_0^{y_0} S_g(u, u)du = \bmu_t\left(\R_+ \times [0, y_0]\right)
	\end{equation}
	for all $0 \leq s \leq y_0 \leq t < y_*$. Moreover, $m_0 < 1$ since $y_0 < y_*$, and Lemma \ref{lem: mass growth rate} yields:
	\begin{equation*}
		\bmu_{y_*} \left(\R_+ \times [0, y_0]\right) \leq \lim_{t \uparrow y_*} \left[\bmu_t\left(\R_+ \times [0, y_0]\right) + \lambda_0\left(y_* - t\right)\right] = m_0 < 1.
	\end{equation*}
	We conclude that $\bar y_*(\bmu_t) \geq y_0$ for all $t \in [0, y_*]$, and thus $\bmu_{y_*}(\R_+ \times [0, y_0]) = m_0$ by (a) of Proposition \ref{prop: integrals below the threshold}. In particular, \eqref{eq: definition of m0} also holds for $t = y_*$.
	
	Fix any $\varepsilon \in (0, 1 - m_0)$ and consider the following sequence of times:
	\begin{equation*}
		t_k \defeq y_* + \frac{k(1 - m_0 - \varepsilon)}{\lambda_0} \quad \text{for all} \quad k \geq 0.
	\end{equation*}
	Next we prove by induction that $\bmu_t(\R_+ \times [0, y_0]) = m_0$ for all $t \in [y_0, t_k]$ and $k \geq 0$, which implies that \eqref{eq: definition of m0} holds for all $t \geq y_0$. We have already established this property for $k = 0$, so we now assume that the property holds for some $k \geq 0$ and prove that then the property must hold for $k + 1$ as well. For this purpose, observe that Lemma \ref{lem: mass growth rate} yields:
	\begin{align*}
		\bmu_t\left(\R_+ \times [0, y_0]\right) - m_0 &= \bmu_t\left(\R_+ \times [0, y_0]\right) - \bmu_{t_k}\left(\R_+ \times [0, y_0]\right) \\
		&\leq \lambda_0\left[t - y_* - \frac{k(1 - m_0 - \varepsilon)}{\lambda_0}\right] \quad \text{for all} \quad t > t_k.
	\end{align*}
	This implies that if $t \in (t_k, t_{k + 1}]$, then $\bmu_t(\R_+ \times [0, y_0]) \leq 1 - \varepsilon$. Thus, $\bar y_*(\bmu_t) \geq y_0$ for all $t \in [0, t_{k + 1}]$ and (a) of Proposition \ref{prop: integrals below the threshold} implies that $\bmu_t(\R_+ \times [0, y_0]) = m_0$ if $t \in [y_0, t_{k + 1}]$. This completes the proof by induction, showing that \eqref{eq: definition of m0} holds for all $t \geq y_0$.
	
	Since $y_0 \in (0, y_*)$ was arbitrary, we have established that
	\begin{equation*}
		\bmu_s\left(\R_+ \times [0, y_0]\right) \leq \bmu_t\left(\R_+ \times [0, y_0]\right) = \lambda_0\int_0^{y_0} S_g(u, u)du < 1
	\end{equation*}
	for all $y_0 \in (0, y_*)$ and $0 \leq s \leq y_0 \leq t$. As noted earlier, if $t < y_* = \tau_*$, then $\bmu_t(\R_+^2) < 1$ and therefore $\bar y_*(\bmu_t) = \infty$. On the other hand, if $t \geq y_*$, then the above equation says that $\bmu_t(\R_+ \times [0, y_0]) < 1$ for all $y_0 < y_*$, and we obtain $\bar y_*(\bmu_t) \geq y_*$. By (b) of Definition \ref{def: fluid problem},
	\begin{equation*}
		\bmu_t\left(\R_+ \times [0, y_*]\right) = \bmu_t\left(\R_+ \times [0, y_*)\right) = \lim_{y_0 \uparrow y_*} \bmu_t\left(\R_+ \times [0, y_0]\right) = \lambda_0 \int_0^{y_*} S_g(u, u)du = 1,
	\end{equation*}
	so in fact $\bar y_*(\bmu_t) = y_*$, which completes the proof.
\end{proof}

\subsubsection{Remainder of the proof of Theorem \ref{the: solution of transport equation for null initial condition}}

The above lemma implies that if $\bmu \in \calS^+$ has the null initial condition, then it solves the transport equation \eqref{eq: transport equation with given threshold} with $\bar \by_t = \infty$ for $t \in [0, y_*)$ and $\bar \by_t = y_*$ for $t \in [y_*, \infty)$. Effectively, we have replaced the nonlinear feedback $\bar r(\bmu)$ in \eqref{eq: transport equation} by an expression independent of $\bmu$, obtaining a transport equation that can be solved using the results of Section \ref{sub: prescribed thresholds}. This allows to complete the proof of Theorem \ref{the: solution of transport equation for null initial condition} by proving \eqref{eq: solution with null initial condition}, which we restate:
\begin{equation*}
	\angled{\phi}{\bmu_t} = \lambda_0 \int_0^t \phi\left(u \wedge y_*, u\right) S_g\left(u \wedge y_*, u\right)du \quad \text{for all} \quad t \geq 0 \quad \text{and} \quad \phi \in C_b\left(\R_+^2\right).
\end{equation*}

\begin{proof}[Proof of Theorem \ref{the: solution of transport equation for null initial condition}]
	If \eqref{eq: solution with null initial condition} holds for all $\phi \in C_c^1(\R_+^2)$, then for each $\phi \in C_b(\R_+^2)$, there exists a sequence of functions in $C_c^1(\R_+^2)$ that is uniformly bounded and converges pointwise to $\phi$, so the bounded convergence theorem implies that \eqref{eq: solution with null initial condition} holds for all $\phi \in C_b(\R_+^2)$. Thus, we only need to establish that \eqref{eq: solution with null initial condition} holds for all functions $\phi \in C_c^1(\R_+^2)$.
	
	Lemma \ref{lem: threshold for null initial condition} and (a) of Proposition \ref{prop: integrals below the threshold} imply that
	\begin{equation}
		\label{eq: solution with null initial condition for small times}
		\angled{\phi}{\bmu_t} = \lambda_0 \int_0^t \phi(u, u) S_g(u, u)du
	\end{equation}
	if $t < y_*$ and $\phi \in C_c^1(\R_+^2)$, which proves \eqref{eq: solution with null initial condition} for all $t \in [0, y_*)$. Moreover, the same lemma and proposition imply that \eqref{eq: solution with null initial condition for small times} holds for all $t \geq 0$ and $\phi \in C_c^1(\R_+ \times [0, y_*))$.
	
	Suppose that $y_* < \infty$ and consider functions $\psi^n \in C_c^1(\R_+^2)$ such that
	\begin{equation*}
		\ind{(x, y) \in [0, n] \times \left[0, y_* - \frac{1}{n}\right]} \leq \psi^n(x, y) \leq \ind{(x, y) \in [0, n + 1] \times \left[0, y_* - \frac{1}{n + 1}\right]} \quad \text{for all} \quad x, y \in \R_+.
	\end{equation*}
	Invoking the bounded convergence theorem and \eqref{eq: solution with null initial condition for small times}, we obtain:
	\begin{equation}
		\label{eq: integrals below the threshold}
		\begin{split}
			\angled{\phi\indc_{\R_+ \times [0, y_*)}}{\bmu_{y_*}} &= \lim_{n \to \infty} \angled{\phi\psi^n}{\bmu_{y_*}} \\
			&= \lim_{n \to \infty} \lambda_0\int_0^{y_*} \phi(u, u)\psi^n(u, u)S_g(u, u)du \\
			&= \lambda_0\int_0^{y_*} \phi(u, u)S_g(u, u)du
		\end{split} 
	\end{equation}
	for all $\phi \in C_b^1(\R_+^2)$. It follows from this equation, Lemma \ref{lem: mass growth rate} and \eqref{eq: solution with null initial condition for small times} that
	\begin{align*}
		 \lambda_0\int_0^{y_*} S_g(u, u)du &= \bmu_{y^*}\left(\R_+ \times [0, y_*)\right) \\
		 &\leq \bmu_{y_*}\left(\R_+^2\right) \\
		 &\leq \lim_{t \uparrow y_*} \left[\bmu_t\left(\R_+^2\right) + \lambda_0(y_* - t)\right] = \lim_{t \uparrow y_*} \lambda_0 \int_0^t S_g(u, u)du = \lambda_0 \int_0^{y_*} S_g(u, u)du.
	\end{align*}
	
	In particular, $\bmu_{y_*}(\R_+ \times [y_*, \infty)) = 0$, and we conclude from \eqref{eq: integrals below the threshold} that
	\begin{equation}
		\label{eq: solution with null initial condition at threshold}
		\angled{\phi}{\bmu_{y_*}} = \angled{\phi\indc_{\R_+ \times [0, y_*)}}{\bmu_{y_*}} = \lambda_0 \int_0^{y_*} \phi(u, u)S_g(u, u)du.
	\end{equation}
	for all $\phi \in C_b^1(\R_+^2)$, which implies that \eqref{eq: solution with null initial condition} holds when $t \in [0, y_*]$.
	
	It remains to prove \eqref{eq: solution with null initial condition} when $y_* < \infty$ and $t > y_*$. Note that $\bmu \in \calS_{\bar \by}[y_*, b]$ with $\bar \by_t = y_*$ for $t \in [y_*, b]$ and any $b > y_*$. Also, \eqref{eq: integrable total variation} holds by Lemma \ref{lem: mass growth rate}, so Proposition \ref{prop: transport equation with constant threshold} gives:
	\begin{equation}
		\label{eq: solution with null initial condition after y*}
		L_{\bmu_t}(\phi) = L_{\bmu_{y_*}}\left(\psi_{y_*}^t\right) + \lambda_0 \int_{y_*}^t \psi_s^t(0, 0)ds \quad \text{for all} \quad t > y_* \quad \text{and} \quad \phi \in C_c^1\left(U_g\right),
	\end{equation}
	where $\psi^t$ is defined in terms of $t$ and $\phi$ as follows:
	\begin{equation*}
		\psi_s^t(x, y) \defeq \phi\left(x + \int_s^t \ind{y + \tau - s \leq y_*}d\tau, y + t - s\right) \quad \text{for all} \quad s \in [y_*, t] \quad \text{and} \quad x, y \in \R_+.
	\end{equation*}
	
	We compute the first term on the right-hand side of \eqref{eq: solution with null initial condition after y*} using \eqref{eq: solution with null initial condition at threshold}, which gives:
	\begin{align*}
		L_{\bmu_{y_*}}\left(\psi_{y_*}^t\right) &= \lambda_0 \int_0^{y_*} \phi\left(u + \int_{y_*}^t\ind{u + \tau - y_* \leq y_*}d\tau, u + t - y_*\right)du \\
		&= \lambda_0 \int_0^{y_*} \phi\left(u + (t - y_*) \wedge (y_* - u), u + t - y_*\right)du \\
		&= \lambda_0 \int_0^{y_*} \phi\left((u + t - y_*) \wedge y_*, u + t - y_*\right)du = \lambda_0 \int_{t - y_*}^t \phi\left(u \wedge y_*, u\right)du.
	\end{align*}
	Next we compute the second term on the right-hand side of \eqref{eq: solution with null initial condition after y*}, which equals:
	\begin{align*}
		\lambda_0 \int_{y_*}^t \psi_s^t(0, 0)ds &= \lambda_0 \int_{y_*}^t \phi\left(\int_s^t \ind{\tau - s \leq y_*}d\tau, t - s\right)ds \\
		&= \lambda_0 \int_{y_*}^t \phi\left((t - s) \wedge y_*, t - s\right)ds = \lambda_0 \int_0^{t - y_*} \phi\left(u \wedge y_*, u\right)du.
	\end{align*}
	Therefore, we conclude from \eqref{eq: solution with null initial condition after y*} that
	\begin{equation*}
		L_{\bmu_t}(\phi) = \lambda_0 \int_0^t \phi(u \wedge y_*, u)du \quad \text{for all} \quad t > y_* \quad \text{and} \quad \phi \in C_c^1(U_g).
	\end{equation*}
	
	If $\phi \in C_c^1(U_g)$, then $\phi S_g \in C_c^1(U_g)$ since $S_g \in C^1(U_g)$ by Assumption \ref{ass: standing assumptions 2}. Hence,
	\begin{equation*}
		\angled{\phi}{\bmu_t} = L_{\bmu_t}\left(\phi S_g\right) = \lambda_0 \int_0^t \phi(u \wedge y_*, u) S_g(u \wedge y_*, u)du
	\end{equation*}
	for all $t > y_*$ and $\phi \in C_c^1(U_g)$. By Lemma \ref{lem: solutions are determined by functions supported in Ug}, this actually holds for all $\phi \in C_c(\R_+^2)$, and the equation holds for all $\phi \in C_b(\R_+^2)$ by the observation at the beginning of the proof.
\end{proof}

\subsection{Structure of solutions}
\label{sub: structure of solutions}

In this section we prove that every solution $\bmu$ of the fluid problem can be expressed as the superposition of two measure-valued functions. Recall that, intuitively, one of these functions describes the evolution of tasks that arrive to the system after time zero, and solves the fluid problem with null initial condition. The other measure-valued function describes the evolution of tasks that where already present at time zero, and solves \eqref{eq: transport equation with given threshold} with $\lambda_0 = 0$ and threshold function $\bar \by = \bar y_*(\bmu)$. In order to derive this decomposition, we first obtain a time-dependent lower bound for the threshold function.

\begin{lemma}
	\label{lem: lower bound for threshold}
	Suppose that $\bmu \in \calS^+$ and define $y_*$ as in \eqref{eq: equilibrium threshold}. Then
	\begin{equation*}
		\bar y_*\left(\bmu_t\right) \geq t \wedge y_* \quad \text{for all} \quad t \geq 0.
	\end{equation*}
\end{lemma}

\begin{proof}
	Consider the set defined as
	\begin{equation*}
		\calT \defeq
		\begin{cases}
			\asetc{t \in [0, y_*]}{\bar y_*(\bmu_t) \geq t} & \text{if} \quad y_* < \infty, \\
			\asetc{t \in [0, y_*)}{\bar y_*(\bmu_t) \geq t} & \text{if} \quad y_* = \infty.
		\end{cases}
	\end{equation*}
	In addition, suppose that $t \in [0, y_*)$ and $[0, t) \subset \calT$; observe that $t = 0$ is possible. First, let us assume that $t > 0$. Then we obtain:
	\begin{equation*}
		\bmu_s\left(\R_+ \times [0, s]\right) = \lambda_0 \int_0^s S_g(u, u)du \quad \text{for all} \quad s \in [0, t)
	\end{equation*}
	by (b) of Proposition \ref{prop: integrals below the threshold} with any $b \in [s, t)$ and $y_0 \geq b$. By Lemma \ref{lem: mass growth rate}, if $s \in [0, t)$, then
	\begin{equation*}
		\bmu_t\left(\R_+ \times [0, s]\right) \leq \bmu_s(\R_+ \times [0, s]) + \lambda_0(t - s) = \lambda_0 \int_0^s S_g(u, u)du + \lambda_0 (t - s).
	\end{equation*}
	Moreover, using (b) of Definition \ref{def: fluid problem} and the continuity from below of $\bmu_t$, we get:
	\begin{equation*}
		\bmu_t\left(\R_+ \times [0, t]\right) = \bmu_t\left(\R_+ \times [0, t)\right) = \lim_{s \uparrow t} \bmu_t\left(\R_+ \times [0, s]\right) \leq \lambda_0 \int_0^t S_g(u, u)du.
	\end{equation*}
	
	The right-hand side of the above inequality is strictly less than one by definition of $y_*$ and because $t < y_*$. Also, the left-most term equals zero if $t = 0$ by (b) of Definition \ref{def: fluid problem}, so it is also strictly less than one in this case as well. Then the following arguments apply both when $t = 0$ and $t > 0$. There exists $\delta > 0$ such that $\bmu_t(\R_+ \times [0, t]) = 1 - \delta$. Further, the continuity from above of the measure $\bmu_t$ implies that there exists $0 < \varepsilon < \delta / (2\lambda_0)$ such that $\bmu_t(\R_+ \times [0, t + \varepsilon]) \leq 1 - \delta / 2$. If $t \leq \tau \leq t + \varepsilon$, then Lemma \ref{lem: mass growth rate} yields:
	\begin{align*}
		\bmu_\tau \left(\R_+ \times [0, \tau]\right) &\leq \bmu_\tau \left(\R_+ \times [0, t + \varepsilon]\right) \\
		&\leq \bmu_t\left(\R_+ \times [0, t + \varepsilon]\right) + \lambda_0(\tau - t) \leq 1 - \frac{\delta}{2} + \lambda_0 \varepsilon < 1.
	\end{align*}
	In particular, we get that $\bar y_*(\bmu_\tau) \geq \tau$ for all $\tau \in [t, t + \varepsilon]$.
	
	We have established that if $t \in [0, y_*)$ and $[0, t) \subset \calT$, then there exists some $\varepsilon > 0$ such that $[0, t + \varepsilon] \subset \calT$. If we assume that $t_0 \defeq \inf \setc{t \geq 0}{t \notin \calT} < y_*$, then $[0, t_0) \subset \calT$ and we get the contradiction that $[0, t_0 + \varepsilon] \subset \calT$ for some $\varepsilon > 0$. So $t_0 = y_*$ and we have:
	\begin{equation*}
		\bar y_*\left(\bmu_t\right) \geq t \quad \text{for all} \quad t \in [0, y_*).
	\end{equation*}
	
	It remains to prove that $\bar y_*(\bmu_t) \geq y_*$ for all $t \geq y_*$ when $y_* < \infty$, and for this we proceed as in the proof of Lemma \ref{lem: threshold for null initial condition}. Specifically, we will prove that if $y_0 < y_*$, then $\bar y_*(\bmu_t) \geq y_0$ for all $t \geq y_*$, which yields the desired property. Fix $y_0 \in (0, y_*)$ and define:
	\begin{equation*}
		m_0 \defeq \lambda_0 \int_0^{y_0} S_g(u, u)du = \bmu_t\left(\R_+ \times [0, y_0]\right);
	\end{equation*}
	the second equality holds for $t \in [y_0, y_*)$ by (b) of Proposition \ref{prop: integrals below the threshold} since $\bar y_*(\bmu_t) \geq t \wedge y_*$ for all $t \in [0, y_*)$. Furthermore, $m_0 < 1$ because $y_0 < y_*$. As in the proof of Lemma \ref{lem: threshold for null initial condition}, we may now fix $\varepsilon \in (0, 1 - m_0)$ and establish by induction that
	\begin{equation}
		\label{eq: induction claim}
		\bmu_t\left(\R_+ \times [0, y_0]\right) = m_0 \quad \text{for all} \quad y_0 \leq t \leq t_k \defeq y_* + \frac{k(1 - m_0 - \varepsilon)}{\lambda_0} \quad \text{and} \quad k \geq 0.
	\end{equation}
	This will prove that $\bmu_t(\R_+ \times [0, y_0]) < 1$ for all $t \geq y_0$ and $y_0 < y_*$, which in turn implies that $\bar y_*(\bmu_t) \geq y_0$ for all $t \geq y_* > y_0$, establishing that $\bar y_*(\bmu_t) \geq y_*$ for all $t \geq y_*$.

	For the base case $k = 0$, observe that
	\begin{equation*}
		\bmu_{y_*}\left(\R_+ \times [0, y_0]\right) \leq \lim_{t \uparrow y_*} \left[\bmu_t\left(\R_+ \times [0, y_0]\right) + \lambda_0(y_* - t)\right] = m_0 < 1.
	\end{equation*}
	It follows that $\bar y_*(\bmu_{y_*}) \geq y_0$ and thus $\bar y_*(\bmu_t) \geq t \wedge y_0$ for all $t \in [0, y_*]$. As a result, (b) of Proposition \ref{prop: integrals below the threshold} implies that $\bmu_{y_*}(\R_+ \times [0, y_0]) = m_0$, which proves the base case.
	
	We now assume now that \eqref{eq: induction claim} holds for $k \geq 0$ and prove it for $k + 1$. Lemma \ref{lem: mass growth rate} yields:
	\begin{align*}
		\bmu_t\left(\R_+ \times [0, y_0]\right) - m_0 &= \bmu_t\left(\R_+ \times [0, y_0]\right) - \bmu_{t_k}\left(\R_+ \times [0, y_0]\right) \\
		&\leq \lambda_0\left[t - y_* - \frac{k(1 - m_0 - \varepsilon)}{\lambda_0}\right] \quad \text{for all} \quad t > t_k.
	\end{align*}
	Then $\bmu_t(\R_+ \times [0, y_0]) \leq 1 - \varepsilon$ for all $t \in (t_k, t_{k + 1}]$, so $\bar y_*(\bmu_t) \geq y_0$ for all $t \in [y_0, t_{k + 1}]$. Thus,
	\begin{equation*}
		\bar y_*(\bmu_t) \geq t \wedge y_0 \quad \text{for all} \quad t \in [0, t_{k + 1}],
 	\end{equation*}
 	and (b) of Proposition \ref{prop: integrals below the threshold} gives $\bmu_t(\R_+ \times [0, y_0]) = m_0$ for all $t \in [y_0, t_{k + 1}]$.
\end{proof}

The following property follows as a corollary.

\begin{lemma}
	\label{lem: threshold for general initial condition}
	If $\bmu \in \calS^+$ and $y_*$ is defined as in \eqref{eq: equilibrium threshold}, then
	\begin{equation*}
		\bar y_*\left(\bmu_t\right) \geq t \quad \text{for all} \quad t \in [0, y_*) \quad \text{and} \quad \bar y_*\left(\bmu_t\right) = y_* \quad \text{for all} \quad t \geq y_*.
	\end{equation*}
\end{lemma}

\begin{proof}
	We know that $\bar y_*(\bmu_t) \geq t \wedge y_*$ for all $t \geq 0$ by Lemma \ref{lem: lower bound for threshold}. In particular, the claim holds trivially if $y_* = \infty$. Suppose then that $y_* < \infty$ and observe that
	\begin{equation*}
		\bmu_t\left(\R_+ \times [0, y_*]\right) = \lambda_0 \int_0^{y_*} S_g(u, u)du = 1 \quad \text{for all} \quad t \geq y_*
	\end{equation*}
	by (b) of Proposition \ref{prop: integrals below the threshold} and the definition of $y_*$. Therefore, $\bar y_*(\bmu_t) \leq y_*$ for all $t \geq y_*$, and hence $\bar y_*(\bmu_t) = y_*$ for all $t \geq y_*$.
\end{proof}

\subsubsection{Remainder of the proof of Theorem \ref{the: structure of solutions}}

The above conditions for the threshold function of a general solution $\bmu \in \calS^+$, combined with the characterization of the solution with null initial condition, are now used to obtain the decomposition for $\bmu$ in Theorem \ref{the: structure of solutions}. We briefly recall this decomposition: let $\bnu \in \calS^+$ be the unique solution with null initial condition and let $\bxi \defeq \bmu - \bnu$. We show that $\bxi_t \in \calM_F^+(\R_+^2)$ and satisfies (a) and (b) of Definition \ref{def: fluid problem} for $t \geq 0$. We also prove \eqref{eq: finite service rate for xi}, i.e.,
\begin{equation*}
	\int_0^t \angled{\bar r\left(\bmu_s\right) h_g^x + h_g^y}{\bxi_s}ds < \infty \quad \text{for all} \quad t \geq 0.
\end{equation*}
In addition, we show that \eqref{eq: transport equation without arrivals} holds as well. Namely, for each $t \geq 0$ and $\varphi \in C_c^1(\R_+^3)$,
\begin{align*}
	\angled{\varphi_t}{\bxi_t} &= \angled{\varphi_0}{\bxi_0} - \int_0^t \angled{\varphi_s \left[\bar r\left(\bmu_s\right) h_g^x + h_g^y\right]}{\bxi_s}ds \\
	&+ \int_0^t \angled{\partial_t \varphi_s + \bar r\left(\bmu_s\right) \partial_x \varphi_s + \partial_y \varphi_s}{\bxi_s}ds.
\end{align*}

\begin{proof}[Proof of Theorem \ref{the: structure of solutions}]
	Because $\bmu$ and $\bnu$ satisfy properties (a) and (b) of Definition \ref{def: fluid problem}, $|\bxi| \leq \bmu + \bnu$ satisfies these properties as well. We will prove later that $\bxi_t$ is a nonnegative measure for all $t \geq 0$, which implies that $\bxi = |\bxi|$ satisfies (a) and (b) of Definition \ref{def: fluid problem}.
	
	Consider a fixed time $t \geq 0$ and observe that:
	\begin{equation*}
		A_t \defeq \asetc{y \in \R_+}{\bar r(\bmu_t, y) \neq \bar r (\bnu_t, y)} \subset (t, \infty) \quad \text{if} \quad t < y_*,
	\end{equation*}
	by Lemmas \ref{lem: threshold for null initial condition} and \ref{lem: threshold for general initial condition}. Indeed, the former lemma gives $\bar r(\bnu_t, y) = 1$ for all $y \geq 0$, and $\bar r(\bmu_t, y) = 1$ for all $y \leq t$ by the latter lemma. Also, $A_t = \emptyset$ if $t \geq y_*$ since $\bar y_*(\bmu_t) = \bar y_*(\bnu_t)$ in this case. It follows from Theorem \ref{the: solution of transport equation for null initial condition} that $\bnu_t(\R_+ \times (t, \infty)) = 0$ for all $t \geq 0$, so we conclude that $A_t$ has measure zero with respect to $\bnu_t$ for all times, which means that $\bar r (\bmu_t) = \bar r(\bnu_t)$ almost everywhere with respect to this measure. As a result,
	\begin{align*}
		\angled{\varphi_t}{\bnu_t} &=  \lambda_0 \int_0^t \varphi_s(0, 0)ds - \int_0^t \angled{\varphi_s\left[\bar r\left(\bmu_s\right) h_g^x + h_g^y\right]}{\bnu_s}ds \\
		&+ \int_0^t \angled{\partial_t\varphi_s + \bar r\left(\bmu_s\right)\partial_x\varphi_s + \partial_y\varphi_s}{\bnu_s}ds \quad \text{for all} \quad t \geq 0 \quad \text{and} \quad \varphi \in C_c^1(\R_+^3),
	\end{align*}
	i.e., it is possible to replace $\bar r (\bnu)$ by $\bar r (\bmu)$ in the transport equation \eqref{eq: transport equation} for $\bnu$. Moreover, property (c) of Definition \ref{def: fluid problem} holds for $\bnu$ with $\bar r(\bnu)$ replaced by $\bar r(\bmu)$. This and (c) of Definition \ref{def: fluid problem} for $\bmu$ imply that the integral in \eqref{eq: finite service rate for xi} is the difference of two finite numbers, thus finite. Further, subtracting the above transport equation for $\bnu$, with threshold given by $\bmu$, from the transport equation \eqref{eq: transport equation} for $\bmu$, we obtain \eqref{eq: transport equation without arrivals}.
	
	It only remains to prove that $\bxi_t$ is a nonnegative measure for $t \geq 0$; observe that $\bxi_t$ is finite since it is the difference of two finite measures. Using standard terminology, we will say that a set $A \subset \R_+^2$ is positive for $\bxi_t$ if all Borel sets $B \subset A$ satisfy that $\bxi_t(B) \geq 0$, and we will say that $A$ is null if all Borel sets $B \subset A$ satisfy $ \bxi_t(B) = 0$, which implies that $A$ is positive. We now fix $t \geq 0$ and prove that $\R_+^2$ is positive for $\bxi_t$. Note that $\R_+ \times (t, \infty)$ is positive for $\bxi_t$ since the support of $\bnu_t$ is contained in $\R_+ \times [0, t]$ by Theorem \ref{the: solution of transport equation for null initial condition}. Also, we have already proved that the $y$-marginal of $|\bxi_t|$ is absolutely continuous with respect to the Lebesgue measure. Therefore, it suffices to show that $\R_+ \times [0, t)$ is positive for $\bxi_t$.

	Suppose that $t \in [0, y_*)$ and let $\bar \by_s \defeq \bar y_*(\bmu_s)$ for all $s \geq 0$. We have already proved that $\bxi \in \calS_{\bar \by}[0, b]$ with $\lambda_0 = 0$ for all $b > 0$, and Lemma \ref{lem: threshold for general initial condition} yields $\bar \by_s \geq s \wedge y_*$ for $s \geq 0$. Then (b) of Proposition \ref{prop: integrals below the threshold} implies that $\langle\phi, \bxi_t\rangle = 0$ for all functions $\phi \in C_c^1(\R_+ \times [0, t))$, so the set $\R_+ \times [0, t)$ is null for $\bxi_t$, and hence $\bxi_t \in \calM_F^+(\R_+^2)$ for all $t \in [0, y_*)$. Furthermore, if $y_* < \infty$, then the same argument shows that $\R_+ \times [0, t)$ is null for $\bxi_t$ if $t \in [0, y_*]$.
	
	Suppose now that $y_* < \infty$ and $t > y_*$. The threshold function $\bar \by$ is constant in $[y_*, \infty)$ by Lemma \ref{lem: threshold for general initial condition}. Also, $\bmu$ and $\bnu$ satisfy \eqref{eq: integrable total variation} over any finite interval by Lemma \ref{lem: mass growth rate} and thus $\bxi$ has the same property. Thus, Proposition \ref{prop: transport equation with constant threshold} implies that if $\phi \in C_c^1(U_g)$, then
	\begin{equation*}
		L_{\bxi_t}(\phi) = L_{\bxi_{y_*}}\left(\psi_{y_*}^t\right) \quad \text{with} \quad \psi_s^t(x, y) \defeq \phi\left(x + \int_s^t \ind{y + \tau - s \leq y_*}d\tau, y + t - s\right)
	\end{equation*}
	for all $s \in [y_*, t]$ and $x, y \in \R_+$. If $\phi$ is nonnegative, then $L_{\bxi_t}(\phi) \geq 0$ because $\psi^t$ and $\bxi_{y_*}$ are nonnegative. Assumption \ref{ass: standing assumptions 2} implies that $S_g$ is continuously differentiable. Therefore, we may replace $\phi$ by $\phi S_g$ to obtain $\langle\phi, \bxi_t\rangle = L_{\bxi_t}(\phi S_g) \geq 0$ for all nonnegative functions $\phi \in C_c^1(U_g)$. Furthermore, it follows from Lemma \ref{lem: solutions are determined by functions supported in Ug} that $\langle\phi, \bxi_t\rangle \geq 0$ in fact holds for all nonnegative $\phi \in C_c(\R_+^2)$ and therefore $\bxi_t$ is nonnegative.
\end{proof}

\subsection{Long-term behavior of solutions}
\label{sub: long-term behavior of solutions}

If $y_*$ is as in \eqref{eq: equilibrium threshold} and  $\phi \in C_c(\R_+^2)$, then the bounded convergence theorem gives:
\begin{equation*}
	\lim_{t \to \infty} \lambda_0 \int_0^t \phi\left(u \wedge y_*, u\right) S_g\left(u \wedge y_*, u\right)du = \lambda_0 \int_0^\infty \phi\left(u \wedge y_*, u\right) S_g\left(u \wedge y_*, u\right)du,
\end{equation*}
as in \eqref{eq: global attractor}. The right-hand side defines the nonnegative Radon measure $\mu_\infty \in \calM^+(\R_+^2)$, which is finite if and only if \eqref{eq: condition for fininte total mass in limit} holds, i.e., if the right-hand side is finite when the test function $\phi$ is replaced by the function that is identically equal to one. The limit implies that the unique solution of the fluid problem with null initial condition converges vaguely to $\mu_\infty$; if $\mu_\infty$ is finite, the limit holds for all $\phi \in C_b(\R_+^2)$ and we get weak convergence.

\subsubsection{Proof of Theorem \ref{the: long-term behavior of solutions}}

Let us recall the statement of the theorem: given any solution $\bmu$ of the fluid problem, we define $\bnu$ and $\bxi$ as in the decomposition of Theorem \ref{the: structure of solutions}. Then we prove that $\bxi_t$ converges vaguely to the null measure as $t \to \infty$ and both $\bmu_t$ and $\bnu_t$ converge vaguely to $\mu_\infty$. Further, if $y_* < \infty$, then $\bxi_t$ in fact converges weakly to the null measure as $t \to \infty$, and if $\mu_\infty$ is a finite measure, then the measures $\bmu_t$ and $\bnu_t$ converge weakly to $\mu_\infty$. Hence, $\mu_\infty$ is the unique fixed point of the fluid problem and a global attractor.

\begin{proof}[Proof of Theorem \ref{the: long-term behavior of solutions}]
	By Theorem \ref{the: structure of solutions}, $\bxi \in \calS_{\bar \by}[0, b]$ with $\lambda_0 = 0$ for all $b > 0$ and the threshold function $\bar \by \defeq \bar y_*(\bmu)$. If $y_* = \infty$, then Lemma \ref{lem: threshold for general initial condition} gives $\bar \by_t \geq t$ for all $t \geq 0$. For each $\phi \in C_c^1(\R_+^2)$, there exists $y \geq 0$ such that the support of $\phi$ is contained in $\R_+ \times [0, y)$. Thus, (b) of Proposition \ref{prop: integrals below the threshold} implies that $\langle\phi, \bxi_t\rangle = 0$ for all $t \geq y$. Since $C_c^1(\R_+^2)$ is dense in $C_c(\R_+^2)$, this holds for all $\phi \in C_c(\R_+^2)$ and we conclude that $\bxi_t$ converges vaguely to the null measure as $t \to \infty$. As already noted before the proof of the theorem, $\bnu_t$ converges vaguely to $\mu_\infty$, and therefore $\bmu_t = \bnu_t + \bxi_t$ has the same property.
	
	Suppose now that $y_* < \infty$ and note that $\bar \by_t = y_*$ for all $t \geq y_*$ by Lemma \ref{lem: threshold for general initial condition}. If $\phi \in C_c^1(U_g)$, then $\phi S_g \in C_c^1(U_g)$ by Assumption \ref{ass: standing assumptions 2}. Hence, for each $t \geq y_*$, we have:
	\begin{equation*}
		\angled{\phi}{\bxi_t} = L_{\bxi_t}\left(\phi S_g\right) = \int_{\R_+^2} \frac{1}{S_g(x, y)} \left(\phi S_g\right)\left(x + \int_{y_*}^t \ind{y + \tau - y_* \leq y_*}d\tau, y + t - y_*\right) \bxi_{y_*}(dx, dy)
	\end{equation*}
	by Proposition \ref{prop: transport equation with constant threshold} applied in the interval $[y_*, t]$ to the test function $\phi S_g$; note that \eqref{eq: integrable total variation} is satisfied by $\bxi$ since it is satisfied by $\bmu$ and $\bnu$ by Lemma \ref{lem: mass growth rate}.
	
	The survival function $S_g(x, y)$ is nonincreasing in $x$ and $y$, so
	\begin{equation*}
		\frac{1}{S_g(x, y)} S_g\left(x + \int_{y_*}^t \ind{y + \tau - y_* \leq y_*}d\tau, y + t - y_*\right) \leq 1 \quad \text{for all} \quad t \geq y_* \quad \text{and} \quad x, y \in \R_+.
	\end{equation*}
	Given a nonnegative $\psi \in C_c(\R_+^2)$, there exist functions $\psi^n \in C_c^1(U_g)$ which approximate $\psi \indc_{U_g}$ pointwise and boundedly as in Lemma \ref{lem: solutions are determined by functions supported in Ug}. Using the bounded convergence theorem, and recalling that $\bxi_t(U_g^c) = 0$ for the first equality below, we obtain:
	\begin{align*}
		\angled{\psi}{\bxi_t} &= \lim_{n \to \infty} \angled{\psi^n}{\bxi_t} \\
		&= \lim_{n \to \infty} \int_{\R_+^2} \frac{1}{S_g(x, y)} \left(\psi^n S_g\right)\left(x + \int_{y_*}^t \ind{y + \tau - y_* \leq y_*}d\tau, y + t - y_*\right) \bxi_{y_*}(dx, dy) \\
		&= \int_{\R_+^2} \frac{1}{S_g(x, y)} \left(\psi \indc_{U_g}S_g\right)\left(x + \int_{y_*}^t \ind{y + \tau - y_* \leq y_*}d\tau, y + t - y_*\right) \bxi_{y_*}(dx, dy) \\
		&\leq \int_{\R_+^2} \frac{1}{S_g(x, y)} \left(\psi S_g\right)\left(x + \int_{y_*}^t \ind{y + \tau - y_* \leq y_*}d\tau, y + t - y_*\right) \bxi_{y_*}(dx, dy)
	\end{align*}
	Moreover, we can approximate the function that is identically equal to one by functions in $C_c(\R_+^2)$ from below. Then the monotone convergence theorem gives:
	\begin{align*}
		\bxi_t(\R_+^2) &\leq \int_{\R_+^2} \frac{1}{S_g(x, y)} S_g\left(x + \int_{y_*}^t \ind{y + \tau - y_* \leq y_*}d\tau, y + t - y_*\right) \bxi_{y_*}(dx, dy) \\
		&\leq \int_{\R_+^2} \frac{S_g\left(x, y + t - y_*\right)}{S_g(x, y)} \bxi_{y_*}(dx, dy) \quad \text{for all} \quad t \geq y_*
	\end{align*}
	
	By the  bounded convergence theorem, we have:
	\begin{equation*}
		\lim_{t \to \infty} \bxi_t(\R_+^2) \leq \lim_{t \to \infty} \int_{\R_+^2} \frac{S_g\left(x, y + t - y_*\right)}{S_g(x, y)} \bxi_{y_*}(dx, dy) = 0,
	\end{equation*}
	which implies that $\bxi_t$ converges weakly to the null measure as $t \to \infty$.
	
	Finally, suppose that $y_* < \infty$ and \eqref{eq: condition for fininte total mass in limit} holds, i.e., $\mu_\infty$ is finite. Then
	\begin{equation*}
		\lim_{t \to \infty} \angled{\phi}{\bnu_t} = \lim_{t \to \infty} \lambda_0 \int_0^t \phi\left(u \wedge y_*, u\right)S_g\left(u \wedge y_*, u\right)du = \angled{\phi}{\mu_\infty} \quad \text{for all} \quad \phi \in C_b\left(\R_+^2\right).
	\end{equation*}
	by Theorem \ref{the: solution of transport equation for null initial condition} and the dominated convergence theorem. Thus, we conclude that $\bnu_t \to \mu_\infty$ weakly as $t \to \infty$. Further, $\bmu_t \to \mu_\infty$ weakly as well, because $\bmu_t = \bnu_t + \bxi_t$ and we have proved that $\bxi_t$ converges weakly to the null measure if $y_* < \infty$.
\end{proof}

If $y_* < \infty$, then $\bar y_*(\bmu_t) = y_*$ for all $t \geq y_*$ and this allows to invoke Proposition \ref{prop: transport equation with constant threshold} for proving weak convergence in the above theorem. In contrast, if $y_* = \infty$, then we only know that $\bar y_*(\bmu_t) \geq t$ for all $t \geq 0$ and lack an expression as that in Proposition \ref{prop: transport equation with constant threshold}. However, under mild additional assumptions, the weak convergence results can still be proved.

\begin{proposition}
	\label{prop: weak convergence underload case}
	Suppose that $\bmu \in \calS^+$ and define $\bnu$ and $\bxi$ as in Theorem \ref{the: structure of solutions}. Assume that there exist $\varepsilon > 0$ and $y_{\varepsilon} \geq 0$ such that $h_g^y(x, y) \geq \varepsilon$ if $y > y_\varepsilon$. Then $\bxi_t$ converges weakly to the null measure as $t \to \infty$, and if \eqref{eq: condition for fininte total mass in limit} holds, then $\bmu_t$ and $\bnu_t$ converge weakly to $\mu_\infty$.
\end{proposition}

\begin{proof}
	The claim holds for $y_* < \infty$ by Theorem \ref{the: long-term behavior of solutions}, so we will assume that $y_* = \infty$. As in Lemma \ref{lem: transport equation for functions with support not compact in t}, we obtain that \eqref{eq: transport equation without arrivals} holds for any test function in $C_b^1(\R_+^3)$ having bounded partial derivatives. Considering the test function that is identically equal to one, we get:
	\begin{align*}
		\bxi_t\left(\R_+^2\right) &= \bxi_0\left(\R_+^2\right) - \int_0^t \int_{U_g} \left[\bar r(\bmu_s, y) h_g^x(x, y) + h_g^y(x, y)\right] \bxi_s(dx, dy)ds \\
		&\leq \bxi_0\left(\R_+^2\right) - \int_0^t \varepsilon \int_{U_g} \ind{y > y_\varepsilon}\bxi_s(dx, dy)ds \\
		&\leq \bxi_0(\R_+^2) - \varepsilon \int_0^t \bxi_s\left(\R_+ \times \left(y_\varepsilon, \infty\right)\right)ds \quad \text{for all} \quad t \geq 0.
	\end{align*}
	
	Lemma \ref{lem: threshold for general initial condition} implies that $\bar y_*(\bmu_t) \geq t$ for all $t \geq 0$. Thus, (b) of Proposition \ref{prop: integrals below the threshold} yields:
	\begin{equation*}
		\bxi_t\left(\R_+ \times \left[0, y_\varepsilon\right]\right) = 0 \quad \text{and} \quad \bxi_t\left(\R_+^2\right) = \bxi_t\left(\R_+ \times \left(y_\varepsilon, \infty\right)\right) \quad \text{for all} \quad t \geq y_\varepsilon,
	\end{equation*}
	because $\bxi$ solves the transport equation \eqref{eq: transport equation with given threshold} with threshold $\bar \by = \bar y_*(\bmu)$ and $\lambda_0 = 0$. Hence,
	\begin{equation*}
		\bxi_t\left(\R_+^2\right) \leq \bxi_0\left(\R_+^2\right) - \varepsilon \int_{y_\varepsilon}^t \bxi_s\left(\R_+^2\right)ds \quad \text{for all} \quad t \geq y_{\varepsilon},
	\end{equation*}
	which implies that the integral converges as $t \to \infty$.
	
	Since $\bxi_t(\R_+^2) \geq 0$ for all $t \geq 0$ and the integral over $t \in [y_\varepsilon, \infty)$ converges, we would like to conclude that $\bxi_t(\R_+^2) \to 0$ as $t \to \infty$. This could not be the case if the function $t \mapsto \bxi_t(\R_+^2)$ had fast oscillations between zero and a positive value. Nevertheless, we now show that such fast oscillations cannot occur. By Lemma \ref{lem: mass growth rate} and Theorem \ref{the: solution of transport equation for null initial condition},
	\begin{align*}
		\bxi_t\left(\R_+^2\right) - \bxi_s\left(\R_+^2\right) &= \bmu_t\left(\R_+^2\right) - \bmu_s\left(\R_+^2\right) - \bnu_t\left(\R_+^2\right) + \bnu_s\left(\R_+^2\right) \\
		&= \bmu_t\left(\R_+^2\right) - \bmu_s\left(\R_+^2\right) - \lambda_0 \int_s^t S_g(u, u)du \leq \lambda_0(t - s)
	\end{align*}
	for all $0 \leq s \leq t$. Therefore, the integral can only converge if $\bxi_t(\R_+^2) \to 0$ as $t \to \infty$.
	
	Then $\bxi_t$ converges weakly to the null measure as $t \to \infty$, and the weak convergence of the measures $\bnu_t$ and $\bmu_t$ to $\mu_\infty$ follows as in the proof of Theorem \ref{the: long-term behavior of solutions}.
\end{proof}


\begin{appendices}

\section{Construction of sample paths}
\label{app: construction of sample paths}

Consider the probability space $(\Omega, \calF, P)$ where the renewal process $\calR$, the random vectors $(b_i, c_i)$ and the random measure $\bmu_0$ are defined. Given $\omega \in \Omega$, below we define
\begin{equation*}
	d_i(\omega), \quad \bx_i(\omega, t) \quad \text{and} \quad \by_i(\omega, t)\quad \text{for all} \quad i \geq 1 - M(\omega) \quad \text{and} \quad t \geq 0,
\end{equation*}
where we note that $M(\omega) = \bmu_0(\omega, \R_+^2)$. The above quantities will be defined in terms of the arrival times $a_i(\omega)$, the vectors $(b_i(\omega), c_i(\omega))$ and the measure $\bmu_0(\omega)$. Also, $\bmu_t(\omega)$ is then given by \eqref{eq: state of the system}. In the sequel we fix and omit $\omega$ from the notation.

\subsection{Events of the system}
\label{sub: events of the system}

First, we construct a sequence of times $\setc{\tau_k}{k \geq 0}$ such that the number of tasks in the system and the service rate assigned to each task are constant in each interval $[\tau_k, \tau_{k + 1})$. For this purpose, we begin with the following definitions:
\begin{equation*}
	\tau_0 \defeq 0 \quad \text{and} \quad r_i^0 \defeq r\left(\bmu_0, \by_i(0)\right) \quad \text{for all} \quad i \geq 1 - M.
\end{equation*}
Here $\bx_i(0) \defeq 0 \eqdef \by_i(0)$ for all $i \geq 1$, while $\bx_i(0)$ and $\by_i(0)$ are encoded in $\bmu_0$ for all $i \leq 0$. For the latter tasks, which are present in the system at time zero, $r_i^0$ represents the service rate of task $i$ selected at time zero, whereas for $i \geq 1$, its value is meaningless. Also, let
\begin{equation*}
	\left[
	\begin{matrix}
		\bx_i^0(t) \\
		\by_i^0(t)
	\end{matrix}
	\right]
	=
	\left[
	\begin{matrix}
		\bx_i(0) \\
		\by_i(0)
	\end{matrix}
	\right]
	+
	t\left[
	\begin{matrix}
		r_i^0 \\
		1
	\end{matrix}
	\right]
	\ind{a_i = 0} \quad \text{for all} \quad i \geq 1 - M \quad \text{and} \quad t \geq 0,
\end{equation*}
which can be thought of as the attained service and time spent in the system of task $i$ if its service rate does not change. Similarly, for a task $i$ in the system at time zero, the following value represents the departure time if the service rate does not change:
\begin{equation*}
	d_i^0 \defeq
	\begin{cases}
		\infty & \text{if} \quad a_i > 0, \\
		c_i - \by_i\left(0\right) & \text{if} \quad a_i = 0,\ r_i^0 = 0, \\
		\min\left\{b_i - \bx_i\left(0\right), c_i - \by_i\left(0\right)\right\} & \text{if} \quad a_i = 0,\ r_i^0 = 1.
	\end{cases}
\end{equation*}

Proceeding inductively, if the above elements are defined for $k \geq 0$, then we set
\begin{equation*}
	\tau_{k + 1} \defeq \inf\asetc{a_i, d_j^k}{i \geq 1, a_i > \tau_k, j \geq 1 - M, d_j^k > \tau_k}.
\end{equation*}
If we refer to the arrivals and departures as the events of the system, then $\tau_{k + 1}$ represents the time of the $k$th event of the system. Observe that $\tau_{k + 1} > \tau_k$ provided that $d_j^k = \infty$ for all but finitely many indices $j$. We further define
\begin{equation*}
	r_i^{k + 1} \defeq r\left(\bmu_{\tau_{k + 1}}^k, \by_i^k\left(\tau_{k + 1}\right)\right),
\end{equation*}
where $\bmu^k$ is defined as in \eqref{eq: state of the system} with $d_i$, $\bx_i$ and $\by_i$ replaced by $d_i^k$, $\bx_i^k$ and $\by_i^k$, respectively; $r_i^{k + 1}$ can be interpreted as the service rate of task $i$ selected at time $\tau_{k + 1}$. Also, let
\begin{equation*}
	d_i^{k + 1} \defeq
	\begin{cases}
		\infty & \text{if} \quad a_i > \tau_{k + 1}, \\
		d_i^k & \text{if} \quad d_i^k \leq \tau_{k + 1}, \\
		\tau_{k + 1} + c_i - \by_i^k\left(\tau_{k + 1}\right) & \text{if} \quad d_i^k > \tau_{k + 1},\ r_i^{k + 1} = 0, \\
		\tau_{k + 1} + \min\left\{b_i - \bx_i^k\left(\tau_{k + 1}\right), c_i - \by_i^k\left(\tau_{k + 1}\right)\right\} & \text{if} \quad d_i^k > \tau_{k + 1},\ r_i^{k + 1} = 1,
	\end{cases}
\end{equation*}
for all $i \geq 1 - M$. The case $d_i^{k + 1} = \infty$ corresponds to the situation where task $i$ has not arrived to the system yet, whereas $d_i^{k + 1} = d_i^k$ indicates that task $i$ has left by time $\tau_{k + 1}$. In the other cases, $d_i^{k + 1}$ would be the departure time of task $i$ if no arrivals occurred after $\tau_{k + 1}$ and the service rate of task $i$ remained fixed after $\tau_{k + 1}$. Finally, we let
\begin{equation*}
	\left[
	\begin{matrix}
		\bx_i^{k + 1}(t) \\
		\by_i^{k + 1}(t)
	\end{matrix}
	\right]
	=
	\left[
	\begin{matrix}
		\bx_i^k(t) \\
		\by_i^k(t)
	\end{matrix}
	\right]
	\ind{t < \tau_{k + 1}}
	+ \left(t - \tau_{k + 1}\right)
	\left[
	\begin{matrix}
		r_i^{k + 1} \\
		1
	\end{matrix}
	\right] \ind{t \geq \tau_{k + 1}, a_i \leq \tau_{k + 1} < d_i^k}.
\end{equation*}
for $i \geq 1 - M$ and $t \geq 0$. Then $\tau_k$, $d_i^k$, $\bx_i^k$ and $\by_i^k$ are defined for all $i \geq 1 - M$ and $k \geq 0$.

\subsection{Properties of the construction}
\label{sub: properties of the construction}

By definition of $d_i^k$, we have $d_i^k = \infty$ when $a_i > \tau_k$. Then it follows from Remark \ref{rem: codomains of arrival process, service requirements and patiences} that $d_i^k = \infty$ for all but finitely many tasks $i$, which corresponds to the intuitively obvious fact that the number of tasks in the system remains finite over time. As observed above, this implies, in particular, that the sequence $\setc{\tau_k}{k \geq 1}$ is strictly increasing.

Another elementary property is that $\tau_k \to \infty$ as $k \to \infty$. In order to derive this property formally, suppose that $\setc{\tau_k}{k \geq 1} \subset [0, T]$ and let $i_T \defeq \min\setc{i \geq 1}{a_i > T}$, which is finite since $a_i \to \infty$ as $i \to \infty$ by Remark \ref{rem: codomains of arrival process, service requirements and patiences}. Observe that $d_i^k = \infty$ for all $i \geq i_T$ and $k \geq 0$. Then there exist infinitely many $k$ such that $\tau_k = d_i^{k - 1}$ for some fixed $i$, which is a contradiction. Indeed, let $k_i^0$ be the smallest $k$ with this property. Then the first identity below holds by assumption and the others by construction:
\begin{equation*}
	\tau_k > \tau_{k_i^0} = d_i^{k_i^0 - 1} = d_i^{k_i^0} = d_i^{k_i^0 + 1} = \dots = d_i^{k - 1} \quad \text{for all} \quad k > k_i^0.
\end{equation*}
Intuitively, we have $\tau_k \to \infty$ as $k \to \infty$ since otherwise the system would have finitely many arrivals and infinitely many departures in a finite interval of time.

Next we establish the following two additional properties.
\begin{enumerate}
	\item[(a)] For each $i \geq 1 - M$, there exists $k_i \geq 1$ such that $d_i^k = d_i^{k_i}$ for all $k \geq k_i$.
	
	\item[(b)] If $k \geq k_i$, then $\bx_i^k(t) = \bx_i^{k_i}(t)$ and $\by_i^k(t) = \by_i^{k_i}(t)$ for all $t \geq 0$.
\end{enumerate}

To prove (a), fix any $i$, let $k_0 \defeq \min\setc{k \geq 1}{\tau_k \geq a_i}$ and observe that $\by_i^k(t) = t - \tau_{k_0}$ if $k \geq k_0$ and $t \geq \tau_{k_0}$. If $d_i^k > \tau_{k + 1}$ for all $k \geq 0$, then $d_i^k \to \infty$ as $k \to \infty$, but
\begin{equation*}
	d_i^{k + 1} \leq \tau_{k + 1} + c_i - \by_i^k\left(\tau_{k + 1}\right) = \tau_{k + 1} + c_i - \left(\tau_{k + 1} - \tau_{k_0}\right) \leq c_i + \tau_{k_0} \quad \text{for all} \quad k \geq 0,
\end{equation*}
a contradiction. Then there exists $k$ such that $d_i^k \leq \tau_{k + 1}$, and it follows that
\begin{equation*}
	k_i \defeq \min\setc{k \geq 0}{d_i^k \leq \tau_{k + 1}}
\end{equation*}
exists and satisfies $d_i^{k + 1} = d_i^k$ for all $k \geq k_i$ by definition of $d_i^{k + 1}$. This proves (a).

By definition, it is clear that:
\begin{equation}
	\label{eq: consistency of constructed attained service and sojourn times}
	\bx_i^l(t) = \bx_i^k(t) \quad \text{and} \quad \by_i^l(t) = \by_i^k(t) \quad \text{for all} \quad l \geq k, \quad i \geq 1 - M \quad \text{and} \quad t < \tau_{k + 1}.
\end{equation}
Furthermore, $\bx_i^{k + 1}(t) = \bx_i^k(\tau_{k + 1})$ and $\by_i^{k + 1}(t) = \by_i^k(\tau_{k + 1})$ if $t \geq \tau_{k + 1}$ and $\tau_{k + 1} \geq d_i^k$. Since the latter condition holds for $k \geq k_i$, we obtain:
\begin{equation*}
	\bx_i^l(t) = \bx_i^{k_i}\left(\tau_{k_i + 1}\right) \quad \text{and} \quad \by_i^l(t) = \by_i^{k_i}\left(\tau_{k_i + 1}\right) \quad \text{for all} \quad l \geq k_i, \quad i \geq 1 \quad \text{and} \quad t \geq \tau_{k_i + 1}.
\end{equation*}
Taking $k = k_i$ in \eqref{eq: consistency of constructed attained service and sojourn times}, we conclude that (b) holds.

\subsection{Sample paths of the system}
\label{sub: sample paths of the system}

Finally, let us define:
\begin{equation*}
	d_i \defeq d_i^{k_i}, \quad \bx_i(t) \defeq \bx_i^{k_i}(t), \quad \text{and} \quad \by_i(t) \defeq \by_i^{k_i}(t) \quad \text{for all} \quad i \geq 1 - M \quad \text{and} \quad t \geq 0.
\end{equation*}
It is easy to check that $\by_i(t) = \by_i(0) + d_i \wedge t - a_i \wedge t$, i.e., the time spent in the system by task $i$ increases at unit rate between the arrival and departure times. This implies that
\begin{equation*}
	r_i^{k + 1} = r\left(\bmu_{\tau_{k + 1}}^k, \by_i^k\left(\tau_{k + 1}\right)\right) = r\left(\bmu_{\tau_{k + 1}}, \by_i\left(\tau_{k + 1}\right)\right) = r\left(\bmu_t, \by_i(t)\right)
\end{equation*}
for all $i \geq 1 - M$ and $\tau_{k + 1} \leq t < \tau_{k + 2}$. The second equality follows from \eqref{eq: consistency of constructed attained service and sojourn times} if $k < k_i$ and from property (b) of Section \ref{sub: properties of the construction} if $k \geq k_i$, while the third equality follows from the definition of $r$ and the fact that no arrivals occur in $(\tau_{k + 1}, \tau_{k + 2})$. Then we obtain:
\begin{equation*}
	\bx_i(t) = \bx_i(0) + \int_{a_i \wedge t}^{d_i \wedge t} r\left(\bmu_s, \by_i(s)\right)ds \quad \text{for all} \quad i \geq 1 - M \quad \text{and} \quad t \geq 0.
\end{equation*}

It only remains to check that $d_i = \inf\setc{t \geq a_i}{\bx_i(t) \geq b_i\ \text{or}\ \by_i(t) \geq c_i}$. For this purpose, note that $\bx_i(\tau_{k + 1}) < b_i$ and $\by_i(\tau_{k + 1}) < c_i$ for all $k < k_i$ since otherwise we would have $d_i^{k + 1} \leq \tau_{k + 1}$ for some $k < k_i$, contradicting the definition of $k_i$. It follows that $\bx_i(t) < b_i$ and $\by_i(t) < c_i$ if $t \leq \tau_{k_i}$, because $\bx_i$ and $\by_i$ are nondecreasing. Further,
\begin{equation*}
	\tau_{k_i + 1} = d_i^{k_i} = d_i =
	\begin{cases}
		\tau_{k_i} + c_i - \by_i\left(\tau_{k_i}\right) & \text{if} \quad r_i^{k_i} = 0, \\
		\tau_{k_i} + \min\left\{b_i - \bx_i\left(\tau_{k_i}\right), c_i - \by_i\left(\tau_{k_i}\right)\right\} & \text{if} \quad r_i^{k_i} = 1,
	\end{cases}
\end{equation*}
where the first equality follows from $d_i^{k_i} \leq \tau_{k_i + 1}$ and the definition of $\tau_{k_i + 1}$. If $r_i^{k_i} = 0$, then $\by_i(d_i) = \by_i(\tau_{k_i + 1}) = c_i$ because $\by_i$ increases at unit rate between $\tau_{k_i}$ and $\tau_{k_i + 1}$. Moreover, both $\bx_i$ and $\by_i$ increase at unit rate between $\tau_{k_i}$ and $\tau_{k_i + 1}$ if $r_i^{k_i} = 1$, which implies that $\bx_i(d_i) = \bx_i(\tau_{k_i + 1}) = b_i$ or $\by_i(d_i) = \by_i(\tau_{k_i + 1}) = c_i$. Therefore, we indeed have
\begin{equation*}
	d_i = \inf\setc{t \geq a_i}{\bx_i(t) \geq b_i\ \text{or}\ \by_i(t) \geq c_i}.
\end{equation*}

\section{Measurability properties}
\label{app: measurability properties}

In this section $\map{\bmu}{\R_+}{\calM_F^+(\R_+^2)}$ is Borel measurable and we prove that
\begin{equation}
	\label{eq: iterated integral}
	\int_0^t \angled{\varphi_s}{\bmu_s}ds
\end{equation}
is well-defined when $\map{\varphi}{\R_+^3}{\R}$ is Borel and the functions $\varphi_s$ are integrable with respect to $\bmu_s$. Specifically, we show that the inner integral is a Borel function of $s$. We further show that the mapping $(s, y) \mapsto \bar r(\bmu_s, y)$ is a Borel function. Hence, the iterated integrals on the right-hand side of the transport equation \eqref{eq: transport equation} are well-defined.

\begin{lemma}
	\label{lem: measurability of inner integral}
	If $\map{\bmu}{\R_+}{\calM_F^+(\R_+^2)}$ and $\map{\varphi}{\R_+^3}{\R}$ are Borel measurable and such that $\varphi_t$ is integrable with respect to $\bmu_t$ for all $t \geq 0$, then $t \mapsto \langle\varphi_t, \bmu_t\rangle$ is a Borel function.
\end{lemma}

\begin{proof}
	If $A \subset \R_+^2$ is a closed set, then the mapping $\nu \mapsto \nu(A)$ defined on $\calM_F^+(\R_+^2)$ is upper semi-continuous by the Portmanteau theorem for finite measures. In particular, this mapping is Borel measurable, and it follows that $t \mapsto \bmu_t(A)$ is Borel measurable.
	
	Let $A \subset \R_+$ and $B \subset \R_+^2$ be closed sets. If $\varphi$ is the indicator function of $A \times B$, then
	\begin{equation*}
		t \mapsto \angled{\varphi_t}{\bmu_t} = \ind{t \in A} \bmu_t(B)
	\end{equation*}
	is a Borel function. Let $\calH$ be the space of all Borel functions $\map{\varphi}{\R_+^3}{\R}$ such that the mapping $t \mapsto \angled{\varphi_t}{\bmu_t}$ is Borel measurable, which is clearly a vector space. If a sequence of functions $\varphi^n \in \calH$ are nonnegative and increase pointwise to a bounded function $\varphi$, then
	\begin{equation*}
		\angled{\varphi_t}{\bmu_t} = \lim_{n \to \infty} \angled{\varphi_t^n}{\bmu_t} \quad \text{for all} \quad t \geq 0
	\end{equation*}  
	by the monotone convergence theorem, and thus $\varphi \in \calH$.
	
	The collection of sets $A \times B$ with $A \subset \R_+$ and $B \subset \R_+^2$ closed sets is a $\pi$-system that generates the Borel $\sigma$-algebra of $\R_+^3$. Thus, the monotone class theorem for functions implies that the space $\calH$ contains all bounded Borel functions $\map{\varphi}{\R_+^3}{\R}$. 
	
	If we let $\map{\varphi}{\R_+^3}{\R_+}$ be a nonnegative and possibly unbounded Borel function, then the sequence of functions $\varphi_t^n(x, y) = \varphi_t(x, y) \wedge n$ are in $\calH$ and increase pointwise to $\varphi$ in the limit as $n \to \infty$. By the monotone convergence theorem,
	\begin{equation*}
		\angled{\varphi_t}{\bmu_t} = \lim_{n \to \infty} \angled{\varphi_t^n}{\bmu_t} \quad \text{for all} \quad t \geq 0,
	\end{equation*}
	which implies that $t \mapsto \langle\varphi_t, \bmu_t\rangle$ is Borel measurable. Further, if instead of assuming that $\varphi$ is nonnegative, we assume that $\varphi_t$ is integrable with respect to $\bmu_t$ for all $t \geq 0$, then
	\begin{equation*}
		\angled{\varphi_t}{\bmu_t} = \angled{\varphi_t^+}{\bmu_t} - \angled{\varphi_t^-}{\bmu_t} \quad \text{for all} \quad t \geq 0,
	\end{equation*}
	where $\varphi_t^+(x, y) \defeq [\varphi_t(x, y)]^+$ and $\varphi_t^-(x, y) \defeq [-\varphi_t(x, y)]^+$ are nonnegative. Both terms on the right-hand side are Borel and finite for all $t \geq 0$, so the left-hand side is Borel.
\end{proof}

Next we establish that $(t, y) \mapsto \bar r(\bmu_t, y)$ is Borel measurable.

\begin{lemma}
	\label{lem: measurability of rate function}
	If $\map{\bmu}{\R_+}{\calM_F^+(\R_+^2)}$ is Borel, then $(t, y) \mapsto \bar r (\bmu_t, y)$ is Borel.
\end{lemma}

\begin{proof}
	If we fix some $y \in \R_+$, then the set $\R_+ \times [0, y]$ is closed. Therefore, as in the proof of Lemma \ref{lem: measurability of inner integral}, we conclude that the cumulative distribution function $\bF_t(y) \defeq \bmu_t(\R_+ \times [0, y])$ is Borel measurable with respect $t$. It is also right-continuous and nondecreasing in $y$.
	
	Next we show that $t \mapsto \bar y_*(\bmu_t)$ is Borel measurable, by proving that
	\begin{equation*}
		\asetc{t \geq 0}{\bar y_*\left(\bmu_t\right) < y_0} = \bigcup_{q \in \Q \cap [0, y_0)} \asetc{t \geq 0}{\bF_t(q) \geq 1} 
	\end{equation*}
	for all $y_0 > 0$; note that the left-hand side is the empty set if $y_0 \leq 0$. For the inclusion of the left-hand side in the right-hand side, $\bar y_*(\bmu_t) < y_0$ implies that there exists $y_1 < y_0$ such that $\bF_t(y_1) \geq 1$. Since $\bF_t(y)$ is nondecreasing in $y$, we get $\bF_t(q) \geq 1$ for all $q \in \Q \cap [y_1, y_0)$. Conversely, the existence of $q < y_0$ such that $\bF_t(q) \geq 1$ implies that $\bar y_*(\bmu_t) \leq q < y_0$.
	
	We conclude that the mapping $(t, y) \mapsto (\bar y_*(\bmu_t), y)$ is a Borel measurable function from $\R_+^2$ to $[0, \infty] \times \R_+$. It follows that the function $(t, y) \mapsto \bar r(\bmu_t, y) = \ind{y \leq \bar y_*(\bmu_t)}$ is Borel.
\end{proof}

The following remark is relevant for Section \ref{sub: prescribed thresholds}.

\begin{remark}
	\label{rem: measurability for signed measures}
	Lemma \ref{lem: measurability of inner integral} also holds for Borel functions $\map{\bmu}{\R_+}{\calM_F(\R_+^2)}$ with values in the space of finite signed measures, and the proof follows from the result for nonnegative measures. Specifically, let $|\bmu_t|$ denote the total variation of $\bmu_t$ and define:
	\begin{equation*}
		\bmu_t^- \defeq \frac{\left|\bmu_t\right| - \bmu_t}{2} \in \calM_F^+\left(\R_+^2\right) \quad \text{and} \quad \bmu_t^+ \defeq \frac{\left|\bmu_t\right| + \bmu_t}{2} \in \calM_F^+\left(\R_+^2\right),
	\end{equation*}
	i.e., the Jordan decomposition. Given a Borel function $\map{\varphi}{\R_+^3}{\R}$, if $\varphi_t$ is integrable with respect to $\bmu_t$, then it is integrable with respect to $\bmu_t^-$ and $\bmu_t^+$ by definition, and
	\begin{equation*}
		\angled{\varphi_t}{\bmu_t} = \angled{\varphi_t}{\bmu_t^+} - \angled{\varphi_t}{\bmu_t^-} \quad \text{for all} \quad t \geq 0.
	\end{equation*}
	By Lemma \ref{lem: measurability of inner integral}, to establish that $t \mapsto \langle\varphi_t, \bmu_t\rangle$ is Borel measurable, it suffices to prove that $t \mapsto \bmu_t^-$ and $t \mapsto \bmu_t^+$ are Borel measurable functions from $\R_+$ to $\calM_F^+(\R_+^2)$, and for this it is enough to show that $\map{|\scdot|}{\calM_F(\R_+^2)}{\calM_F^+(\R_+^2)}$ is Borel measurable.
	
	In order to prove this, note that, in the weak topology, the Borel $\sigma$-algebra of $\calM_F^+(\R_+^2)$ is generated by the evaluation maps $\mu \mapsto \langle\phi, \mu\rangle$ with $\phi \in C_b(\R_+^2)$. Now fix any function $\phi \in C_b(\R_+^2)$ and let $\psi^n \in C_c(\R_+^2)$ be nonnegative functions that converge pointwise to one. By the dominated convergence theorem,
	\begin{equation*}
		\lim_{n \to \infty} \angled{\phi \psi^n}{\mu} = \angled{\phi}{\mu} \quad \text{for all} \quad \mu \in \calM_F^+\left(\R_+^2\right).
	\end{equation*}
	This shows that the evaluation map associated with any function $\phi \in C_b(\R_+^2)$ is the limit of evaluation maps associated with functions in $C_c(\R_+^2)$. Therefore, the Borel $\sigma$-algebra of $\calM_F^+(\R_+^2)$ is also generated by the evaluation maps $\mu \mapsto \langle\phi, \mu\rangle$ with $\phi \in C_c(\R_+^2)$, so we can prove that $\map{|\scdot|}{\calM_F(\R_+^2)}{\calM_F^+(\R_+^2)}$ is Borel by showing that its composition with evaluation maps associated with functions in $C_c(\R_+^2)$ is Borel measurable.
	
	If $\mu \in \calM_F(\R_+^2)$ and $\phi \in C_c(\R_+^2)$ is nonnegative, then
	\begin{equation*}
		\angled{\phi}{\left|\mu\right|} = \sup\asetc{\angled{\psi}{\mu}}{\psi \in C_c\left(\R_+^2\right)\ \text{and}\ |\psi| \leq \phi}.
	\end{equation*}
	Because $C_c(\R_+^2)$ is separable, the above supremum can be taken over a fixed countable set of compactly supported functions. Since $\mu \mapsto \langle\psi, \mu\rangle$ is measurable from $\calM_F(\R_+^2)$ to $\R$, it follows that $\mu \mapsto \langle\phi, |\mu|\rangle$ is measurable as well. More generally,
	\begin{equation*}
		\mu \mapsto \angled{\phi}{\left|\mu\right|} = \angled{\phi^+}{\left|\mu\right|} - \angled{\phi^-}{\left|\mu\right|} \quad \text{with} \quad \phi^- \defeq \max\left\{-\phi, 0\right\} \quad \text{and} \quad \phi^+ \defeq \max\left\{\phi, 0\right\}
	\end{equation*}
	is measurable from $\calM_F(\R_+^2)$ to $\R$ for all $\phi \in C_c(\R_+^2)$. Thus, $\map{|\scdot|}{\calM_F(\R_+^2)}{\calM_F^+(\R_+^2)}$ is indeed a Borel measurable mapping.
\end{remark}

\section{Proofs of auxiliary results}
\label{app: proofs of auxiliary results}

In this section we provide the proofs of some auxiliary lemmas.

\subsection{Lemmas stated in Section \ref{sec: proof of the fluid limit}}

\lemmafourthree*

\begin{proof}
	Since $S_g$ is continuous and strictly positive in $C \subset U_g$, its minimum over $C$ exists and is strictly positive. Define $\delta > 0$ such that the distance between $C$ and $U_g^c$ equals $2 \delta$ with respect to the maximum norm, and define $K \subset U_g$ as the compact set consisting of the points at distance at most $\delta$ from $C$ in the maximum norm. Suppose that $(x, y) \in C$, $\theta \in [0, \delta]$ and $\alpha \in \{0, 1\}$. By the mean-value theorem:
	\begin{equation*}
		\left|S_g(x, y) - S_g(x + \alpha \theta, y + \theta)\right| \leq \sup_{(u, v) \in K} \left|\alpha \frac{\partial S_g}{\partial x}(u, v) + \frac{\partial S_g}{\partial y}(u, v)\right| \theta,
	\end{equation*}
	Because $K$ is compact and the partial derivatives of $S_g$ are continuous by Assumption \ref{ass: standing assumptions 2}, the above inequality proves the second claim of the lemma.
\end{proof}

\lemmafourfive*

\begin{proof}
	If $(x_0, y_0) \in U_g$, then $S_g(x, y) \geq S_g(x_0, y_0) > 0$ for all $(x, y)$ in the closed rectangle
	\begin{equation*}
		R \defeq \asetc{(x, y) \in \R_+^2}{x \leq x_0\ \text{and}\ y \leq y_0}.
	\end{equation*}
	It follows that $R \subset U_g$ and that $U_g$ is a countable union of closed rectangles, e.g., with vertexes in those points of $U_g$ having rational coordinates. Therefore, we can express $U_g$ as the union of sets $A_k$ such that $A_k \subset A_{k + 1} \subset U_g$ and $A_k$ is a finite union of closed rectangles for each $k \geq 1$. Then we conclude that
	\begin{equation*}
		\lim_{k \to \infty} \int_{A_k} g(x, y)dxdy = \int_{U_g} g(x, y)dxdy = 1,
	\end{equation*}
	which proves the claim for closed rectangles. The case of open rectangles is analogous.
\end{proof}

\lemmafoursixteen*

\begin{proof}
	Recall that $d_{BL}$ denotes the Bounded-Lipschitz metric in $\calM_F^+(\R_+^3)$. Specifically,
	\begin{equation*}
		d_{BL}\left(\xi, \zeta\right) \defeq \sup_{\psi \in \calL} \left|\angled{\psi}{\xi} - \angled{\psi}{\zeta}\right|,
	\end{equation*}
	where $\calL$ is the space of $\map{\psi}{\R_+^3}{[-1, 1]}$ with unit Lipschitz constant. Also, let $\Gamma_T$ be the space of $\map{\bgamma}{[0, T]}{[0, T]}$ that are continuous, increasing and bijective. The metric 
	\begin{equation*}
		d_T^m\left(\bxi, \bzeta\right) \defeq \inf_{\bgamma \in \Gamma_T} \max\left\{\sup_{t \in [0, T]} \left|\bgamma(t) - t\right|, \sup_{t \in [0, T]} d_{BL}\left(\bxi_t, \bzeta_{\bgamma(t)}\right)\right\}
	\end{equation*}
	is compatible with the Skorohod-$J_1$ topology in $D_{\calM_F^+(\R_+^3)}[0, T]$. Similarly, the Skorohod-$J_1$ topology in $D_\R[0, T]$ is induced by the following metric:
	\begin{equation*}
		d_T^f\left(\bx, \by\right) \defeq \inf_{\bgamma \in \Gamma_T} \max\left\{\sup_{t \in [0, T]} \left|\bgamma(t) - t\right|, \sup_{t \in [0, T]} \left|\bx_t - \by_{\bgamma(t)}\right|\right\}
	\end{equation*}
	
	Fix $\varphi \in C_b(\R_+^3)$ with Lipschitz constant $L > 0$. If we define $K \defeq \max\{\norm{\varphi}_\infty, L\}$, then the function $\psi \defeq \varphi / K$ is in $\calL$, and therefore we obtain:
	\begin{align*}
		&\left|\bY_s^N(\varphi) - \bY_t(\varphi)\right| = K \left|\bY_s^N(\psi) - \bY_t(\psi)\right| \leq K d_{BL}\left(\bY_s^N, \bY_t\right) \quad \text{for all} \quad s, t \geq 0,
	\end{align*}
	by linearity of the integrals. Therefore,
	\begin{equation*}
		d_T^f\left(\bY^N(\varphi), \bY(\varphi)\right) \leq K d_T^m\left(\bY^N, \bY\right) \quad \text{for all} \quad T \geq 0.
	\end{equation*}
	Hence, $\bY^N(\varphi) \to \bY(\varphi)$ in $D_\R[0, T]$ as $N \to \infty$ for all $T \geq 0$, and thus also in $D_\R[0, \infty)$.
\end{proof}

\lemmafoureighteen*

\begin{proof}
	Fix $\omega \in \Omega$ such that $\bar\bX^N(\omega) \to \bar\bX(\omega)$ as $N \to \infty$. Since $\bar\calR(t) = \lambda_0 t$,
	\begin{equation}
		\label{eq: functional law of large numbers for reneal process}
		\lim_{N \to \infty} \frac{\calR^N(\omega, t)}{N\lambda_0} = t \quad \text{for all} \quad t \geq 0.
	\end{equation}
	The above properties hold with probability one, so it suffices to prove \eqref{eq: limit of arrival process} for the fixed $\omega$, which is omitted from the notation in the remainder of the proof.
	
	For each $t \geq 0$ and $\varphi \in C_b(\R_+^3)$, we have:
	\begin{equation*}
		\bar\bA_t^N(\varphi) = \lambda_0\int_0^{\frac{\calR^N(t)}{N\lambda_0}} \bar\varphi_s^Nds \quad \text{with} \quad \bar\varphi_t^N \defeq \sum_{i = 1}^\infty \varphi_{a_i^N}\left(0, 0\right) \ind{\frac{i - 1}{N\lambda_0} < t \leq \frac{i}{N\lambda_0}}.
	\end{equation*}
	If $t > 0$ is any continuity point of $\bar\bA$, then $\bar\bA_t^N \to \bar\bA_t$ in $\calM_F^+(\R_+^3)$ as $N \to \infty$. In particular, this implies that $\bar\bA_t^N(\varphi) \to \bar\bA_t(\varphi)$ as $N \to \infty$. We will prove that
	\begin{equation}
		\label{eq: dominated convergence for arrival process}
		\bar\bA_t(\varphi) = \lim_{N \to \infty} \bar\bA_t^N(\varphi) = \lim_{N \to \infty} \lambda_0\int_0^{\frac{\calR^N(t)}{N\lambda_0}} \bar\varphi_s^Nds = \lambda_0\int_0^t \varphi_s(0, 0)ds.
	\end{equation}
	These limits show that \eqref{eq: limit of arrival process} holds if $t$ is a continuity point of $\bar\bA$. However, \eqref{eq: dominated convergence for arrival process} in fact proves \eqref{eq: limit of arrival process} for all $t \geq 0$ since $\bar\bA$ is c\`adl\`ag and the right-hand side is continuous in $t$.
	
	Let $t > 0$ be any continuity point of $\bar\bA$ and fix a function $\varphi \in C_b(\R_+^3)$. For each $s \geq 0$, we consider the sequence of indices $i_s^N \defeq \ceil{N \lambda_0 s}$. Then
	\begin{equation*}
		\lim_{N \to \infty} \frac{i_s^N}{N \lambda_0} = s \quad \text{and} \quad \lim_{N \to \infty} a_{i_s^N}^N = s \quad \text{for all} \quad s \geq 0.
	\end{equation*}
	The first limit is immediate. For the second limit, assume that the limit superior is strictly larger than $s$ and let us show that this leads to a contradiction. Specifically, suppose that there exists some $\varepsilon > 0$ such that we have $a_{i_s^N}^N > s + \varepsilon$ for all sufficiently large $N$. Then $\calR^N(s + \varepsilon) < i_s^N$ for large enough $N$. Dividing by $N \lambda_0$, letting $N \to \infty$ and using \eqref{eq: functional law of large numbers for reneal process}, we obtain $s + \varepsilon < s$, so the limit superior must be smaller than or equal to $s$. A similar argument shows that the limit inferior is larger than or equal to $s$, proving the limit.
	
	The above limits and the continuity of $\varphi$ imply that $\bar\varphi^N$ converges pointwise to the mapping $s \mapsto \varphi_s(0, 0)$ as $N \to \infty$. Thus, the bounded convergence theorem yields
	\begin{equation*}
		\lim_{N \to \infty} \int_0^t \bar\varphi_s^Nds = \int_0^t \varphi_s(0, 0)ds.
	\end{equation*}
	Furthermore, it follows from \eqref{eq: functional law of large numbers for reneal process} that
	\begin{equation*}
		\lim_{N \to \infty} \left|\int_0^{\frac{\calR^N(t)}{N \lambda_0}} \bar\varphi_s^Nds - \int_0^t \bar\varphi_s^Nds\right| \leq \lim_{N \to \infty} \left|\frac{\calR^N(t)}{N \lambda_0} - t\right|\norm{\varphi}_\infty = 0.
	\end{equation*}
	We conclude that \eqref{eq: dominated convergence for arrival process} holds, which completes the proof.
\end{proof}

\lemmafourtwentysix*

\begin{proof}
	Fix $\omega \in \Gamma$, a function $\ell$ as above and $t \in \calC(\omega)$; we will omit $\omega$ from the notation. Observe that $\bar\bmu_t^N \to \bar\bmu_t$ in $\calM_F^+(\R_+^2)$ as $N \to \infty$ since $t \in \calC$. Then \cite[Theorem A.3.12]{dupuis2011weak} and the fact that $\ell$ is lower semicontinuous and nonnegative imply that:
	\begin{equation*}
		\angled{\ell}{\bar\bmu_t} \leq \liminf_{N \to \infty} \angled{\ell}{\bar\bmu_t^N}.
	\end{equation*}
	Note that \cite[Theorem A.3.12]{dupuis2011weak} considers probability measures. However, \emph{Proof 1} of this theorem carries over without any modifications to the case of finite measures. 
	
	For the second inequality, consider the measures $\bar\bxi_t^N, \bar\bxi_t \in \calM_F^+(\R_+^2)$ such that:
	\begin{align*}
		&\int_{\R_+^2} \phi(x, y) \bar\bxi_t^N(dx, dy) \defeq \int_{\R_+^2} \phi(x, y) \bar r\left(\bar\bmu_t^N, y\right) \bar\bmu_t^N(dx, dy), \\
		&\int_{\R_+^2} \phi(x, y) \bar\bxi_t(dx, dy) \defeq \int_{\R_+^2} \phi(x, y) \bar r\left(\bar\bmu_t, y\right) \bar\bmu_t(dx, dy),
	\end{align*}
	for all $\phi \in C_b(\R_+^2)$. It follows from Proposition \ref{prop: limits of integrals involving r} that $\bar\bxi_t^N \to \bar\bxi_t$ in $\calM_F^+(\R_+^2)$ as $N \to \infty$, so \cite[Theorem A.3.12]{dupuis2011weak} can also be invoked to establish the second inequality.
\end{proof}

\subsection{Lemmas stated in Section \ref{sec: analysis of the fluid problem}}

\lemmafivetwo*

\begin{proof}
	Because $C_c^1(\R_+^2)$ is dense in $C_c(\R_+^2)$, there exists $\setc{\phi^n}{n \geq 1} \subset C_c^1(\R_+^2)$ such that $\phi^n \to \phi$ uniformly as $n \to \infty$. Consider now the following compact sets:
	\begin{equation*}
		K_n \defeq \asetc{z \in [0, n]^2}{\mathrm{dist}(z, U_g^c) \geq 1 / n},
	\end{equation*}
	where $\mathrm{dist}(z, U_g^c)$ is the distance from the point $z$ to the closed set $U_g^c$ with respect to some fixed norm. We may choose functions $\alpha^n \in C_c^1(U_g)$ such that:
	\begin{equation*}
		\ind{(x, y) \in K_n} \leq \alpha^n(x, y) \leq \ind{(x, y) \in K_{n + 1}} \quad \text{for all} \quad x, y \in \R_+.
	\end{equation*}
	
	Observe that $\psi^n \defeq \alpha^n\phi^n \in C_c^1(U_g)$ and converges pointwise to $\phi\indc_{U_g}$ as $n \to \infty$. Also, these functions are uniformly bounded since $\phi^n \to \phi$ uniformly. Therefore,
	\begin{equation*}
		\lim_{n \to \infty} \angled{\psi^n}{\bmu_t} = \angled{\phi\indc_{U_g}}{\bmu_t} = \angled{\phi}{\bmu_t}
	\end{equation*}
	by the bounded convergence theorem and property (a) of Definition \ref{def: fluid problem}.
\end{proof}

\lemmafivethree*

\begin{proof}
	Consider a sequence of functions $\alpha^n \in C_c^1(\R_+^3)$ such that
	\begin{equation*}
		\ind{(t, x, y) \in [0, n]^3} \leq \alpha_t^n(x, y) \leq \ind{(t, x, y) \in [0, n + 1]^3} \quad \text{for all} \quad t, x, y \in \R_+ \quad \text{and} \quad n \geq 1.
	\end{equation*}
	Then \eqref{eq: transport equation} holds for $\varphi \alpha^n$ since $\varphi \alpha^n \in C_c^1(\R_+^3)$ for all $n \geq 1$. In addition, we can choose the functions $\alpha^n$ such that their partial derivatives are bounded over $\R_+^3$ and $n \geq 1$.
	
	Observe that $\varphi \alpha^n$ and its partial derivatives converge pointwise to $\varphi$ and its partial derivatives, respectively, as $n \to \infty$. Because $\varphi$ and its partial derivatives are bounded, we conclude from the bounded convergence theorem that \eqref{eq: transport equation} holds for $\varphi$. For the last two terms of \eqref{eq: transport equation} we use the bounded convergence theorem twice, first for the inner integral with respect to $x$ and $y$ and then for the outer integral with respect to $s$.
\end{proof}

\lemmafiveeight*

\begin{proof}
	Consider a sequence of points $(x_n, y_n) \in B$ that converge to some $(x, y) \in \R_+^2$. By definition of $B$, there exist $(\tilde x_n, \tilde y_n) \in A$ such that $x_n \leq \tilde x_n$ and $y_n \leq \tilde y_n$ for all $n \geq 1$. Because the set $A$ is compact, there exists $(\tilde x, \tilde y) \in A$ such that $(\tilde x_n, \tilde y_n) \to (\tilde x, \tilde y)$ along some subsequence, and by construction: $x \leq \tilde x$ and $y \leq \tilde y$. Therefore, we conclude that $(x, y) \in B$ and that $B$ is closed. It is clear that the compact set $A$ is bounded, and this implies that $B$ is bounded and hence compact. Note that $A \subset B$ by definition of $B$. Also, if $(x_0, y_0) \in A \subset U_g$, then $(x, y) \in U_g$ for all $x \leq x_0$ and $y \leq y_0$ and thus $B \subset U_g$.
	
	It is clear that $D$ is compact and $D \subset \R_+ \times U_g$, so it only remains to show that $C \subset D$. For this purpose, note that if $(s, x, y) \in C$, then there exists $(x_0, y_0) \in A$ such that
	\begin{equation*}
		x_0 - x =\int_s^t \alpha_\tau(x, y) d\tau \quad \text{and} \quad y_0 - y = t - s.
	\end{equation*}
	If $s \leq t$, then $x \leq x_0$ and $y \leq y_0$, which implies that $(x, y) \in B$. Otherwise,
	\begin{equation*}
		s \in (t, t + \varepsilon], \quad |x - x_0| \leq s - t \leq \varepsilon \quad \text{and} \quad |y - y_0| = s - t \leq \varepsilon,
	\end{equation*}
	which implies that $(x, y)$ is at distance less than $\varepsilon$ from $A \subset B$. Hence, $(s, x, y) \in D$.
\end{proof}

\lemmafiveten*

\begin{proof}
	Let $d$ be the Euclidean distance from $\supp(\varphi)$ to $\R_+ \times U_g^c$ and define:
	\begin{equation*}
		\tilde K \defeq \asetc{z \in \R^3}{\mathrm{dist}\left(z, \supp(\varphi)\right) \leq \frac{d}{2}} \quad \text{and} \quad K \defeq \tilde K \cap \R_+^3,
	\end{equation*}
	where $\mathrm{dist}(z, \supp(\varphi))$ denotes the Euclidean distance from the point $z$ to the set $\supp(\varphi)$. It is then clear that $\supp(\varphi) \subset K \subset \R_+ \times U_g$ and that $K$ is compact.
	
	Let $\alpha \in C_c^1(\R^3)$ be any function such that
	\begin{equation*}
		\ind{(t, x, y) \in \supp(\varphi)} \leq \alpha(t, x, y) \leq \ind{(t, x, y) \in \tilde K} \quad \text{for all} \quad t, x, y \in \R,
	\end{equation*}
	and let $\tilde \varphi(t, x, y) = \alpha (t, x, y) \varphi(|t|, |x|, |y|)$ for all $t, x, y \in \R$. Then $\tilde \varphi|_{\R_+^3} = \varphi$, the function $\tilde \varphi$ is continuous, its support is contained in $\tilde K$, the lateral partial derivatives of $\tilde \varphi$ exist at each $z \in \R^3$ and there exists $\tilde L \geq 0$ such that we have:
	\begin{equation}
		\label{eq: tilde L bound}
		\max\left\{\left|\partial_i^- \tilde\varphi(z)\right|, \left|\partial_i^+ \tilde\varphi(z)\right|\right\} \leq \tilde L \quad \text{for all} \quad z \in \R^3 \quad \text{and} \quad 1 \leq i \leq 3.
	\end{equation} 
	In the sequel, we only need the bound for the right-partial derivatives. Nonetheless, since $\tilde \varphi(t, x, y)$ depends on $\varphi(|t|, |x|, |y|)$, in order to obtain \eqref{eq: tilde L bound}, we must require in (b) that both of the lateral partial derivatives of $\varphi$ are uniformly bounded.
	
	Consider the standard mollifier $\map{\zeta}{\R^3}{\R}$, defined as:
	\begin{equation*}
		\zeta(z) \defeq \begin{cases}
			0 & \text{if} \quad \norm{z}_2 \geq 1, \\
			c\e^{\frac{1}{\norm{z}_2^2 - 1}} & \text{if} \quad \norm{z}_2 < 1,
		\end{cases}
	\end{equation*}
	where $c > 0$ is a normalizing constant such that the integral of $\zeta$ over $\R^3$ equals one. Let
	\begin{equation*}
		\zeta^n(z) \defeq n^3\zeta(nz), \quad \tilde\varphi^n(z) \defeq \left(\tilde\varphi * \zeta^n\right)(z) = \int_{\R^3} \tilde\varphi(z - u)\zeta^n(u)du \quad \text{and} \quad \varphi^n \defeq \tilde \varphi^n|_{\R_+^3}
	\end{equation*}
	for all $z \in \R^3$ and $n \geq 1$, where $*$ denotes the convolution. Note that the support of $\zeta^n$ is the ball with center the origin and radius $1 / n$, and thus the support of $\varphi^n$ is contained in $K$ for all $n \geq 4 / d$. Introducing a shift in the index of the sequence if needed, we may assume without any loss of generality that $\supp(\varphi^n) \subset K$ for all $n \geq 1$. It follows from \cite[Proposition 8.10]{folland1999real} that the functions $\tilde \varphi^n$ are infinitely differentiable. Furthermore, since the function $\tilde \varphi$ is bounded and uniformly continuous, \cite[Theorem 8.14]{folland1999real} yields:
	\begin{equation*}
		\lim_{n \to \infty} \norm{\varphi - \varphi^n}_\infty \leq \lim_{n \to \infty} \norm{\tilde \varphi - \tilde \varphi^n}_\infty = 0.
	\end{equation*}
	
	Fix $z \in \R^3$, $1 \leq i \leq 3$ and $n \geq 1$. In addition, define $\map{\psi^k}{\R^3}{\R}$ as:
	\begin{equation*}
		\psi^k(u) \defeq k\left[\tilde\varphi\left(z + \frac{e_i}{k} - u\right) - \tilde\varphi\left(z - u\right)\right] \quad \text{for all} \quad u \in \R^3 \quad \text{and} \quad k \geq 1,
	\end{equation*}
	where $e_i$ is the $i$th vector in the canonical basis of $\R^3$. It follows from \eqref{eq: tilde L bound} that $|\psi^k(u)| \leq \tilde L$ for all $u \in \R^3$. Furthermore, $\psi^k(u) \to \partial_i^+ \tilde\varphi(z - u)$ for all $u \in \R^3$. Hence,
	\begin{equation}
		\label{eq: partial derivatives of convolutions}
		\partial_i\tilde\varphi^n(z) = \lim_{k \to \infty} \int_{\R^3} \psi^k(u)\zeta^n(u)du = \int_{\R^3} \partial_i^+\tilde\varphi(z - u)\zeta^n(u)du = \left[\left(\partial_i^+\tilde\varphi\right) * \zeta^n\right](z),
	\end{equation}
	for all $z \in \R^3$, $1 \leq i \leq 3$ and $n \geq 1$, by the dominated convergence theorem. As a result,
	\begin{equation*}
		\lim_{n \to \infty} \partial_i\varphi^n(z) = \lim_{n \to \infty} \partial_i\tilde\varphi^n(z) = \lim_{n \to \infty} \left[\left(\partial_i^+\tilde\varphi\right) * \zeta^n\right](z) = \partial_i^+\tilde\varphi(z) = \partial_i\varphi(z)
	\end{equation*}
	for all $z \in \R_+^3 \setminus E$ and $1 \leq i \leq 3$ by \cite[Theorem 8.15]{folland1999real}.
	
	Finally, using \eqref{eq: tilde L bound}, we obtain:
	\begin{equation*}
		\left|\partial_i \varphi^n(z)\right| = \left|\partial_i \tilde\varphi^n(z)\right| \leq \int_{\R^3} \left|\partial_i^+\tilde\varphi(z - u)\zeta^n(u)\right|du \leq \tilde L \norm{\zeta^n}_1 = \tilde L
	\end{equation*}
	for all $z \in \R_+^3$, $1 \leq i \leq 3$ and $n \geq 1$ by \eqref{eq: partial derivatives of convolutions}.
\end{proof}

\lemmafiveeleven*

\begin{proof}
	The set $E \defeq \setc{(t, x, y) \in \R_+^3}{(t, y) \in F}$ satisfies (a) of Lemma \ref{lem: approximation using mollifiers}. Then there exist $\tilde L \geq 0$, a compact set $\supp(\varphi) \subset K \subset \R_+ \times U_g$ and functions $\varphi^n \in C_c^\infty(\R_+ \times U_g)$ as in the statement of Lemma \ref{lem: approximation using mollifiers}. For each $n \geq 1$, Proposition \ref{prop: spatial rescaling} gives:
	\begin{equation}
		\label{eq: transport equation after change of variables restated}
		\begin{split}
			L_{\bmu_t}\left(\varphi_t^n\right) &= L_{\bmu_0}\left(\varphi_0^n\right) + \lambda_0 \int_a^t \varphi_s^n(0, 0)ds \\
			&+ \int_a^t L_{\bmu_s}\left(\partial_t\varphi_s^n + \bar \br_s \partial_x \varphi_s^n + \partial_y\varphi_s^n\right)ds \quad \text{for all} \quad t \in [a, b].
		\end{split}
	\end{equation}
	
	We will prove that $\varphi$ satisfies the same equation by taking the limit as $n \to \infty$. Let $K_g \defeq \setc{(x, y) \in \R_+^2}{(t, x, y) \in K\ \text{for some}\ t \geq 0} \subset U_g$ be the projection of $K$ into $\R_+^2$, which is clearly a compact set. Then we obtain:
	\begin{align*}
		\lim_{n \to \infty} \left|L_{\bmu_t}\left(\varphi_t\right) - L_{\bmu_t}\left(\varphi_t^n\right)\right| &\leq \lim_{n \to \infty} \int_{K_g} \frac{\left|\varphi_t(x, y) - \varphi_t^n(x, y)\right|}{S_g(x, y)}|\bmu_t|(dx, dy) \\
		&\leq \max_{(x, y) \in K_g} \frac{|\bmu_t|\left(\R_+^2\right)}{S_g(x, y)} \lim_{n \to \infty} \norm{\varphi - \varphi^n}_\infty = 0 \quad \text{for all} \quad t \geq 0.
	\end{align*}
	This gives the limits of the left-hand side and first term on the right-hand side of \eqref{eq: transport equation after change of variables restated}. For the second term on the right-hand side, observe that:
	\begin{equation*}
		\lim_{n \to \infty} \left|\int_a^t \varphi_s(0 , 0)ds - \int_a^t \varphi_s^n(0, 0)ds\right| \leq (t - a) \lim_{n \to \infty} \norm{\varphi - \varphi^n}_\infty = 0. 
	\end{equation*}
	
	It remains to consider the third term on the right-hand side of \eqref{eq: transport equation after change of variables restated}. First, note that
	\begin{equation}
		\label{eq: e has nu measure zero}
		\begin{split}
			\int_a^t \int_{K_g} &\frac{\ind{(s, y) \in F}}{S_g(x, y)} |\bmu_s|(dx, dy)ds \leq \\
			&\max_{(x, y) \in K_g} \frac{1}{S_g(x, y)} \int_a^t \int_{\R_+^2} \ind{(s, y) \in F} |\bmu_s|(dx, dy)ds = 0.
		\end{split}
	\end{equation}
	Indeed, the $y$-marginal of $|\bmu_s|$ is absolutely continuous with respect to the Lebesgue measure for each $s \geq 0$ by property (b) of Definition \ref{def: fluid problem}, and hence
	\begin{equation*}
		\int_a^t \int_{\R_+^2} \ind{(s, y) \in F} |\bmu_s|(dx, dy)ds = \int_a^t \int_{y \in \R_+} \ind{(s, y) \in F} \int_{x \in \R_+} |\bmu_s|(dx, dy)ds = 0.
	\end{equation*}
	
	For each $s \in [a, t]$, let $\bmu_s = \bmu_s^+ - \bmu_s^-$ be the Jordan decomposition of $\bmu_s$, and define $\bnu_t^+$ and $\bnu_t^-$ as the nonnegative measures determined by the following integrals:
	\begin{align*}
		\angled{\psi}{\bnu_t^+} \defeq \int_a^t \int_{K_g} \frac{\psi_s(x, y)}{S_g(x, y)}\bmu_t^+(dx, dy)ds \quad \text{and} \quad \angled{\psi}{\bnu_t^-} \defeq \int_a^t \int_{K_g} \frac{\psi_s(x, y)}{S_g(x, y)}\bmu_t^-(dx, dy)ds
	\end{align*}
	for all $\psi \in C_b(\R_+^3)$. It follows from \eqref{eq: integrable total variation} that these measures are finite, because
	\begin{equation*}
		\max\left\{\bnu_t^+\left(\R_+^2\right), \bnu_t^-\left(\R_+^2\right)\right\} \leq \int_a^b \max_{(x, y) \in K_g} \frac{\left|\bmu_s\right|\left(\R_+^2\right)}{S_g(x, y)} ds < \infty.
	\end{equation*}
	Also, we have shown in \eqref{eq: e has nu measure zero} that $E$ has measure zero with respect to $\bnu_t^+$ and $\bnu_t^-$. With this new notation, we can express the third term on the right-hand side of \eqref{eq: transport equation after change of variables restated} as follows:
	\begin{equation*}
		\angled{\partial_t\varphi_s^n + \bar \br_s \partial_x \varphi_s^n + \partial_y\varphi_s^n}{\bnu_t^+} - \angled{\partial_t\varphi_s^n + \bar \br_s \partial_x \varphi_s^n + \partial_y\varphi_s^n}{\bnu_t^-}.
	\end{equation*}
	
	By construction of the functions $\varphi^n$, the following limit holds for all $(s, x, y) \notin E$, and therefore almost everywhere with respect to $\bnu_t^+$ and $\bnu_t^-$:
	\begin{align*}
		\lim_{n \to \infty} \left[\partial_t \varphi_s^n(x, y) + \ind{y \leq \bar \by_s} \partial_x \varphi_s^n(x, y) + \partial_y \varphi_s^n (x, y)\right] \\
		= \partial_t \varphi_s(x, y) + \ind{y \leq \bar \by_s} \partial_x \varphi_s(x, y) + \partial_y \varphi_s(x, y)
	\end{align*}
	Moreover, for all $s, x, y \in \R_+$, we have:
	\begin{equation*}
		\left|\partial_t \varphi_s^n(x, y) + \ind{y \leq \bar \by_s} \partial_x \varphi_s^n(x, y) + \partial_y \varphi_s^n(x, y)\right| \leq 3\tilde L.
	\end{equation*}
	Therefore, the bounded convergence theorem, applied to $\bnu_t^+$ and $\bnu_t^-$, implies that for each $t \in [a, b]$ the third term on the right-hand side of \eqref{eq: transport equation after change of variables restated} converges as $n \to \infty$ to
	\begin{align*}
		\angled{\partial_t\varphi_s + \bar \br_s \partial_x \varphi_s + \partial_y\varphi_s}{\bnu_t^+} - \angled{\partial_t\varphi_s + \bar \br_s \partial_x \varphi_s + \partial_y\varphi_s}{\bnu_t^-} \\
		= \int_a^t L_{\bmu_s}\left(\partial_t\varphi_s + \bar \br_s \partial_x \varphi_s + \partial_y\varphi_s\right)ds,
	\end{align*}
	which completes the proof.
\end{proof}

\end{appendices}
	
\newcommand{\noop}[1]{}
\bibliographystyle{IEEEtranS}
\bibliography{bibliography}
	
\end{document}